\documentclass[12pt]{article}
\usepackage[margin=1in, a4paper]{geometry}

\usepackage{amsmath,enumerate,amsfonts,color,amssymb, amsthm,soul}

\usepackage{hyperref}

\usepackage[normalem]{ulem}

\usepackage{tikz}

\DeclareFontFamily{OT1}{rsfs}{}
\DeclareFontShape{OT1}{rsfs}{m}{n}{ <-7> rsfs5 <7-10> rsfs7 <10-> rsfs10}{}
\DeclareMathAlphabet{\mycal}{OT1}{rsfs}{m}{n}

\def\NN{{\mathbb N}}
\def\ZZ{{\mathbb Z}}
\def\RR{{\mathbb R}}
\def\CC{{\mathbb C}}

\def\mcA{{\mycal A}}
\def\mcB{{\mycal B}}
\def\mcF{{\mycal F}}

\newtheorem{theorem}{Theorem}[section]
\newtheorem{corollary} [theorem] {Corollary}
\newtheorem{lemma} [theorem] {Lemma}

\newtheorem{remark}[theorem]{Remark}

\numberwithin{equation}{section}

\def\vareps{{\varepsilon}}

\newcounter{marnote}

\begin{document}

\title{Mountain pass for the Ginzburg-Landau energy in a strip: solitons and solitonic vortices}

\author{ Amandine Aftalion~\thanks{Universit\'e Paris-Saclay, CNRS,  Laboratoire de math\'ematiques d'Orsay, 91405, Orsay, France. Email: amandine.aftalion@math.cnrs.fr.}~ and Luc Nguyen~\thanks{Mathematical Institute and St Edmund Hall, University of Oxford, Andrew Wiles Building, Radcliffe Observatory Quarter, Woodstock Road, Oxford OX2 6GG, UK. Email: luc.nguyen@maths.ox.ac.uk}}

\date{\today}

\maketitle

\begin{abstract} 
Motivated by recent experiments, we study critical points of the Ginzburg-Landau energy in an infinite strip where phase imprinting is applied to half of the domain. We prove that there is a critical width of the cross section below which the soliton solution is a  mountain pass solution and the minimizer within the subspace of odd functions. Above the critical width, we find that the mountain pass solution is a vortex with a solitonic behaviour in the infinite direction, called a solitonic vortex. Moreover, depending on the width, we prove that the minimizer in a space with some symmetries can display one or several solitonic vortices. While the problem shares some similarities with the analysis of stability and minimality of the Ginzburg-Landau vortex of degree one in a disk or the whole plane, the change in geometry introduces subtle analytical differences. Extensions to the case of an infinite cylinder in 3D are also given.
\end{abstract}



\section{Introduction}

We study the Ginzburg-Landau energy in an infinite strip or cylinder with what is known experimentally as the phase imprinting condition. More precisely, for $d > 0$, we let $\Omega_d = \{(x,x') \in \RR^N: |x'| < d\}$ with $N = 2$ or $N = 3$, and consider the Ginzburg-Landau functional for maps $u: \Omega_d \rightarrow \RR^2 \cong \CC$:
\[
\mcF_d[u] = \int_{\Omega_d} \Big( \frac{1}{2}|\nabla u|^2 + \frac{1}{4} (1 - |u|^2)^2\Big)\,dx\,dx'.
\]
The experimental procedure of phase imprinting, that we will describe below, stipulates that  one half of the condensate has a conjugate phase with respect to the other half, which can be translated as $\bar u(-x,x')=u(x,x')$ or equivalently, if $u= (u_1,u_2)$,
\begin{equation}
u_1(-x,x') = u_1(x,x'), \quad u_2(-x,x') = - u_2(x,x').
	\label{Eq:Imprinting}
\end{equation}
We are interested in critical points of $\mcF_d$ which satisfy \eqref{Eq:Imprinting} together with the boundary condition
\begin{equation}
\begin{cases}
\partial_\nu u = 0 \text{ on } \partial \Omega_d,\\
u(x,\cdot) \rightarrow (0,\pm 1) \text{ as } x \rightarrow \pm \infty.
\end{cases}
\label{BC}
\end{equation}
The Euler-Lagrange equation for these critical points is
\begin{equation}
-\Delta u = (1 - |u|^2)u \text{ in } \Omega_d.
	\label{EL}
\end{equation} 
An important and explicit critical point of $\mcF_d$ satisfying \eqref{Eq:Imprinting}, known as the soliton solution, is given by
\begin{equation}\label{usdef}
u_s(x,x') =  \tanh\frac{x}{\sqrt{2}}\, e_2, \quad e_2 = (0,1).
\end{equation} Soliton stability or instability is a phenomenon that has been extensively studied from both mathematical and physical perspectives \cite{dauxois,manton,newell,rousset09,tao}. The mathematical results of this paper are motivated by recent experiments that we will now describe.
\subsection{Motivation} A soliton is a one-dimensional solitary wave propagating at a constant speed. The simplest situation that we will consider here is actually the case of zero speed. Solitons have been observed in a wide variety of nonlinear media such as in nonlinear optics, in crystals, and more recently in ultracold atomic vapors. Recent experiments for  both fermions and bosons with repulsive interaction, has renewed the interest in solitons \cite{dalibard2025cours}. In ultracold quantum gases, solitons offer a unique window into the interplay between interactions and coherence, enabling controlled studies of non-equilibrium dynamics, integrability, and quantum transport. Ultracold atoms can be controlled at will and therefore allow the study of collective behavior. They are described by a wave function whose phase plays the role of orientation in magnetization. One characteristic is the observation of defects. The simplest defects in cold atoms are dark solitons, which are one-dimensional solitary waves; they correspond to an envelope with a density dip with a phase shift of $\pi$ across the density minimum. 
 Other types of defects are vortices in which the order parameter rotates around a density dip  or a  vortex ring or closed loop where the density vanishes. Matter waves and dark solitons can be created in experiments using various methods and in particular phase imprinting and density engineering.
 The first gray solitons in Bose-Einstein condensates were experimentally realized by \cite{bbde1} and \cite{densch}. In both experiments, the soliton is produced by phase imprinting. Half of the condensate is illuminated by a non-resonant laser beam for a short period of time. The illuminated part acquires a phase proportional to the light intensity, which is adjusted to be $\pi$. At the boundary between the illuminated and dark areas, the strong phase gradient sets in motion a density wave corresponding to a soliton.

Several experimental groups have attempted to study solitons by imposing  phase imprinting in an elongated condensate for bosonic atoms, rubidium \cite{PRLsoli} and sodium \cite{kibble,PRLsol2} and  for fermionic atoms (lithium) \cite{heavysol,PRLrings,PRLsol}. At first, they thought they had observed solitons \cite{kibble,heavysol}. Further studies were needed to fully understand the phenomenon \cite{chevysol}: in the case of lithium, they thought it was not a soliton but a vortex ring \cite{PRLrings}, until \cite{PRLsol,LevinPRL14} argued that in fact it was a single straight vortex called solitonic vortex; in the case of sodium, it was also confirmed that it was not a soliton but a solitonic vortex \cite{PRLsol2}.
 A solitonic vortex is a vortex in a band whose transverse size is much smaller than its length and which has the same asymptotic phase profile as a soliton. In fact, the solitonic vortex is exponentially localized in the longitudinal direction on the scale of the transverse dimension. This is a remarkable property since the energy density of isolated vortices or vortex-antivortex pairs decreases algebraically, not exponentially, in the absence of boundaries \cite{BBH,betsaut2,bgs,betsaut,JR}. Mathematically, the exponential localization and phase difference between the two ends of the band result from the infinite and periodic structure of the image vortices \cite{AS-23}.

In experiments, the emergence of  the solitonic vortex may be due to the destabilization of a soliton after what is known as a snake instability. It has been observed that the soliton is only stable with a sufficiently small number of particles or in a narrow band. The mode associated with snake instability is the solitonic vortex. The snake instability metaphor has been used to describe the curvature of the soliton wavefront and the decomposition of the soliton into vortex rings. In a series of studies, Brand and Reinhardt \cite{brand}, Mu\~noz Mateo and Brand \cite{Brand-PRL14, mateo}, Komineas \cite{Komi}, Komineas et al. \cite{kopapaL,kopapawaves, kopapasol} analyze how the soliton destabilizes and bifurcates depending on the width of the band both in 2D and 3D.
In some intermediate cases, the soliton initially deforms into a vortex-antivortex pair or a 3D vortex ring. In \cite{Brand-PRL14}, they also find numerically the  $\Phi$ structure that was exhibited in \cite{zwierlein}, which appears after the vortex ring. These structures eventually decompose into a stable solitonic vortex. Numerical calculations reveal that the vortex-antivortex pair or vortex ring is unstable, but its lifetime is long enough to be observed in both simulations and experiments. 

The study of the bifurcation of the soliton in an infinite strip has been made in \cite{AGS} and \cite{AS-23}. However, some questions motivated by the experiments remain open. In particular, we would like to tackle here the issue of connecting the behaviours for small and large $d$ regimes, and to better understand the function spaces in which stability or instability occurs. Our results do not concern the time dependent problem, where much remains to be explored. For example, it would be interesting to investigate the orbital stability of the soliton in the small width regime, in the spirit of the analysis of \cite{GS15}, as well as the orbital stability of the solitonic vortex in the large width regime. The latter question would require further analysis of the solitonic vortex than what is undertaken here, e.g. in the spirit of \cite{ collot25, GPS22}.

  In this paper, motivated by these recent experiments on fermions, and experiments in progress on bosons after the thesis of Franco Rabec \cite{rabec}, we analyze the geometry of the infinite strip and cylinder, with an antisymmetric condition in the infinite direction, modelling phase imprinting. This brings new interesting mathematical questions since vortices no longer have an algebraic decay at infinity, they do not behave like $e^{i\theta}$, but like the hyperbolic tangent in the infinite direction, that is a soliton solution. The issue is to analyze the minimization or extremization of the Ginzburg-Landau energy in a suitable space and describe solitons and solitonic vortices.

 \subsection{Main results}
 
Recall that the soliton solution is called $u_s$ and given by \eqref{usdef}. A natural space on which one considers the functional $\mcF_d$ is the affine space $\mcA_d^{\rm x}$ of maps in $u_s + H^1(\Omega_d,\RR^2)$ satisfying the symmetry conditions \eqref{Eq:Imprinting}. However, there $\inf_{\mcA_d^{\rm x}} \mcF_d$ is zero and the infimum is not achieved. For example, a minimizing sequence is given by e.g. $u^{(n)}(x,y) = (\cos \chi(x/n), \sin \chi(x/n))$ for any smooth odd function $\chi \in C^\infty(\RR)$ satisfying $\chi = \frac{\pi}{2}$ in $(1,\infty)$. As is known, part of the issues relates to the fact that $\mcA_d^{\rm x}$ contains maps which do not vanishes anywhere (e.g. the maps $u^{(n)}$). As noted in Remark 2 of the paper of Mironescu \cite{mironescu} on the stability of the degree one vortex in the whole plane, if one finds a minimization space with a zero somewhere, the mathematical analysis becomes possible. A somewhat artificial way to prescribe the zero at the origin can be made by imposing imparity of solutions as in \cite{AS-23} and in some results below.
  On the other hand, without imposing the zero at the origin, we still need some kind of vanishing and    we will show that the soliton for small width, or the solitonic vortex above a critical width is a moutain pass solution for the Ginzburg-Landau problem.

\subsubsection{The case $N = 2$}

 In Aftalion and Sandier \cite{AS-23}, where only $N = 2$ is considered, the choice made to enforce a zero for $u$ was to suppose an additional symmetry, namely the oddness of $u$:
\begin{equation}
u_1(-x,-y) = - u_1(x,y), \quad u_2(-x,-y) = - u_2(x,y),
	\label{Eq:Imparity}
\end{equation} 
where we have used the variable $y$ instead of $x'$. Note that \eqref{Eq:Imparity} forces $u$ to be zero at the origin. Let $\mcA_d$ be the set of maps in $u_s + H^1(\Omega_d,\RR^2)$ satisfying the symmetry conditions \eqref{Eq:Imprinting} and \eqref{Eq:Imparity}. In \cite{AS-23}, it is showed, among other statements, that:
\begin{itemize}
\item For any $d > 0$, there is a minimizer of $\mcF_d$ in $\mcA_d$. 
\item When $d$ is sufficiently small, $u_s$ is the unique minimizer of $\mcF_d$ in $\mcA_d$.
\item When $d$ is sufficiently large, minimizers of $\mcF_d$ in $\mcA_d$ are different from $u_s$, have a zero with degree $\pm 1$ at the origin, and have energy $\pi \ln d + O(1)$ as $d \rightarrow \infty$, where the constant term is lower than that of the vortex in the whole plane. This minimizer is known as the solitonic vortex solution. The degree $-1$ solitonic vortex solution can be obtained from the degree $1$ solitonic vortex solution by the transformation $(u_1, u_2) \mapsto (-u_1, u_2)$.
\end{itemize}

Our first theorem gives a clean-cut refinement of the above result: for $d \leq \sqrt{2}\pi/2$, the soliton solution is the unique minimizer of $\mcF_d$ in $\mcA_d$, while for $d > \sqrt{2}\pi/2$, there are exactly two solitonic vortex solutions and they are the only minimizers of $\mcF_d$ in $\mcA_d$.

\begin{theorem}
\label{Thm1}
Let $N = 2$. 
\begin{enumerate}[(i)]
\item If $0 < d \leq \sqrt{2}\pi/2$, then the soliton solution $u_s$ is the unique minimizer of $\mcF_d$ in $\mcA_d$.

\item Suppose $d > \sqrt{2}\pi/2$. Then the soliton solution $u_s$ is an unstable critical point of $\mcF_d$ in $\mcA_d$ and $\mcF_d$ has exactly two minimizers $u_d^\pm =  (\pm U_d, V_d)$ in $\mcA_d$. Moreover,
\begin{align}
\begin{cases}
U_d^2 + V_d^2 < 1 \text{ in } \bar\Omega_d,\\
U_d(x,y) < 0 \text{ for } y > 0, U_d(x,y) = 0 \text{ for } y = 0, U_d(x,y) > 0 \text{ for } y < 0,\\
V_d(x,y) > 0 \text{ for } x > 0, V_d(x,y) = 0 \text{ for } x = 0, V_d(x,y) < 0 \text{ for } x < 0.
\end{cases}
\label{Eq:UVSign}
\end{align}
In particular, the origin is the unique zero of $u_d^\pm$ and $u_d^\pm$ has degree $\pm 1$ at the origin.
 
\end{enumerate}
\end{theorem}

\begin{remark}
 Let $d > \sqrt{2}\pi/2$. As $x \rightarrow \pm\infty$, $u_d^\pm - u_s$ decays faster than any polynomial -- see Remark \ref{Rem:PD}. This is very different from the asymptotic behavior of the entire vortex solution of the Ginzurg-Landau functional.
\end{remark}

The proof uses several ideas developed in \cite{ IN-IHP24, INSZ_CRAS18,INSZ-ENS} for the analysis of local stability of the degree one vortex in a disk or a ball in arbitrary dimensions. (See also \cite{INSZ_CRAS, INSZ3} for earlier related work for the Landau-de Gennes functional.)  It relies on the following two ingredients. The first ingredient is the statement that if the soliton is stable in the first direction, then it is the unique minimizer (Theorem \ref{Thm:S=>M}). The second 
ingredient is the statement that, when $u_s$ is unstable, there exists a critical point satisfying the sign conditions on the last two lines of \eqref{Eq:UVSign}, and any such critical point is minimizing (Theorem \ref{Thm:UV}). The critical value $\sqrt{2}\pi/2$ is found as  the value above which there is a non-zero function $\varphi$ such that
$$\int_{\Omega_d} \Big(|\nabla \varphi|^2 - \mathrm{sech}^2\frac{x}{\sqrt{2}} \varphi^2\Big)\,dx\,dy \geq 0$$ (see Corollary \ref{Cor:d1}). 

We note that although the solitonic vortices $u_d^\pm$ are minimizers of $\mcF_d$ in $\mcA_d$, they are unstable critical points of $\mcF_d$ in $\mcA_d^{\rm x}$. See Remark \ref{Rem:uvUnstable}.

Our next result concerns a mountain pass characterization of the soliton solution and the solitonic vortex solutions in $\mcA_d^{\rm x}$. Fix an arbitrary decreasing odd function $\chi \in C^\infty(\RR)$ satisfying $\chi \equiv \frac{\pi}{2}$ in $[1,\infty)$, and define
\begin{equation}
\psi_n^\pm(x,\cdot) = (\pm \cos \chi(x/n), \sin \chi(x/n)).
	\label{Eq:psinDef}
\end{equation}
Then $\psi_n^\pm \in \mcA_d^{\rm x}$ and
\[
\mcF_d[\psi_n^\pm] \rightarrow 0 \text{ as } n \rightarrow \infty.
\]
Note that $\psi_n^\pm \in C^\infty(\bar\Omega_d,\mathbb{S}^1)$ and $\psi_n^+$ and $\psi_n^-$ are not homotopic as $\mathbb{S}^1$-valued maps. We suppose $n$ is sufficiently large such that
\begin{equation}
\mcF_d[\psi_n] < \begin{cases}
\mcF_d[u_s] &\text{ if } 0 < d \leq \sqrt{2}\pi/2,\\
\mcF_d[u_d^\pm] &\text{ if } d >\sqrt{2}\pi/2.
\end{cases}
\label{nCond}
\end{equation}
Define
 \begin{align}
\Gamma_d 
	&= \Big\{h \in C([-1,1],\mcA_d^{\rm x}): h(-1)  = \psi_n^-, h(1) = \psi_n^+\Big\}, \label{Eq:GamDef}\\
c_{d} 
	&= \inf_{h \in \Gamma_d} \max_{t \in [-1,1]} \mcF_d[h(t)].
\label{Eq:MPEnergy}
\end{align}

\begin{theorem} \label{Thm:MPC}
Let $N = 2$, $d > 0$ and let $\psi_n$ be given by \eqref{Eq:psinDef} with $n = n(d)$ sufficiently large such that \eqref{nCond} holds. Then the mountain pass energy $c_d$ defined by \eqref{Eq:MPEnergy} satisfies
\begin{align*}
c_d  &= \mcF_d[u_s]  \text{ if } d \leq \sqrt{2}\pi/2,\\
\frac{1}{16} \leq c_d  &= \mcF_d[u_d^\pm] <  \mcF_d[u_s] \text{ if } d > \sqrt{2}\pi/2.
\end{align*}
\end{theorem}

 The proof of the upper bound for $c_d$ involves an explicit construction of a path in $\Gamma_d$ connecting $u_s$ when $d \leq \sqrt{2}\pi/2$, or $u_d^+$ when $d > \sqrt{2}\pi/2$, to $\psi_n^\pm$ such that the largest energy along that path is more or less the desired value. In both cases, we use the fact that any map with a non-negative first component can be continuously deformed to $\psi_n^+$ with essentially no increase in energy. When $d \leq \sqrt{2}\pi/2$, the soliton solution $u_s$ can be continuously deformed to a map with a non-negative first component by exploiting its instability in $\mcA_d^{\rm x}$. When $d > \sqrt{2}\pi/2$, this can be achieved simply by translating the solitonic vortex along the $y$ direction until the vortex lies on the boundary.

For the lower bound for $c_d$, in addition to the ideas from \cite{ IN-IHP24, INSZ_CRAS18,INSZ-ENS}, we further establish and make use of the following two inequalities:
\begin{itemize}
\item If $d \leq \sqrt{2}\pi/2$ and if $\psi \in H^1(\Omega_d)$ satisfies $\int_{-d}^d \psi(0,y)\,dy = 0$, then
\begin{equation}
\int_{\Omega_d} \big[|\nabla \psi|^2 - (1 - |u_s|^2) \psi^2\big]\,dx\,dy \geq 0;
	\label{Eq:MPI-1}
\end{equation}
\item If $d > \sqrt{2}\pi/2$ and if $\psi \in H^1(\Omega_d)$ satisfies $\int_{\Omega_d} (1 - U_d^2 - V_d^2) \psi\,dx\,dy = 0$, then
\begin{equation}
\int_{\Omega_d} \big[|\nabla \psi|^2 - (1 - U_d^2 - V_d^2) \psi^2\big]\,dx\,dy \geq 0.
	\label{Eq:MPI-2}
\end{equation}
\end{itemize}
The proof of \eqref{Eq:MPI-1} can be done by a routine Fourier decomposition in the $y$-direction. The proof of \eqref{Eq:MPI-2} is more delicate. We first show that it holds for $d$ close to $\sqrt{2}\pi/2$ (Lemma \ref{Lem:ECS-Low}). We then argue by contradiction: If it does not hold for all $d > \sqrt{2}\pi/2$, then we deduce that the second eigenvalue of a certain weighted Neumann eigenvalue problem on $\Omega_d$ has multiplicity three and an associated eigenfunction has a zero set which is a closed simple curve which meet the boundary $\partial \Omega_d$ exactly and non-tangentially at $(0,d)$ (Lemma \ref{Lem:ECS-d*}). We then show that the even extension of this eigenfunction across $y = d$ is the third eigenfunction a different but related weighted Neumann eigenvalue problem on the strip $\{0 < y < 2d\}$ (Lemma \ref{Lem:Sib3rd}). It turns out that, in the new domain, this third eigenfunction has four nodal domains, contradicting Courant's nodal domain theorem. Throughout the argument, the fact that $1 - U_d^2 - V_d^2$ decays faster than the usual Hardy weight provides important compactness.

We note that there are other critical points of $\mcF_d$ in $\mcA_d^{\rm x}$ when $d$ is sufficiently large. In \cite{AGS}, under solely the phase imprinting condition \eqref{Eq:Imprinting}, Aftalion, Gravejat and Sandier showed that, as the width $d$ enlarges and for every $k \in \NN^*$, there exist two branches $u_{d,k}^\pm$ of critical point of $\mcF_d$ in $\mcA_d^{\rm x}$ bifurcating from the soliton solution at the bifurcating points $d_k = \sqrt{2}\pi k/2$. These two branches  can be obtained from one another by a global switch of sign in the first component and have $k$ vortices lying on the $y$-axis with alternate degrees $\pm 1$. Moreover, for $d$ close to $d_k$, $u_{d,k}^\pm$ and $u_s$ are the only critical points of $\mcF_d$ near $u_s$. 
 We will prove below, in Theorem \ref{Thm2}, that $u_{d,k}^\pm$ exist for all $d > \sqrt{2}\pi k/2$ and are minimizers of $\mcF_d$ in a suitable space with symmetries on strips of size $d/k$. We remark that, for $k = 2$, these dipole solutions have a vortex at $y = d/2$ and an anti-vortex at $y = -d/2$, have some symmetry with respect to $y=0$ and they can be considered as the 2D analogue of the previously mentioned 3D vortex ring solution.

\begin{remark}\label{Rem:Fdk}
Observe that, for $N = 2$, $k > 1$ and $d > \sqrt{2}\pi k/2$,
\[
\mcF_d[u_{d,k}^\pm] > \mcF_d[u_d^\pm].
\]
In particular, the dipole solution in two dimensions, when it exists, has higher energy than the solitonic vortex solutions.
\end{remark}
For odd $k$, the remark follows from the fact that $u_d^\pm$ are the unique minimizers of $\mcF_d$ in $\mcA_d$, therefore $u_{d,k}^\pm$ have bigger energy. For $k=2$, by considering a periodic extension of $u_{d,2}^\pm$ using the Neumann boundary condition, and then translating this solution to have a vortex at the origin, yields a test function in $\mcA_d$. For even $k$, we refer to Subsection \ref{SSec2.2}.

 We believe that $\mcF_d$ has no critical point in $\mcA_d^{\rm x}$ which has energy less than $\mcF_d[u_s]$ when $d \leq \sqrt{2}\pi/2$, or $\mcF_d[u_d^\pm]$ when $d > \sqrt{2}\pi/2$, but we do not have a proof of this statement.

\subsubsection{The case $N = 3$}

Theorem \ref{Thm1} has a direct analogue in three dimensions. To catch the three dimensional solitonic vortex solutions, with the vortex line being the intersection of the $z$-axis with $\Omega_d$, we consider the following symmetry ansatz:
\begin{equation}
\begin{cases}
u_1(x,y,-z)  = u_1(x,y,z) = - u_1(x,-y,z) ,\\
u_2(x,y,-z) = u_2(x, y, z) = u_2(x, - y, z),
\end{cases}
	\label{Eq:3DSVSym}
\end{equation}
which replaces \eqref{Eq:Imparity}. Here we have written $x' = (y,z)$. 

We continue to use $\mcA_{d}$ to denote the set of maps $u$ in $u_s + H^1(\Omega_d,\RR^2)$ satisfying the phase imprinting condition \eqref{Eq:Imprinting} as well as \eqref{Eq:3DSVSym}. Let $j_{1,1}'$ be the smallest positive zero of $J_1'$, the derivative of the Bessel function $J_1$ of the first kind and of order one.

\begin{theorem}
\label{Thm1Ext2}
Let $N = 3$.
\begin{enumerate}[(i)]
\item If $0 < d \leq \sqrt{2}j_{1,1}'$, then the soliton solution $u_s$ is the unique minimizer of $\mcF_d$ in $\mcA_d$.

\item If $d > \sqrt{2}j_{1,1}'$, then the soliton solution $u_s$ is an unstable critical point of $\mcF_d$ in $\mcA_d$ and $\mcF_d$ has exactly two minimizers in $\mcA_d$ of the form $(\pm u_1, u_2)$ with $u_1 \neq 0$. Moreover, the zero set of these minimizers is a line segment, namely the intersection of $z$-axis with $\Omega_d$.
\end{enumerate}
\end{theorem}

In fact, in the body of the paper, we construct critical points whose zero set is any $\ell$ vortex lines on the $(y,z)$-plane which intersect equiangularly at the origin (see Theorem \ref{Thm1Ext2Gen}). These are recognized as the solutions constructed numerically in Mu\~noz Mateo and Brand \cite{Brand-PRL14} with quantum numbers $(0,\ell)$. They are also known as the solitonic vortex solutions when $\ell = 1$ and spoke wheel solutions when $\ell > 1$. As a consequence of our result, for any given $\ell$, the numerically constructed solutions in \cite{Brand-PRL14} with quantum numbers $(p,\ell)$ will have strictly larger energy when $p > 0$.

Our next result is a generalization of Theorem \ref{Thm:MPC} in 3D. 

\begin{theorem} \label{Thm:MPC-3D}
Let $N = 3$, $d > 0$ and let $\psi_n$ be given by \eqref{Eq:psinDef} with $n = n(d)$ sufficiently large so that $\mcF_d[\psi_n] < \mcF_d[u_s]$. We have the following estimates for the mountain pass energy $c_d$ defined in \eqref{Eq:MPEnergy}:
\begin{align}
c_d  &= \mcF_d[u_s]  \text{ if } d \leq \sqrt{2} j'_{1,1},\label{Eq:3DMP-1}\\
0 < \underline{c} \leq c_d  &<  \mcF_d[u_s] \text{ if } d > \sqrt{2} j'_{1,1},\label{Eq:3DMP-2}
\end{align}
where $\underline{c}$ is a universal positive constant.
\end{theorem}

\begin{remark}
For $d$ close to but larger than $\sqrt{2} j'_{1,1}$, it can be shown using the same argument used in the 2D setting that the mountain pass energy $c_d$ is at least the energy of the solitonic vortex solutions. Whether $c_d$ is exactly the energy of the solitonic vortex solutions requires further work, which we have not pursued.
\end{remark}

We point out that, unlike the two-dimensional setting, Theorem \ref{Thm1Ext2} and its proof do not provide any insight on the existence of the previously mentioned vortex ring solution. (This corresponds to the numerically constructed solution in \cite{Brand-PRL14} with quantum numbers $(1,0)$.) We have attempted to prove its existence using a mountain pass argument based on energy estimates like those in Theorem \ref{Thm:MPC-3D}, but fell just short of this goal. Let us describe this in more details. Let $\mcA_d^{\rm rad}$ denote the space of maps $u$ in $\mcA_d^{\rm x}$ which satisfy 
\begin{equation}
u(x,x') = u(x,Rx') \text{ for any rotation $R$ on the $(y,z)$-plane},
	\label{Eq:RadSym}
\end{equation}
and define
 \begin{align}
\Gamma_d^{\rm rad}
	&= \Big\{h \in C([-1,1],\mcA_d^{\rm rad}): h(-1)  = \psi_n^-, h(1) = \psi_n^+\Big\}, \label{Eq:GamRadDef}\\
c_{d}^{\rm rad}
	&= \inf_{h \in \Gamma_d^{\rm rad}} \max_{t \in [-1,1]} \mcF_d[h(t)].
\label{Eq:MPRadEnergy}
\end{align}
Arguing as in the proof of Theorem \ref{Thm:MPC-3D}, we get the estimates
\begin{align}
c_d^{\rm rad}  &= \mcF_d[u_s]  \text{ if } d \leq \sqrt{2} j'_{0,1},\label{Eq:3DMPRad-1}\\
0 < \underline{c}^{\rm rad} \leq c_d^{\rm rad}  &<  \mcF_d[u_s] \text{ if } d > \sqrt{2} j'_{0,1},\label{Eq:3DMPRad-2}
\end{align}
where $j_{0,1}'$ be the smallest positive zero of $J_0'$, the derivative of the Bessel function $J_0$ of the first kind and of order zero. However, due to the failure of the Palais-Smale condition, we do not know if $c_d^{\rm rad}$ is a critical value for $\mcF_d$ in $\mcA_d^{\rm rad}$ for $d > \sqrt{2} j'_{0,1}$. In fact, applying the mountain pass theorem on finite cylinders and sending the length of the cylinders to infinity, we obtain a candidate solution:

\begin{theorem}\label{Thm:RingCand}
Let $N = 3$ and $d > \sqrt{2}j'_{0,1}$. There exists a pair of finite energy solutions $\hat u_d^\pm = (\pm\hat U_d, \hat V_d) \in H^1_{\rm loc}(\Omega_d,\RR^2)$  to the problem
\begin{equation}
\begin{cases}
-\Delta \hat u_d^\pm = (1 - |\hat u_d^\pm|^2) \hat u_d^\pm \text{ in } \Omega_d,\\
\partial_n \hat u_d^\pm = 0 \text{ on } \partial\Omega_d,
\end{cases}
	\label{Eq:CritNoBC}
\end{equation}
which satisfies the phase imprinting condition \eqref{Eq:Imprinting}, the symmetry condition \eqref{Eq:RadSym}, and the following properties.
\begin{enumerate}[(i)]
\item $\hat U_d$ changes sign on the cross-section $\{x = 0\} \cap \Omega_d$. In particular, $\hat u_d^\pm$ are different from the soliton solution and the zero set of $\hat u_d^\pm$ is the union of a (possibly infinite) number of circles on the cross-section $\{x = 0\} \cap \Omega_d$.
\item It holds that
\begin{equation}
\begin{cases}
\hat U_d^2 + \hat V_d^2 < 1 \text{ in } \bar \Omega_d,\\
\hat V_d > 0 \text{ for } x > 0, \hat V_d = 0 \text{ for } x = 0, \text{ and } \hat V_d < 0 \text{ for } x < 0.
\end{cases}
	\label{Eq:hatUVSign}
\end{equation}
\item As $x \rightarrow \pm\infty$, $\hat u_d(x, \cdot) \rightarrow (a_d,\pm b_d)$ for some constant vector $(a_d,b_d) \in \mathbb{S}^1$ with $b_d \in (0,1]$.
\end{enumerate}
\end{theorem}

\begin{remark}
It is an open problem if $(a_d,b_d) = (0,1)$. If this holds, $\hat u_d^\pm$ are then critical points of $\mcF_d$ in $\mcA_d^{\rm rad}$. 
\end{remark}

\section{$2D$ equilibrium configurations}

In this section, we study the two-dimensional case. Beside the main aim of proving Theorem \ref{Thm1} (see Subsection \ref{SSec2.1}) and Theorem \ref{Thm:MPC} (see Subsection \ref{SSec2.2}), we also establish the instability of the soliton solution and solitonic vortex solution in $\mcA_d^{\rm x}$ as well as the super-polynomial decay at infinity of the difference of these solutions.

\subsection{Minimizing $\mcF_d$ in $\mcA_d$}\label{SSec2.1}
\begin{theorem}\label{Thm:S=>M}
Let $d > 0$. If the soliton solution $u_s$ is stable in the first direction, i.e.
\begin{multline*}
\frac{d^2}{dt^2}\Big|_{t = 0} \mcF_d[u_s + t(\varphi,0)] = \int_{\Omega_d} \big(|\nabla \varphi|^2 - (1 - |u_s|^2) \varphi^2\big)\,dx \geq 0 \text{ for all } \varphi \in H^1(\Omega_d) \\
	\text{ satisfying }  \varphi(x,y) = \varphi(-x,y) =  - \varphi(-x,-y),
\end{multline*}
then $u_s$ is the unique minimizer of $\mcF_d$ in $\mcA_d$. 
\end{theorem}

\begin{proof}
Let us first prove the minimality of $u_s$. We follow the strategy in \cite{INSZ_CRAS18,INSZ-ENS}. Let $w \in H^1(\Omega,\RR^2)$ be such that $u_s + w \in \mcA_d$. We compute
\begin{align}
\mcF_d[u_s + w] - \mcF_d[u_s]
	&= \int_{\Omega_d} \big(\nabla u_s : \nabla w - (1 - |u_s|^2) u_s \cdot w\big)\,dx\,dy\nonumber\\
		&\qquad + \frac{1}{2}\int_{\Omega_d} \Big(|\nabla w|^2 - (1 - |u_s|^2) |w|^2 + \frac{1}{2} (2 u_s \cdot w + |w|^2)^2\Big)\,dx\,dy\nonumber\\
	&= \frac{1}{2}\int_{\Omega_d} \Big(|\nabla w|^2 - (1 - |u_s|^2) |w|^2 + \frac{1}{2} (2 u_s \cdot w + |w|^2)^2\Big)\,dx\,dy.
	\label{Eq:usM-0}
\end{align}
Therefore, in order to prove the minimality of $u_s$, it suffices to show that
\[
\int_{\Omega_d} \Big(|\nabla w|^2 - (1 - |u_s|^2) |w|^2 \Big)\,dx\,dy \geq 0.
\]
The stability of $u_s$ in the first direction implies that
\[
\int_{\Omega_d} \Big(|\nabla w_1|^2 - (1 - |u_s|^2) w_1^2 \Big)\,dx\,dy \geq 0.
\]
Therefore, we only need to show that
\begin{multline*}
\int_{\Omega_d} \big(|\nabla \psi|^2 - (1 - |u_s|^2) \psi^2\big)\,dx\,dy \geq 0 \text{ for all } \psi\in H^1(\Omega_d) \\
	\text{ satisfying }  \psi(x,y) = -\psi(-x,y) =  - \psi(-x,-y).
\end{multline*}
In fact, we prove a slightly stronger inequality:
\begin{equation}
\int_{\Omega_d} \big(|\nabla \psi|^2 - (1 - |u_s|^2) \psi^2\big)\,dx\,dy \geq 0 \text{ for all } \psi\in H^1(\Omega_d)  \text{ satisfying }  \psi(0,\cdot) = 0.
	\label{Eq:usM-1}
\end{equation}
We will prove \eqref{Eq:usM-1} into two steps.

\medskip
\noindent\underline{Step 1:} Proof of \eqref{Eq:usM-1} when $\psi \in C_c^\infty(\bar\Omega_d)$. 

Let
\[
\chi(x,y) = \frac{1}{x} \psi(x,y) \text{ for } x \neq 0.
\]
Since $\psi(0,y) = 0$, we have that
\begin{align*}
\chi(x,y) 
	&= \frac{1}{x}\int_0^x \partial_x \psi(\xi,y)\,d\xi,\\
\partial_x \chi(x,y)
	&= \frac{1}{x^2} \Big[x \partial_x  \psi(x,y) -  \int_0^x \partial_x \psi(\xi,y)\,d\xi\Big]\\
	&= \frac{1}{x^2} \int_0^x \int_\xi^x \partial_x^2(t,y)dt\,d\xi,\\
\partial_y\chi(x,y) 
	&= \frac{1}{x}\int_0^x \partial_y \partial_x \psi(\xi,y)\,d\xi.
\end{align*}
In particular, $\chi \in C_c^{0,1}(\bar \Omega_d)$. It follows that we may write
\[
\psi(x,y) = \tanh \frac{x}{\sqrt{2}}\, \zeta(x,y) = u_{s,2}(x,y)  \zeta(x,y)\text{ with } \zeta \in C_c^{0,1}(\bar \Omega_d).
\]

We now compute
\begin{align}
&\int_{\Omega_d} \big(|\nabla \psi|^2 - (1 - |u_s|^2) \psi^2\big)\,dx\,dy\nonumber\\
	&\qquad= \int_{\Omega_d} \big(|\nabla (u_{s,2} \zeta)|^2 - (1 - |u_s|^2) u_{s,2}^2 \zeta^2\big)\,dx\,dy\nonumber\\
	&\qquad= \int_{\Omega_d} \Big( u_{s,2}^2 |\nabla \zeta|^2 + \nabla(\zeta^2 u_s) \cdot \nabla u_s - (1 - |u_s|^2) u_{s,2}^2 \zeta^2\big)\,dx\,dy\nonumber\\
	&\qquad= \int_{\Omega_d}   u_{s,2}^2 |\nabla \zeta|^2 \,dx\,dy \geq 0,
	\label{Eq:usM-2}
\end{align}
where for the last identity we integrated by parts (noting that $\partial_y u_s = 0$ and $\zeta$ has compact support) and used the identity $-\Delta u_{s,2} = (1 - |u_s|^2) u_{s,2}$. This concludes Step 1.

\medskip
\noindent\underline{Step 2:} Proof of \eqref{Eq:usM-1} for general $\psi$, which yields the minimality of $u_s$.

Using an even extension across $y = \pm d$, we may assume that $\psi \in H^1(\Omega_{3d})$. Using an even mollification kernel, we may find $\psi_j \in C^\infty(\bar \Omega_d) \cap H^1(\Omega_d)$ such that $\psi_j \rightarrow \psi$ in $H^1(\Omega_d)$ and $\psi_j(0, \cdot) = 0$. Multiplying $\psi_j$ by $\eta(x/k_j)$ for some even cut-off function $\eta \in C_c^\infty(-2,2)$ with $\eta \equiv 1$ in $(-1,1)$ and some suitable sequence $k_j \rightarrow \infty$, we may also assume that $\psi_j$ has compact support. By \eqref{Eq:usM-2} in Step 1, we have
\[
\int_{\Omega_d} \big(|\nabla \psi_j|^2 - (1 - |u_s|^2) \psi_j^2\big)\,dx\,dy
	= \int_{\Omega_d}   u_{s,2}^2 \Big|\nabla \big(\frac{\psi_j}{u_{s,2}}\big)\Big|^2 \,dx\,dy.
\]
Sending $j \rightarrow \infty$ and using Fatou's lemma on the right hand side, we obtain\footnote{By applying \eqref{Eq:usM-2} to $\psi_j - \psi_m$, it is readily seen that $u_{s,2} \nabla \big(\frac{\psi_j}{u_{s,2}}\big)$ is Cauchy and hence converges in $L^2(\Omega_d)$. Thus, the first inequality in \eqref{Eq:usM-3} is in fact an identity, but we do not need this fact.}
\begin{equation}
\int_{\Omega_d} \big(|\nabla \psi |^2 - (1 - |u_s|^2) \psi ^2\big)\,dx\,dy
	\geq \int_{\Omega_d}   u_{s,2}^2 \Big|\nabla \big(\frac{\psi}{u_{s,2}}\big)\Big|^2 \,dx\,dy
	\geq 0.
	\label{Eq:usM-3}
\end{equation}
We have thus proved \eqref{Eq:usM-1}. As explained earlier, this gives the minimality of  $u_s$.

Now, let us turn to the proof of the assertion that $u_s$ is the unique minimizer of $\mcF_d$ in $\mcA_d$. Indeed, suppose that $u_s + w \in \mcA_d$ is such that $\mcF_d[u_s] = \mcF_d[u_s + w]$. Then, the above argument gives
\[
0 =  \int_{\Omega_d} \big(|\nabla w_2|^2 - (1 - |u_s|^2) w_2^2 \big)\,dx\,dy =  \int_{\Omega_d}  (2 u_s \cdot w + |w|^2)^2 \,dx\,dy.
\]
By \eqref{Eq:usM-3} with $\psi = w_2$, we deduce that $\frac{w_2}{u_{s,2}}$ is constant. As $w_2 \in H^1(\Omega_d)$, we thus have $w_2 = 0$. Consequently
\[
0  =  \int_{\Omega_d}  (2 u_s \cdot w + |w|^2)^2 \,dx\,dy = \int_{\Omega_d} w_1^4\,dx\,dy,
\]
which implies that $w_1 = 0$. We have shown that $w = 0$, as desired.
\end{proof}

\begin{corollary}\label{Cor:d1}
Let $d > 0$. If $d \leq \sqrt{2}\pi/2$, then the soliton solution $u_s$ is the unique minimizer of $\mcF_d$ in $\mcA_d$. If $d > \sqrt{2}\pi/2$, then $u_s$ is an unstable critical point of $\mcF_d$ in $\mcA_d$.
\end{corollary}

\begin{proof} Consider the case $d \leq \sqrt{2}\pi/2$.  By Theorem \ref{Thm:S=>M}, we only need to show that $u_s$ is stable in the first direction. Indeed, let $\varphi \in H^1(\Omega_d)$ be such that $\varphi(x,y) = \varphi(-x,y) = -\varphi(-x,-y)$. We need to show that 
\[
\int_{\Omega_d} \Big(|\nabla \varphi|^2 - \mathrm{sech}^2\frac{x}{\sqrt{2}} \varphi^2\Big)\,dx\,dy \geq 0.
\]

By density, we may assume that $\varphi \in C_c^\infty(\bar\Omega_d)$. Since $\varphi$ is odd with respect to $y$ and hence $\varphi = 0$ on $\{y = 0\}$. By Wirtinger's inequality, we have
\[
\int_{-d}^d (\partial_y \varphi)^2 \,dy \geq \frac{\pi^2}{4d^2} \int_{-d}^d \varphi^2\,dy.
\]
Therefore, we only need to show that
\[
\int_{\Omega_d} \Big((\partial_x \varphi)^2 + \frac{\pi^2}{4d^2} \varphi^2 - \mathrm{sech}^2\frac{x}{\sqrt{2}} \varphi^2\Big)\,dx\,dy \geq 0.
\]
Indeed, we factorize $\varphi = A \eta$ with $A(x) = \mathrm{sech} \frac{x}{\sqrt{2}} > 0$ and compute
\begin{align*}
&
\int_{\Omega_d} \Big((\partial_x \varphi)^2 + \frac{\pi^2}{4d^2} \varphi^2 - \mathrm{sech}^2\frac{x}{\sqrt{2}} \varphi^2\Big)\,dx\,dy\\
	&\qquad= \int_{\Omega_d} \Big((\partial_x (A\eta))^2 + \frac{\pi^2}{4d^2} A^2 \eta^2 - \mathrm{sech}^2\frac{x}{\sqrt{2}} A^2 \eta^2\Big)\,dx\,dy\\
	&\qquad = \int_{\Omega_d} \Big(A^2 (\partial_x \eta)^2 + \partial_x( \eta^2 A) \partial_x A + \frac{\pi^2}{4d^2} A^2 \eta^2 - \mathrm{sech}^2\frac{x}{\sqrt{2}} A^2 \eta^2\Big)\,dx\,dy\\
	&\qquad = \int_{\Omega_d} \Big(A^2 (\partial_x \eta)^2 + (\frac{\pi^2}{4d^2} - \frac{1}{2}) A^2 \eta^2 \Big)\,dx\,dy \geq 0,
\end{align*}
where for the last identity we integrated by parts and used the identity 
\[
\partial_x^2 A = (\frac{1}{2} - \mathrm{sech}^2\frac{x}{\sqrt{2}}) A.
\]
We have thus shown that, when $d \leq \sqrt{2}\pi/2$, $u_s$ is stable in the first direction, and hence, by Theorem \ref{Thm:S=>M}, is the unique minimizer of $\mcF_d$ in $\mcA_d$.

Consider the case $d > \sqrt{2}\pi/2$. Choosing $\varphi(x,y) = A(x) \sin \frac{\pi y}{2d}$, we compute
\begin{align*}
\frac{d^2}{dt^2}\Big|_{t = 0} \mcF[u_s + t(\varphi,0)] 
	&= \int_{\Omega_d} \big(|\nabla \varphi|^2 - (1 - |u_s|^2) \varphi^2\big)\,dx\,dy\\
	&= d \int_{-\infty}^\infty \Big[(A')^2 + \frac{\pi^2}{4d^2} A^2 - \mathrm{sech}^2  \frac{x}{\sqrt{2}} \, A^2\Big]\,dx\\
	&= d \int_{-\infty}^\infty \Big(\frac{3}{2} \tanh^2 \frac{x}{\sqrt{2}} + \frac{\pi^2}{4d^2} - 1\Big)\mathrm{sech}^2  \frac{x}{\sqrt{2}} \,dx\\
	&= d \int_{-1}^1 \Big(\frac{3}{2} \tau^2 + \frac{\pi^2}{4d^2} - 1\Big) \sqrt{2} d\tau\\
	&= d\sqrt{2}(\frac{\pi^2}{2d^2} - 1) < 0.
\end{align*}
This shows that $u_s$ is an unstable critical point of $\mcF_d$ in $\mcA_d$. The proof is complete.
\end{proof}

\begin{remark}\label{Rem:usUnstable}
We note that $u_s$ is an unstable critical point of $\mcF_d$ in $\mcA_d^{\rm x}$ for every $d > 0$. In other words, $u_s$ is unstable when the symmetry condition \eqref{Eq:Imparity} is lifted. 
\end{remark}

\begin{proof}
Let $\varphi(x,y) = A(x) = \mathrm{sech}\frac{x}{\sqrt{2}}$. Then $\varphi \in C^\infty(\Omega_d) \cap H^1(\Omega_d)$, $u_s + t(\varphi,0) \in \mcA_d^{\rm x}$ for all $t \in \RR$, and
\begin{align*}
\frac{d^2}{dt^2}\Big|_{t = 0} \mcF[u_s + t(\varphi,0)] 
	&= \int_{\Omega_d} \big(|\nabla \varphi|^2 - (1 - |u_s|^2) \varphi^2\big)\,dx\,dy\\
	&= 2d \int_{-\infty}^\infty \Big[(A')^2  - \mathrm{sech}^2  \frac{x}{\sqrt{2}} \, A^2\Big]\,dx\\
	&= 2d \int_{-\infty}^\infty \Big(\frac{3}{2} \tanh^2 \frac{x}{\sqrt{2}}  - 1\Big)\mathrm{sech}^2  \frac{x}{\sqrt{2}} \,dx\\
	&= 2d \int_{-1}^1 \Big(\frac{3}{2} \tau^2   - 1\Big) \sqrt{2} d\tau\\
	&= -2d\sqrt{2} < 0.
\end{align*}
This shows that $u_s$ is an unstable critical point of $\mcF_d$ in $\mcA_d^{\rm x}$.
\end{proof}

\begin{theorem}\label{Thm:UV}
Suppose that $d > \sqrt{2} \pi/2$. Then $\mcF_d$ has exactly two minimizers $u_d^\pm$ in $\mcA_d$. They take the form $u_d^\pm = (\pm U_d, V_d)$ with $U_d$ and $V_d$ satisfying \eqref{Eq:UVSign}. In particular, the origin is the unique zero of $u_d^\pm$ and $u_d^\pm$ has degree $\pm 1$ at the origin.
\end{theorem}

\begin{proof} \underline{Step 1:} Existence of a minimizer $(U,V)$ of $\mcF_d$ in $\mcA_d$ satisfying \eqref{Eq:UVSign}. 


Let $Q_d = (0,\infty) \times (0,d)$ and $\mcB_d$ be the set of maps $u = (u_1, u_2)$ in $u_s + H^1(Q_d,\RR^2) = e_2 + H^1(Q_d,\RR^2)$ such that $u_1 = 0$ on $\{y = 0\}$ and $u_2 = 0$ on $\{x = 0\}$. Consider
\[
\bar \mcF_d[u] = \int_{Q_d}\Big( \frac{1}{2}|\nabla u|^2 + \frac{1}{4} (1 - |u|^2)^2\Big)\,dx\,dy, \qquad u \in \mcB_d.
\]
We claim that $\bar \mcF_d$ has a minimizer $(U,V)$ in $\mcB_d$ with $U^2 + V^2 \leq 1$, $U \leq 0$ and $V \geq 0$. Let $\{u_m\} \subset \mcB_d$ be a minimizing sequence. By truncation, we may assume that $|u_m| \leq 1$. Replacing $u_m$ by $(-|u_{m,1}|, |u_{m,2}|)$, we may assume that $u_{m,1} \leq 0$ and $u_{m,2} \geq 0$. Passing to a subsequence, we may assume that $u_m \rightharpoonup (U,V)$ in $H^1_{\rm loc}(Q_d,\RR^2)$, $U^2 + V^2 \leq 1$, $U \leq 0$ and $V \geq 0$. Moreover, by Ekeland's variational principle (\cite{Ekeland74}), there exists a sequence $\{\tilde u_m\} \subset \mcB_d$ such that $u_m - \tilde u_m \rightarrow 0$ in $H^1(Q_d,\RR^2)$ and and $-\Delta \tilde u_m - (1 - |\tilde u_m|^2)\tilde u_m \rightarrow 0$ in $H^{-1}(Q_d,\RR^2)$. It follows that $\tilde u_m \rightarrow(U,V)$ in $H^1_{\rm loc}(Q_d,\RR^2)$ and $(U,V)$ satisfies the Euler-Lagrange equation of $\bar\mcF_d$. To see that $(U,V)$ is a minimizer of $\bar\mcF_d$ in $\mcB_d$, it suffices to show that $(U,V) \in \mcB_d$. By construction, we have $\nabla U, \nabla V \in (L^2(Q_d))^2$. Since $U = 0$ on $\{y = 0\}$, we have by Wirtinger's inequality that $U \in L^2(Q_d)$. It remains to show that $V - u_{s,2} \in L^2(Q_d)$, which is equivalent to $V - 1 \in L^2(Q_d)$. Arguing as in \cite[Section 2.1]{AS-23}, we have that all the partial derivatives of $U$ and $V$ are bounded in $Q_d$ and that $\nabla U \rightarrow 0, \nabla V \rightarrow 0, U \rightarrow 0, V \rightarrow  1$ uniformly  as $x \rightarrow +\infty$. Recall that
\[
\begin{cases}
-\Delta V = (1 - U^2 - V^2) V \text{ in } Q_d,\\
V = 0 \text{ on } \{x = 0\},\\
\partial_y V = 0 \text{ on } \{y = 0\} \cup \{y = d\}.
\end{cases}
\]
Multiplying the equation by $V - 1$, integrating over $[x_1, x_2] \times [0,d]$ and integrating by parts yield
\begin{multline*}
\int_{[x_1, x_2] \times [0,d]} \big[|\nabla V|^2 + (1 - U^2 - V^2) V |V - 1|\big]\,dx \\
= \int_0^d \partial_x V (V - 1)\big|_{x = x_2}\,dy
	- \int_0^d \partial_x V (V - 1)\big|_{x = x_1}\,dy.
\end{multline*}
By the boundedness of $\nabla V$ and the fact that $0 \leq V \leq 1$, the right hand side is bounded by a constant $C$ independent of $x_1$ and $x_2$. Since $U \in L^2$ and $0 \leq V \leq 1$, the contribution of the term $U^2V|V-1|$ to the integral on the left hand side is also bounded by a constant $C$ independent of $x_1$ and $x_2$. We deduce that there exists a constant $C$ such that
\[
\int_{[x_1, x_2] \times [0,d]}   (V-1)^2 V \,dx \leq C \text{ for all } 0 \leq x_1 < x_2 < \infty.
\]
Since $V \rightarrow 1$ uniformly as $x \rightarrow \infty$, the above inequality implies that $V - 1 \in L^2(Q_d)$. We have thus shown that $(U,V)$ is a minimizer of $\bar\mcF_d$ in $\mcB_d$.

By an abuse of notation, we also use $U$ to denote its extension to $\Omega_d$ which is even in $x$ and odd in $y$, and $V$ to denote its extension of $\Omega_d$ which is odd in $x$ and even in $y$. Since $\mcF_d[u] = 4\bar\mcF_d[u|_{Q_d}]$, we have that $(U,V)$ is a minimizer for $\mcF_d$ in $\mcA_d$.

It is clear that $V = 0$ on $\{x = 0\}$. Note that $V$ satisfies
\begin{equation}
\begin{cases}
-\Delta V = (1 - U^2 - V^2)V \text{ in } \{x > 0\},\\
V \geq 0 \text{ in } \{x > 0\},\\
V(x,y) \rightarrow 1 \text{ as } x \rightarrow \infty.
\end{cases}
	\label{Eq:U}
\end{equation}
By the strong maximum principle, we have that $V > 0$ in $\{x > 0\}$. Since $V$ is odd in $x$, this also gives $V < 0$ in $\{x < 0\}$.

Next, we have that $U = 0$ on $\{y = 0\}$ and $U$ satisfies
\begin{equation}
\begin{cases}
-\Delta U = (1 - U^2 - V^2)U \text{ in } \{0 < y < d\},\\
U \leq 0 \text{ in } \{0 < y < d\},\\
U(x,y) \rightarrow 0 \text{ as } x \rightarrow \pm\infty.
\end{cases}
	\label{Eq:V}
\end{equation}
By the strong maximum principle, either $U \equiv 0$ in $\{0 < y < d\}$ or $U < 0$ in $\{0 < y < d\}$. If $U \equiv 0$, we then have
\[
\mcF_d[(U,V)]
	\geq \int_{-d}^d \int_{-\infty}^\infty \Big( \frac{1}{2}(\partial_x V)^2 + \frac{1}{4} (1 - V^2)^2\Big)\,dx\,dy.
\]
Minimizing the inner integral for each fixed $y$, we deduce that
\[
\mcF_d[(U,V)]
	\geq \int_{-d}^d \int_{-\infty}^\infty \Big( \frac{1}{2}(\partial_x u_{s,2})^2 + \frac{1}{4} (1 - u_{s,2}^2)^2\Big)\,dx\,dy = \mcF_d[u_s].
\]
This implies that $u_s$ is also a minimizer of $\mcF_d$ in $\mcA_d$, which contradicts Corollary \ref{Cor:d1} as $d > \sqrt{2}\pi/2$. We must therefore have that $U < 0$ in $\{0 < y < d\}$ and hence $U > 0$ in $\{-d < y < 0\}$. 

We next show that $U^2 + V^2 < 1$ in $\bar \Omega_d$. Suppose by contradiction that there exists $(x_0, y_0) \in \bar\Omega_d$ such that $U^2(x_0,y_0) + V^2(x_0,y_0) = 1$. Extend $(U,V)$ to $\Omega_{3d}$ by evenly reflecting across $y = \pm d$. Then the function $\Psi := 1 - U^2 - V^2$ satisfies
\[
\begin{cases}
\Delta \Psi = - 2|\nabla U|^2 - 2|\nabla V|^2 + 2  (U^2 + V^2) \Psi \leq 2  (U^2 + V^2) \Psi \text{ in } \Omega_{3d},\\
\Psi \geq 0 \text{ in } \Omega_{3d},\\
\min_{\bar\Omega_{3d}} \Psi \text{ is attained at } (x_0,y_0) \in \bar \Omega_d \subset \Omega_{3d}.
\end{cases}
\]
By the strong maximum principle, this implies that $\Psi \equiv 0$ and $\nabla U \equiv \nabla V \equiv 0$ in $\Omega_{3d}$. The latter conclusion is inconsistent with the fact that $V \rightarrow \pm 1$ as $x\rightarrow \pm \infty$. We thus have that $U^2 + V^2 < 1$ in $\bar \Omega_d$.

Summarizing, we have shown that there is a minimizer $(U,V)$ of $\mcF_d$ in $\mcA_d$ which satisfies \eqref{Eq:UVSign}.

\medskip
\noindent\underline{Step 2:} We prove that if $w \in H^1(\Omega_d,\RR^2)$ is such that $(U,V) + w \in \mcA_d$, then
\begin{equation}
\mcF_d[(U,V) + w] - \mcF_d[(U,V)] \geq \frac{1}{2} \int_{\Omega_d}\Big[ U^2 \Big|\nabla\big(\frac{w_1}{U}\big)\Big|^2 + V^2 \Big|\nabla\big(\frac{w_2}{V}\big)\Big|^2\Big]\,dx\,dy.
	\label{Eq:UV-1}
\end{equation}

We follow an argument in \cite{INSZ_CRAS18, INSZ-ENS}. A computation similar to \eqref{Eq:usM-0} gives
\begin{align*}
\mcF_d[(U,V) + w] - \mcF_d[(U,V)]
	&= \frac{1}{2} \int_{\Omega_d} \Big[|\nabla w|^2 - (1 - U^2 - V^2) |w|^2 \\
		&\qquad+ \frac{1}{2}\big(2 (U,V) \cdot w + |w|^2\big)^2\Big]\,dx\,dy\\
	&\geq \frac{1}{2} \int_{\Omega_d} \Big[|\nabla w|^2 - (1 - U^2 - V^2) |w|^2 \Big]\,dx\,dy.
\end{align*}
Therefore, to prove \eqref{Eq:UV-1} it suffices to show that
\begin{align}
\int_{\Omega_d} \Big[|\nabla w_1|^2 - (1 - U^2 - V^2) w_1^2 \Big]\,dx\,dy
	& \geq  \int_{\Omega_d} U^2 \Big|\nabla\big(\frac{w_1}{U}\big)\Big|^2\,dx\,dy,
	\label{Eq:UV-1a}\\
\int_{\Omega_d} \Big[|\nabla w_2|^2 - (1 - U^2 - V^2) w_2^2 \Big]\,dx\,dy
	& \geq  \int_{\Omega_d}  V^2 \Big|\nabla\big(\frac{w_2}{V}\big)\Big|^2\Big]\,dx\,dy.
	\label{Eq:UV-1b}
\end{align}

By a density argument (as in Step 2 of the proof of Theorem \ref{Thm:S=>M}), it suffices to consider the case in which $w \in C_c^\infty(\bar\Omega_d,\RR^2)$. Since a point has zero Newtonian capacity in two dimensions, we may further assume that $w \in C_c^\infty(\bar\Omega_d \setminus \{(0,\pm d)\},\RR^2)$. Noting that by \eqref{Eq:U}, \eqref{Eq:V} and the Hopf lemma,
\[
\partial_y U < 0 \text{ on } \{y = 0\} \quad \text{ and } \quad \partial_x V > 0 \text{ on } \{x = 0, -d < y < d\}.
\]
Hence, the argument at the start of Step 1 of the proof of Theorem \ref{Thm:S=>M} shows that $\eta := \frac{w_1}{U}$ and $\zeta := \frac{w_2}{V}$ belong to $C_c^{0,1}(\bar\Omega_d \setminus \{(0,\pm d)\})$. 

Now, using
\[
\begin{cases}
-\Delta U = (1 - U^2 - V^2) U,\\
-\Delta V = (1 - U^2 - V^2) V,
\end{cases}
\]
we compute
\begin{align*}
&\int_{\Omega_d} \Big[|\nabla w_1|^2 - (1 - U^2 - V^2) w_1^2 \Big]\,dx\,dy\\
	&\qquad= \int_{\Omega_d} \Big[|\nabla (U\eta)|^2 - (1 - U^2 - V^2) U^2 \eta^2 \Big]\,dx\,dy\\
	&\qquad= \int_{\Omega_d} \Big[U^2 |\nabla \eta|^2 + \nabla (\eta^2 U) \cdot \nabla U - (1 - U^2 - V^2) U^2 \eta^2 \Big]\,dx\,dy\\
	&\qquad= \int_{\Omega_d} U^2 |\nabla \eta|^2 \,dx\,dy,
\end{align*}
and
\begin{align*}
&\int_{\Omega_d} \Big[|\nabla w_2|^2 - (1 - U^2 - V^2) w_2^2 \Big]\,dx\,dy\\
	&\qquad= \int_{\Omega_d} \Big[|\nabla (V\zeta)|^2 - (1 - U^2 - V^2) V^2 \zeta^2 \Big]\,dx\,dy\\
	&\qquad= \int_{\Omega_d} \Big[V^2 |\nabla \zeta|^2 + \nabla (\zeta^2 V) \cdot \nabla V - (1 - U^2 - V^2) V^2 \zeta^2 \Big]\,dx\,dy\\
	&\qquad= \int_{\Omega_d} V^2 |\nabla \zeta|^2 \,dx\,dy.
\end{align*}
This concludes the proof of \eqref{Eq:UV-1a}, \eqref{Eq:UV-1b} and hence \eqref{Eq:UV-1}.

\medskip
\noindent\underline{Step 3:} We prove that if $u$ is a minimizer of $\mcF_d$ in $\mcA_d$, then either $u = (U,V)$ or $u = (-U,V)$.

Since both $u$ and $(U,V)$ are minimizers of $\mcF_d$ in $\mcA_d$, we have by \eqref{Eq:UV-1} that $\frac{u_1 - U}{U}$ and $\frac{u_2 - V}{V}$ are constant in $\Omega_d$. In other words, $u_1 = c_1U$ and $u_2 = c_2 V$ for some constants $c_1, c_2$. By the boundary condition for $V$ as $x \rightarrow \infty$, we must have that $c_2 = 1$ and $u_2 = V$. It follows that, on one hand,
\[
-\Delta u_2 = (1 - u_1^2 - u_2^2)u_2 = (1 - c_1^2 U^2 - V^2)V,
\]
and on the other hand
\[
-\Delta u_2 = -\Delta V = (1 - U^2 - V^2)V.
\]
Since $U < 0$ and $V > 0$ in $Q_d$, these identities imply that $c_1^2 = 1$, i.e. $u_1 = \pm U$. The proof is complete.
\end{proof}

\begin{proof}[Proof of Theorem \ref{Thm1}]
The theorem follows from Corollary \ref{Cor:d1} and Theorem \ref{Thm:UV}.
\end{proof}

\begin{remark}\label{Rem:PD}
As $x \rightarrow \infty$, $u_d^\pm - u_s$ decays faster than any polynomials. In fact, for any $k > 0$ and bounded interval $I \subset (\sqrt{2}\pi/2,\infty)$, 
\begin{equation}
\sup_{d \in I} \sup_{(x,y) \in \Omega_d} (1 + |x|)^k |u_d^\pm - u_s|(x,y) < \infty.
	\label{Eq:UD}
\end{equation}
\end{remark}

\begin{proof}
\underline{Step 1:} Let us first prove that $(U_d,V_d) - u_s$ decays faster than any polynomials for a fixed $d > \sqrt{2}\pi/2$. 

We abbreviate $(U,V) = (U_d,V_d)$. It is convenient to extend $(U,V)$ to all of $\RR^2$ by repeatedly and evenly reflecting them across the lines $y = (2\ell + 1)d$ for $\ell \in \ZZ$.

We will only consider the asymptotic as $x \rightarrow +\infty$, as the other side then follows from the phase imprinting condition \eqref{Eq:Imprinting}. Let
\[
\omega(x) = \sup_{s > x} (|U(s,y)| + |V(s,y) - 1|).
\]
Fix $R_0 > 2$ such that $\frac{1}{2} \leq V(x,y) < 1$ for $x \geq R_0$.

In the rest of this step, $C$ denotes a constant depending only on $R_0$ and $d$ which may change from lines to lines.

Let $\eta \in C^\infty(\RR)$ be such that $\eta = 0$ in $(-\infty,1)$, $\eta = 1$ in $(2,+\infty)$ and $|\eta'| \leq 2$. Let $R > R_0$ and $\eta_R(x) = \eta(x/R)$. Multiplying the equation $-\Delta (V - 1) = (1 - U^2 - V^2) V$ by $\eta_R^2(V - 1)$ and integrating over $\Omega_d$, we get
\[
\int_{\Omega_d} \eta_R^2 \big[|\nabla V|^2 + (1 - U^2 - V^2)|V||V -1|\big]\,dx\,dy
	=\int_{\Omega_d} 2\eta_R \eta_R' (V-1)  \partial_x V\,dx\,dy.
\]
Using Cauchy-Schwarz' inequality, we deduce that
\begin{multline}
\int_{2R}^\infty \int_{-d}^d\big[|\nabla V|^2 + (1 - U^2 - V^2)|V||V - 1|\big]\,dy\,dx\\
	 \leq \frac{16}{R^2}  \int_{R}^{2R} \int_{-d}^d   (V-1)^2 \,dy\,dx \leq C R^{-1}\omega(R)^2.
	\label{Eq:D1}
\end{multline}
In a similar vein, we test $-\Delta U = (1 - U^2 - V^2) U$ against $\eta_R^2U$ and obtain
\[
\int_{2R}^\infty \int_{-d}^d \big[|\nabla U|^2 - (1 - U^2 - V^2)U^2\big]\,dy\,dx
	\leq \frac{16}{R^2}  \int_{R}^{2R} \int_{-d}^d  U^2 \,dy\,dx \leq CR^{-1}\omega(R)^2.
\]
Using $(1 - U^2 - V^2)U^2 \leq (1 - U^2 - V^2)(1 - V^2) \leq C(1 - U^2 - V^2)|V||V - 1|$, we obtain from \eqref{Eq:D1} that
\[
\int_{2R}^\infty \int_{-d}^d  (1 - U^2 - V^2)U^2 \,dy\,dx \leq CR^{-1}\omega(R)^2.
\]
Putting the last two estimates together, we get
\begin{equation}
\int_{2R}^\infty \int_{-d}^d  |\nabla U|^2 \,dy\,dx
	\leq C R^{-1}\omega(R)^2.
	\label{Eq:D2}
\end{equation}
In view of \eqref{Eq:D1} and \eqref{Eq:D2}, we may apply elliptic estimates (on square of unit size) to the equation obtained by differentiating \eqref{EL} to get
\begin{equation}
|\nabla U(x,y)| + |\nabla V(x,y)| \leq C x^{ - 1/2}\omega(x/2) \text{ for } x \geq 2.
	\label{Eq:D3}
\end{equation}
Since $U(x,0) = 0$, we deduce from \eqref{Eq:D3} that
\begin{equation}
|U(x,y)| = |U(x,y) - U(x,0)| \leq C x^{ - 1/2}\omega(x/2) \text{ for } x \geq 2.
	\label{Eq:D3X}
\end{equation}

Recall that the function $\Psi := 1 - U^2 - V^2$ satisfies
\[
\Delta \Psi = - 2|\nabla U|^2 - 2|\nabla V|^2 + 2  (U^2 + V^2) \Psi.
\]
Fix some point $(x_0,y_0)$ with $x_0 > 4R_0$. Then, by \eqref{Eq:D3} and the fact that $V \geq  \frac{1}{2}$ for $x > R_0$, we have  that
\[
-\Delta \Psi +  \frac{1}{2} \Psi \leq  Cx_0^{-1}\omega(x_0/4) \text{ in the disk } D_{x_0/2}(x_0,y_0).
\]
Since $\Psi \leq 1$ on the boundary of the disk, a simple comparison argument gives
\[
\Psi \leq Cx_0^{-1}\omega(x_0/4) + \exp \frac{(x - x_0)^2 + (y - y_0)^2 - \frac{1}{4}x_0^2}{4x_0} \text{ in } D_{x_0/2}(x_0,y_0).
\]
In particular, $\Psi(x_0,y_0) \leq Cx_0^{-1}\omega(x_0/4) + \exp(-x_0/16)$. We thus have
\begin{equation}
0 < 1 - U^2(x,y) - V^2(x,y) \leq Cx^{-1}\omega(x/4) + \exp(-x/16) \text{ for } x \geq 4.
	\label{Eq:D4}
\end{equation}
Combining \eqref{Eq:D3X} and \eqref{Eq:D4}, we obtain
\[
0 < 1 - V(x,y) \leq 1 - V^2(x,y) \leq Cx^{-1}\omega(x/4) + \exp(-x/16) \text{ for } x \geq 4,
\]
and hence
\begin{equation}
\omega(x) \leq Cx^{-1/2}\omega(x/4) + \exp(-x/16) \text{ for } x \geq 4.
	\label{Eq:D5}
\end{equation}
In particular, for every $k \geq 0$, there exists $C_k > 0$ such that
\[
\sup_{x > 1} \omega(x) x^{k+1/2} \leq C_k \sup_{x > 1} \omega(x) x^{k} + C_k.
\]
As $\omega$ is bounded, this shows that $\omega$ decays faster than any polynomials as $x \rightarrow \infty$, as desired.

\medskip
\noindent\underline{Step 2:} We turn to the proof of \eqref{Eq:UD}. By inspecting the argument in Step 1, we only need to show that the constant $R_0$ can be chosen independent of $d \in I$. 

In this step, $C$ denotes a universal constant larger than $1$, which is in particular independent of $d \in I$.

Suppose by contradiction that the constant $R_0$ cannot be chosen independent of $d \in I$. Then there exist sequences $\{d_m\} \subset I$ and $\{R_m\} \subset (1,\infty)$ such that $d_m \rightarrow d_\infty \in \bar I$, $R_m \geq m$ such that 
\begin{equation}
\max_{y \in [-d_m,d_m]} V_{d_m}(R_m,y) < \frac{1}{2}.
	\label{Eq:UD-1}
\end{equation}

By abusing of notation, we denote $(U_{\sqrt{2}d/2}, V_{\sqrt{2}d/2}) := u_s$. By the minimality of $(U_d,V_d)$ (see Corollary \ref{Cor:d1} and Theorem \ref{Thm:UV}), we have
\begin{equation}
\sup_{d \in \bar I} \mcF_d[U_d,V_d] \leq \sup_{d \in \bar I} \mcF_d[u_s] < \infty. 
	\label{Eq:UD-0}
\end{equation}
By elliptic estimates, we have 
\[
\sup_{d \in \bar I} \sup_{\Omega_d} \big[|\nabla^\ell U_d| + |\nabla^\ell V_d|\big] < \infty.
\]
Hence, by Theorem \ref{Thm:UV}, $(U_d,V_d)$ depends smoothly on $d \in I$. In addition, by the minimality of $(U_d,V_d)$, as $d \rightarrow \sqrt{2}d/2$, $(U_d,V_d)$ tends to a minimizer of $\mcF_{\sqrt{2}d/2}$, which must be $u_s$, by Corollary \ref{Cor:d1}. Hence $(U_d,V_d)$ depends smoothly on $d \in \bar I$. 

Let
\[
\bar V_d(x) = \frac{1}{2d}\int_{-d}^d V_d(x,y)\,dy.
\]
Since $-\Delta V_d = (1-U_d^2 - V_d^2) V_d > 0$ in $\{x > 0\}$, the function $\bar V_d$ is concave in $(0,\infty)$. As $V_d(x) \rightarrow 1$ as $x \rightarrow +\infty$, it follows that $\bar V_d' > 0$ in $(0,\infty)$ and $\bar V_d'(x) \searrow 0$ as $x \rightarrow +\infty$.

Since $|\nabla V_d| \leq C$, we deduce from \eqref{Eq:UD-1} that there exists an subinterval of $[-d_m,d_m]$ of length at least $\frac{1}{C}$ on which $V_{d_m}(R_m,\cdot) < \frac{3}{4}$. Since $V_{d_m} > 0$ in $\{x > 0\}$, it follows that $\bar V_{d_m}(R_m) \leq 1 - \frac{1}{C}$. As $\bar V_{d_m}' > 0$ in $(0,\infty)$, it further follows that $\bar V_{d_m}(x) \leq  1 - \frac{1}{C}$ for $x \in [0,R_m]$. Sending $m \rightarrow \infty$, we deduce that
\[
\bar V_{d_\infty}(x) \leq  1 - \frac{1}{C} \text{ in } [0,\infty),
\]
which contradicts the fact that $V_{d_\infty}(x,y) \rightarrow 1$ uniformly as $x \rightarrow \infty$. The proof is complete.
\end{proof}

\begin{remark}\label{Rem:uvUnstable}
We note that, like the soliton solution, the solitonic vortex solution $u_d^\pm$ are unstable critical points of $\mcF_d$ in $\mcA_d^{\rm x}$ for $d > \sqrt{2}\pi/2$. That is, they are unstable when the symmetry condition \eqref{Eq:Imparity} is removed.
\end{remark}

\begin{proof}
It suffices to consider the stability of $u_d^+$, which we abbreviate as $u_d$. The underlying idea is the following: If we consider $u_d$ as a map defined on $\RR^2$ which is $2d$-periodic with respect to $y$, then the restriction of its translations in the $y$-direction to $\Omega_d$ has the same energy as $u_d$. Since these translated maps do not have zero Neumann derivative along $\partial\Omega_d$, we can perturb them slightly to achieve smaller energy. Proceeding in this way, we seek $w \in H^1(\Omega_d,\RR^2)$ satisfying the phase imprinting condition \eqref{Eq:Imprinting} such that
\begin{equation}
II[\partial_y u_d + w] := \frac{d^2}{dt^2}\Big|_{t=0} \mcF_d[u_d + t(\partial_y u_d + w)] < 0.
	\label{Eq:uvUnstab-1}
\end{equation}

We compute
\begin{align*}
&II[\partial_y u_d + w]\\
	&\quad= \frac{1}{2} \int_{\Omega_d} \Big\{|\nabla (\partial_y u_d + w)|^2 - (1 - |u_d|^2) |\partial_y u_d + w|^2 + 2 [u_d \cdot(\partial_y u_d + w)]^2\Big\}\,dx\,dy\\
	&\quad = \frac{1}{2} \int_{\Omega_d} \Big\{|\nabla (\partial_y u_d)|^2 - (1 - |u_d|^2) |\partial_y u_d |^2 + 2 [u_d \cdot(\partial_y u_d) ]^2\Big\}\,dx\,dy\\
		&\quad\quad +    \int_{\Omega_d} \Big\{ \nabla (\partial_y u_d) : \nabla w - (1 - |u_d|^2) (\partial_y u_d) \cdot w + 2 [u_d \cdot(\partial_y u_d)] ( u_d \cdot w) \Big\}\,dx\,dy\\
		&\quad\quad +   \frac{1}{2} \int_{\Omega_d} \Big\{|\nabla w|^2 - (1 - |u_d|^2) |w|^2 + 2 (u_d \cdot w )^2\Big\}\,dx\,dy\\
	&\quad =: T_1 + T_2(w) + T_3(w).
\end{align*}

Note that $\partial_y u_d$ satisfies
\begin{equation}
\begin{cases}-\Delta \partial_y u_d = (1 - |u_d|^2) \partial_y u_d - 2 [u_d \cdot (\partial_y u_d)] u_d  \text{ in } \Omega_d,\\
\partial_y u_d = 0 \text{ on } \partial \Omega_d.
\end{cases}
	\label{Eq:pyud}
\end{equation}
Testing \eqref{Eq:pyud} agains $\partial_y u_d$, we see that $T_1 = 0$. (Alternatively, $T_1$ is the second variation of $\mcF_d$ at $u_d$ and along the translations of $u_d$ and is thus zero.) Therefore
\[
II[\partial_y u_d + w] = T_2(w) + T_3(w).
\]
Since $T_2(w)$ is linear in $w$ and $T_3(w)$ is quadratic in $w$, in order to achieve \eqref{Eq:uvUnstab-1}, it suffices to show the existence of some $w \in H^1(\Omega_d,\RR^2)$ satisfying the phase imprinting condition \eqref{Eq:Imprinting} such that $T_2(w) \neq 0$. To this end, we test \eqref{Eq:pyud} against $w$ and obtain
\begin{equation}
T_2(w) = - \int_{-\infty}^\infty \partial_y^2 u_d(x,d) \cdot w(x,d)\,dx +  \int_{-\infty}^\infty \partial_y^2 u_d(x,-d) \cdot w(x,-d)\,dx.
	\label{Eq:uvUnstab-2}
\end{equation}
Now, by unique continuation (applied to \eqref{Eq:pyud}, considered as a linear system of PDE for $\partial_y u$ with homogeneous Dirichlet boundary condition), if $\partial_y^2 u_d \equiv 0$ on $\partial\Omega_d$, then $\partial_y u_d \equiv 0$ in $\Omega_d$. Since $U_d = 0$ on the $x$-axis, this implies $U_d \equiv 0$ in $\Omega_d$, contradicting the fact that $U_d < 0$ in $\{0 < y < d\}$. Thus, $\partial_y^2 u_d$ is nonzero somewhere on $\partial\Omega_d$, say in a neighborhood of a point $p \in \partial\Omega_d$. Since $\partial_y^2 u_d$  satisfies the phase imprinting condition \eqref{Eq:Imprinting}, it is also non-zero in a neighborhood of the mirror image $\tilde p$ of $p$ about the $y$-axis. It is now clear from \eqref{Eq:uvUnstab-2} that we may choose $w$ satisfying the phase imprinting condition \eqref{Eq:Imprinting} and supported near $p$ and $\tilde p$ such that $T_2(w) \neq 0$. The proof is complete.
\end{proof}

\subsection{Minimizing $\mcF_d$ in other subspaces of $\mcA_d^{\rm x}$}\label{SSec2.2}

Recall that, for positive integer $k$ and $d$ larger than $\sqrt{2}\pi k/2$ but close to this value, it was shown in \cite{AGS} that there are exactly two (local) branches of critical points bifurcating from the soliton solution, which we denoted by $u_{d,k}^\pm$ in the introduction. The natural symmetry property of these critical points is
\begin{equation}
\begin{cases}
u_1(x,y) = -u_1(x,\tilde y), u_2(x,y) = u_2(x,\tilde y) &\text{ if } \frac{y + \tilde y}{2} \in \{-d + \frac{(2j+1)d}{k}: j = 0, \ldots, k-1\},\\
u_1(x,y) = u_1(x,\tilde y), u_2(x,y) = u_2(x,\tilde y) &\text{ if } \frac{y + \tilde y}{2} \in \{-d + \frac{2jd}{k}: j = 1, \ldots, k-1\}.
\end{cases}
	\label{Eq:kSym}
\end{equation}
In other words, if we equally subdivide $\Omega_d$ into $k$ substrips of width $\frac{2d}{k}$, then \eqref{Eq:kSym} demands that $u_1$ is even about the boundary of the substrips and odd about the center lines of the substrips while $u_2$ is even about both the boundary and the center lines of the substrips.

Let $\mcA_d^{k}$ be the set of maps in $u_s + H^1(\Omega_d,\RR^2)$ satisfying the symmetry conditions \eqref{Eq:Imprinting} and \eqref{Eq:kSym}. Note that $\mcA_d^1 = \mcA_d$ and $\mcA_d^k \subset \mcA_d^{\rm x}$ for all $k$ and $\mcA_d^k \subset \mcA_d$ for all odd $k$. Moreover, 
\[
\mcF_d[u] = k\mcF_{d/k}[u_\flat] \text{ with } u_\flat(x,y) := u(x, y + d/k - d).
\]
Hence, $u \in \mcA_d^k$ is a minimizer of $\mcF_d$ in $\mcA_d^k$ if and only if $u_\flat$ is a minimizer of $\mcF_{d/k}$ in $\mcA_{d/k}$. An immediate consequence of this statement and Theorem \ref{Thm1} is the following result.

\begin{theorem}\label{Thm2}
Let $N = 2$ and $k \in \NN^*$. If $0 < d \leq \sqrt{2}\pi k /2$, then the soliton solution $u_s$ is the unique minimizer of $\mcF_d$ in $\mcA_d^k$. If $d >  \sqrt{2}\pi k /2$, then $\mcF_d$ has exactly two minimizers $u_{d,k}^\pm$ in $\mcA_d^k$ given by
\[
u_{d,k}^\pm(x,y) = \begin{cases}
	\displaystyle u_{d/k}^\pm\Big(x, \big\{\frac{(y+d)k}{2d}\big\} \frac{2d}{k} - \frac{d}{k}\Big)
		&\text{ if } \Big\lfloor\frac{(y+d)k}{2d}\Big\rfloor \text{ is even},\\
	\displaystyle  u_{d/k}^\pm\Big(x, -\big\{\frac{(y+d)k}{2d}\big\} \frac{2d}{k} + \frac{d}{k}\Big) 
		&\text{ if } \Big\lfloor\frac{(y+d)k}{2d}\Big\rfloor \text{ is odd},
\end{cases}
\]
where $u_{d/k}^\pm$ are the miminizers of $\mcF_{d/k}$ in $\mcA_{d/k}$ given by Theorem \ref{Thm1} and $\{\cdot\}$ denotes the fractional part of a real number.
\end{theorem}

\begin{proof}[Proof of Remark \ref{Rem:Fdk}]
As mentioned earlier, when $k$ is odd, Remark \ref{Rem:Fdk} follows from the fact that $u_{d,k}^\pm \in \mcA_d$ and $u_d^\pm$ are the unique minimizers of $\mcF_d$ in $\mcA_d$. Assume that $k$ is even. Consider the map $v_{d,k}$ obtained by extending $u_{d/k}^\pm$ such that, when we equally subdivide $\Omega_d$ into $k$ substrips of width $\frac{2d}{k}$, the first component of $v_{d,k}$ is odd about the boundary of the substrips and even about the center lines of the substrips while the second component of $v_{d,k}$ is even about both the boundary and the center lines of the substrips. Then $v_{d,k}$ belongs to $\mcA_d$ and has the same energy as $u_{d,k}^\pm$. Moreover, since 
\[
\partial_y (v_{d,k})_1(x,d) = 
\begin{cases}
\partial_y U_{d/k}(x,0) & \text{ if } k \equiv 0 \mod 4,\\
-\partial_y U_{d/k}(x,0) &\text{ if } k \equiv 2 \mod 4,
\end{cases}
\]
and since $\partial_y U_{d/k}(x,0) < 0$ (which can be seen by applying the Hopf lemma to \eqref{Eq:U}), we have that $\partial_y   (v_{d,k})_1 \neq 0$ at $y = d$, and hence $v_{d,k}$ is different from $u_d^\pm$. We again have $\mcF_d[u_d^\pm] < \mcF_d[v_{d,k}] = \mcF_d[u_{d,k}^\pm]$.
\end{proof}

\subsection{Mountain pass characterization of $2D$ soliton and solitonic vortex}\label{SSec2.3}

In this section, we prove Theorem \ref{Thm:MPC}. We start with the case $0 < d \leq \sqrt{2}\pi/2$.

\begin{lemma}\label{Lem:M1}
Theorem \ref{Thm:MPC} holds for $0 < d \leq \sqrt{2}\pi/2$.
\end{lemma}
\begin{proof}
\underline{Step 1:} We show that $c_d \geq \mcF_d[u_s]$ when $d \leq \sqrt{2}\pi/2$.

Indeed, fix $h \in \Gamma_d$. Note that, by definition, $\psi_n^\pm(0,\cdot) = \pm 1$. Thus, by continuity, there exists $t_0 \in (-1,1)$ such that
\begin{equation}
\int_{-d}^d h(t_0)_1(0,y)\,dy = 0.
	\label{Eq:ZA-1}
\end{equation}
To conclude this step, it suffices to show that $\mcF_d[h(t_0)] \geq \mcF_d[u_s]$. Let $w = h(t_0) - u_s$. As in the proof of Theorem \ref{Thm:S=>M}, it suffices to show that
\[
\int_{\Omega_d} [|\nabla w|^2 - (1 - |u_s|^2)|w|^2]\,dx\,dy \geq 0.
\]
By the phase imprinting condition \eqref{Eq:Imprinting}, we have $w_2(0,\cdot) = 0$. Hence, by \eqref{Eq:usM-1},  
\[
\int_{\Omega_d} [|\nabla w_2|^2 - (1 - |u_s|^2)w_2^2]\,dx\,dy \geq 0.
\]
It remains to prove that
\begin{equation}
\int_{\Omega_d} [|\nabla w_1|^2 - (1 - |u_s|^2)w_1^2]\,dx\,dy \geq 0.
	\label{Eq:ZA-2}
\end{equation}

To prove \eqref{Eq:ZA-2}, we write
\[
w_1(x,y) = a_0(x) + \sum_{k=1}^\infty \Big(a_k(x) \cos\frac{k\pi y}{2d} + b_k(x) \sin\frac{k\pi y}{2d}\Big).
\]
Then \eqref{Eq:ZA-2} becomes
\begin{multline}
2d\int_{-\infty}^\infty [(a_0')^2 - (1 - |u_s|^2)a_0^2]\,dx \\
	+ d\sum_{k=1}^\infty \int_{-\infty}^\infty \Big[(a_k')^2 + (b_k')^2 - \Big(1 - \frac{k^2\pi^2}{4d^2} - |u_s|^2\Big)(a_k^2 + b_k^2)\Big]\,dx \geq 0.
	\label{Eq:ZA-3}
\end{multline}
Observe that by \eqref{Eq:ZA-1}, $a_0(0) = 0$. We may then apply \eqref{Eq:usM-1} to obtain
\[
\int_{-\infty}^\infty [(a_0')^2 - (1 - |u_s|^2)a_0^2]\,dx \geq 0. 
\]
Therefore, to prove \eqref{Eq:ZA-3}, it suffices to show that
\begin{equation}
\int_{-\infty}^\infty \Big[(\zeta')^2    - \Big(1 - \frac{\pi^2}{4d^2} - |u_s|^2\Big)\zeta^2 \Big]\,dx   \geq 0,
	\label{Eq:ZA-4}
\end{equation}
where $\zeta$ stands for either $a_k$ or $b_k$. To this end, we recall the function $A(x) = \mathrm{sech}\frac{x}{\sqrt{2}}$ and write $\zeta = A \eta$. Using
\[
A'' = \Big(\frac{1}{2} - \mathrm{sech}^2\frac{x}{\sqrt{2}}\Big)A,
\]
we compute
\begin{align*}
&\int_{-\infty}^\infty \Big[(\zeta')^2    - \Big(1 - \frac{\pi^2}{4d^2} - |u_s|^2\Big)\zeta^2 \Big]\,dx\\
	&\qquad = \int_{-\infty}^\infty \Big[((A\eta)')^2    + \Big(\frac{\pi^2}{4d^2} - \mathrm{sech}^2\frac{x}{\sqrt{2}}\Big)A^2 \eta^2 \Big]\,dx\\
	&\qquad = \int_{-\infty}^\infty \Big[A^2(\eta')^2 + (\eta^2 A)' A'    + \Big(\frac{\pi^2}{4d^2} - \mathrm{sech}^2\frac{x}{\sqrt{2}}\Big)A^2 \eta^2 \Big]\,dx\\
	&\qquad = \int_{-\infty}^\infty \Big[A^2(\eta')^2  + \Big(\frac{\pi^2}{4d^2} - \frac{1}{2}\Big)A^2 \eta^2 \Big]\,dx \geq 0,
\end{align*}
where we have used $0 < d \leq \sqrt{2}\pi/2$. This proves \eqref{Eq:ZA-4} and concludes Step 1.

\bigskip
\noindent
\underline{Step 2:} We show that $c_d \leq \mcF_d[u_s]$ when $d \leq \sqrt{2}\pi/2$.

To this end, we exhibit a path $h \in \Gamma_d$ such that $ \mcF_d[u_s] = \max_{t \in [-1,1]} \mcF_d[h(t)]$. In our construction, $h(t)$ is independent of $y$. We fix $h(0) = u_s$. From the proof of Remark \ref{Rem:usUnstable}, we can fix some small $t_1 > 0$ such that
\[
\mcF_d[u_s + (tA,0)] < \mcF_d[u_s] \text{ for } t \in [-t_1,t_1] \setminus \{0\}.
\]
We declare $h(t) = u_s + (tA,0)$ for $t \in [-t_1,t_1]$ and proceed to construct $h(t)$ for $|t| > t_1$. We will only consider the case $t > t_1$, as the case $t < -t_1$ is similar. 

Noting that $h(t_1)_1 = t_1 A > 0$ and $|h(t_1)|^2 < \mathrm{sech}^2\frac{x}{\sqrt{2}} + \tanh^2 \frac{x}{\sqrt{2}} = 1$ in $\Omega_d$. Thus, we may write
\[
h(t_1)(x) = \rho(x) (\cos\tilde \chi(x), \sin\tilde \chi(x))
\]
where $\rho$ is even, $0 < \rho < 1$, $\tilde \chi$ is odd, $-\pi/2 < \tilde \chi  < \pi/2$,  and $\tilde \chi(\pm \infty) = \pm \pi/2$. For $t > t_1$, we suppose the ansatz
\[
h(t)(x) = \rho_t(x) (\cos\tilde \chi_t(x), \sin\tilde \chi_t(x)).
\]
We have
\[
\mcF_d[h(t)] = 2d\int_{-\infty}^\infty \Big[\frac{1}{2}(\rho_t')^2 + \frac{1}{2}\rho_t^2 (\tilde \chi_t')^2 + \frac{1}{4} (1 - \rho_t^2)^2\Big]\,dx.
\]
We can now define $h(t)$ for $t \in (t_1, t_2]$ for some $t_2$ slightly large than $t_1$ by keeping $\rho_t = \rho$ while $\chi_t$ is obtained by first stretching $\tilde \chi$ and then deforming it to some $\chi_{\tilde n}$ for some very large $\tilde n > n$, so that the contribution of the term $ \rho_t^2 (\tilde \chi_t')^2$ is always no more than $ \rho^2 (\tilde \chi')^2$ and eventually becomes as negligible as desired. We next define $h(t)$ for $t \in (t_2,t_3]$ for some $t_3$ slightly large than $t_2$ by keeping $\chi_t = \chi_{t_2} = \chi_{\tilde n}$, while letting $\rho_t = \frac{t_3 - t}{t_3 - t_2}\rho + \frac{t - t_2}{t_3 - t_2}$. In this deformation, $(\rho_t')^2$ decreases, $\rho_t$ increases, $(1 - \rho_t^2)^2$ decreases, while the term $\rho_t^2 (\tilde \chi_t')^2$ remains as small as one wishes. At $t = t_3$, we have $\rho_{t_3} = 1$ and $\tilde\chi_{t_3} = \chi_{\tilde n}$. For $t > t_3$, we simply keep $\rho_t = 1$ and deform $\chi_t$ back to $\chi_n$. Along the whole path $h$, we have maintained $\mcF_d[h(t)] < \mcF_d[u_s]$ except for the case $t = 0$. This gives the desired path $h \in \Gamma_d$ and concludes Step 2 and the proof.
\end{proof}

Consider next the case $d > \sqrt{2}\pi/2$ where we treat separately the upper bound and lower bound for $c_d$.

\begin{lemma}\label{Lem:M2}
Let $\Gamma_d$ and $c_d$ be as in Theorem \ref{Thm:MPC}. If $d > \sqrt{2}\pi/2$, then
\[
c_d \leq \mcF_d[u_d^\pm].
\]
\end{lemma}

\begin{proof} Let $u_d = u_d^+$ and fix $\delta > 0$. We need to construct $h \in \Gamma_d$ such that $\mcF_d[h(t)] \leq \mcF_d[u_d] + \delta$ for all $t \in [-1,1]$.

Extend $u_d$ to $\Omega_{3d}$ by evenly reflecting across $y = \pm d$. Slightly abusing notation, we still denote the extension as $u_d$.

We declare that $h(0) = u_d$. We will only give the construction of $h(t)$ for $t > 0$. That for $t < 0$ is similar. For some small $t_1 > 0$ and $t \in (0,t_1]$, let $h(t)(x,y) = u_d(x,y-\frac{t}{t_1}d)$. Then $\mcF_d[h(t)] = \mcF_d[u_d]$ for $t \in [0,t_1]$. Moreover, $h(t_1)$ satisfies
\[
\begin{cases}
|h(t_1)|(x,y) < 1 \text{ in } \bar\Omega_d,\\
h(t_1)_1(x,y) > 0 \text{ in } \bar\Omega_d \setminus \{(0,\pm d)\},\\
h(t_1)_2(x,y) < 0 \text{ for } x > 0, h(t_1)_2(x,y) = 0 \text{ for } x = 0, \text{ and } h(t_1)_2(x,y) > 0 \text{ for } x < 0.
\end{cases}
\]
We next define $h(t)$ for $t \in (t_1, t_2]$ by $h(t) = h(t_1) + (t - t_1)(\sqrt{1 - |h(t_1)|^2},0)$ where $t_2 - t_1 > 0$ is sufficiently small such that $|h(t)| < 1$ in $\bar\Omega_d$ and $\mcF[h(t)] \leq \mcF[h(t_1)] + \delta/2$ for $t \in [t_1, t_2]$. It is readily seen that
\[
\begin{cases}
|h(t_2)|(x,y) < 1 \text{ in } \bar\Omega_d,\\
h(t_1)_1(x,y) > 0 \text{ in } \bar\Omega_d ,\\
h(t_1)_2(x,y) < 0 \text{ for } x > 0, h(t_1)_2(x,y) = 0 \text{ for } x = 0, \text{ and } h(t_1)_2(x,y) > 0 \text{ for } x < 0.
\end{cases}
\]
At this point, we can follow the procedure in Step 2 to deform $h(t_2)$ to $\psi_n^+$: We first deform the phase of $h(t)$ to $\chi_{\tilde n}$ for some large $\tilde n$ while keeping its modulus fixed, then deform the modulus of $h(t)$ to $1$ while keeping the phase of $h(t)$ fixed, and eventually deform the phase of $h(t)$ to $\chi_n$ while keeping the modulus fixed. Here $\tilde n$ is chosen such that along the deformation, the energy is not enlarged more than $\delta/2$. We have thus defined $h \in \Gamma_d$ with $\mcF_d[h(t)] \leq \mcF_d[u_d] + \delta$ for all $t \in [-1,1]$. The proof is completed.
\end{proof}

We turn to lower bounding $c_d$.
\begin{lemma}\label{Lem:M3X}
Let $\Gamma_d$ and $c_d$ be as in Theorem \ref{Thm:MPC}. Then
\[
c_d \geq \frac{1}{16} \text{ for all } d > \sqrt{2}\pi/2.
\]
\end{lemma}

\begin{proof}
Let $h \in \Gamma_d$. Since $\psi_n^\pm(0,\cdot) = \pm 1$, there exists $t_0 \in (-1,1)$ such that 
\[
\int_{-d}^d h(t_0)_1(x,y)\,dy = 0.
\]
To conclude, it suffices to show that $\mcF_d[h(t_0)] \geq \frac{1}{16}$. 

For simplicity of notation, we write $u = h(t_0)$. Since $\mcF_d[u] \geq \frac{1}{2}\|\nabla u\|_{L^2(\Omega_d)}^2$, we may assume without loss of generality that $\|\nabla u\|_{L^2(\Omega_d)}^2 < \frac{1}{8}$. We claim that
\begin{equation}
\int_{-d}^d |u(0,y)|^2\,dy \leq  \frac{3d}{8}.
	\label{Eq:uCenter}
\end{equation}
Once this claim is proved, we then have that 
\[
|\{y: |u(0,y)| > \frac{1}{2}\}| \leq \frac{3d}{2}  \quad \text{ and } \quad |\{y: |u(0,y)| \leq \frac{1}{2}\}| \geq \frac{d}{2},
\]
and so
\begin{align*}
\mcF_d[u] 
	&\geq 2\int_{-d}^d \int_0^\infty \Big[\frac{1}{2} (\partial_x |u|)^2 + \frac{1}{4} (1 - |u|^2)^2\Big]\,dx\,dy 
	\\
	&\geq \sqrt{2} \int_{\{y: |u(0,y)| \leq \frac{1}{2}\}} \int_0^\infty (1 - |u|^2)\partial_x |u|\,dx\,dy \\
	&\geq \sqrt{2} \frac{d}{2} \Big[s - \frac{s^3}{3}\Big]_{1/2}^1 = \frac{5d\sqrt{2}}{48} > \frac{5\pi}{48} > \frac{1}{16},
\end{align*}
which gives the desired conclusion.

To prove the claim \eqref{Eq:uCenter}, observe first that
\[
\int_{-d}^d |u(x,y) - u(0,y)|^2\,dy = \int_{-d}^d \Big|\int_0^x \partial_x u(s,y)\,ds\Big|^2\,dy 
\leq |x| \int_0^x \int_{-d}^d  |\partial_x u(s,y)|^2\,dy\,ds,
\]
which implies
\begin{equation}
\int_0^d \int_{-d}^d |u(x,y) - u(0,y)|^2\,dy \leq \frac{d^2}{2} \int_0^d \int_{-d}^d  |\nabla u(x,y)|^2\,dy\,dx \leq \frac{d^2}{16}.
	\label{Eq:uC-1}
\end{equation}
Next, if we let
\[
\bar u(x) = \frac{1}{2d} \int_{-d}^d u(x,y)\,dy,
\]
then $\bar u(0) = 0$ and
\[
|\bar u(x)|^2 = |\bar u(x) - \bar u(0)|^2 = \frac{1}{4d^2} \Big(\int_0^x \int_{-d}^d \partial_x u(s,y)\,dy\,ds\Big)^2 \leq \frac{|x|}{2d} \int_0^x \int_{-d}^d |\partial_x u(s,y)|^2\,dy\,ds.
\]
Also, by Wirtinger's inequality, we have
\[
\int_{-d}^d |u(x,y) - \bar u(x)|^2\,dy \leq \frac{4d^2}{\pi^2} \int_{-d}^d |\partial_y u(x,y)|^2\,dy.
\]
Combining the last two inequalities, we get
\[
\int_{-d}^d |u(x,y)|^2\,dy \leq \frac{8d^2}{\pi^2} \int_{-d}^d |\partial_y u(x,y)|^2\,dy + 2|x| \int_0^x \int_{-d}^d |\partial_x u(s,y)|^2\,dy\,ds.
\]
and hence
\begin{equation}
\int_0^d \int_{-d}^d |u(x,y)|^2\,dy \leq d^2 \int_0^d \int_{-d}^d |\nabla u(x,y)|^2\,dy\,dx 
\leq \frac{d^2}{8}. 
	\label{Eq:uC-2}
\end{equation}
The claim \eqref{Eq:uCenter} is readily seen from \eqref{Eq:uC-1} and \eqref{Eq:uC-2}.
\end{proof}

To prove that $c_d \geq \mcF_d[u_d^\pm]$ for $d > \sqrt{2}\pi/2$, we need the following result.

\begin{lemma}\label{Lem:EC}
If $d > \sqrt{2}\pi/2$, then
\[
\int_{\Omega_d} [|\nabla \psi|^2 - (1 - U_d^2 - V_d^2) \psi^2]\,dx\,dy \geq 0
\]
for all $\psi \in H^1(\Omega_d)$ satisfying $\int_{\Omega_d} (1 - U_d^2 - V_d^2) \psi\,dx\,dy = 0$.
\end{lemma}

Let us postpone momentarily the proof of Lemma \ref{Lem:EC}, and instead proceed to see how it implies the assertion that $c_d \geq \mcF_d[(U_d,V_d)]$ for $d > \sqrt{2}\pi/2$.

\begin{lemma}\label{Lem:M3}
Let $\Gamma_d$ and $c_d$ be as in Theorem \ref{Thm:MPC}. If $d > \sqrt{2}\pi/2$, then
\[
c_d \geq \mcF_d[u_d^\pm].
\]
\end{lemma}

\begin{proof} Let $d >  \sqrt{2}\pi/2$ and let $h \in \Gamma_d$. We need to show that $\mcF[h(t)] \geq \mcF_d[(U_d,V_d)]$ for some $t \in [-1,1]$. We will abbreviate $(U,V) = (U_d, V_d)$.

Since $\pm\psi^\pm_n > 0$ in $\Omega_d$, there exists $t_0 \in (-1,1)$ such that 
\begin{equation}
\int_{\Omega_d} (1 - U^2 - V^2) h(t_0)_1 \,dx\,dy = 0.
	\label{Eq:O-1}
\end{equation}
We will show that $\mcF_d[h(t_0)] \geq \mcF_d[(U,V)]$. To this end, we write $w = h(t_0) - (U,V)$ and proceed to show as in Step 2 of the proof of Theorem \ref{Thm:UV} that
\[
\int_{\Omega_d} \big[|\nabla w|^2 - (1 - U^2 - V^2) w^2\big]\,dx\,dy \geq 0.
\]
Indeed, by \eqref{Eq:O-1} together with the oddness of $U$ with respect to $y$,
\[
\int_{\Omega_d} (1 - U^2 - V^2) w_1 \,dx\,dy = 0. 
\]
Hence, by Lemma \ref{Lem:EC}, 
\[
\int_{\Omega_d} \big[|\nabla w_1|^2 - (1 - U^2 - V^2) w_1^2\big]\,dx\,dy \geq 0.
\]
On the other hand, by the phase imprinting condition \eqref{Eq:Imprinting}, $w_2 = 0$ on $\{x = 0\}$. We may thus use the argument in Step 2 of the proof of Theorem \ref{Thm:UV} to obtain
\[
\int_{\Omega_d} \big[|\nabla w_2|^2 - (1 - U^2 - V^2) w_2^2\big]\,dx\,dy \geq 0.
\]
Putting things together, we conclude the proof.
\end{proof}

\begin{proof}[Proof of Theorem \ref{Thm:MPC}]
The result follows from Lemmas \ref{Lem:M1}, \ref{Lem:M2}, \ref{Lem:M3X} and \ref{Lem:M3}.
\end{proof}

It remains to prove Lemma \ref{Lem:EC}. Let $X_d$ and $Y_d$ denote the completions of the space of bounded smooth functions on $\bar\Omega_d$ with respect to the norms
\begin{align*}
\|\psi\|_{X_d}^2 &= \int_{\Omega_d} [|\nabla \psi|^2 + (1 - U_d^2 - V_d^2)\psi^2]\,dx\,dy,\\
\|\psi\|_{Y_d}^2 &= \int_{\Omega_d}  (1 - U_d^2 - V_d^2)\psi^2 \,dx\,dy.
\end{align*}
At points in the discussion to follow, we will consider the limit $d \searrow \sqrt{2}\pi/2$. In view of the compactness estimate \eqref{Eq:UD}, as $d \searrow \sqrt{2}\pi/2$ and up to extraction of subsequence, the solitonic vortex solution $(U_d,V_d)$ converges locally smoothly. Since the solitonic vortex solution is a minimizer of $\mcF_{d}$ in $\mcA_{d}$ (by Theorem \ref{Thm1}), any such limit is a minimizer of $\mcF_{\sqrt{2}\pi/2}$ in $\mcA_{\sqrt{2}\pi/2}$ which is the soliton solution $u_s$ (again by Theorem \ref{Thm1}). Thus, as the full sequence of $(U_d,V_d)$ converges locally smoothly as $d \searrow \sqrt{2}\pi/2$ to $u_s$. In view of this fact and abusing of notation, we will write $(U_{\sqrt{2}\pi/2}, V_{\sqrt{2}\pi/2}) = u_s$.

Claim: $X_d$ embeds compactly into $Y_d$. Indeed, consider a bounded sequence $\{\psi_m\} \subset X_d$. After passing to a subsequence, we may assume that that $\psi_m$ converges strongly in $L^2_{\rm loc}(\Omega_d)$ to some $\psi \in X_d$. We need to show that, for any given $\vareps > 0$, it holds that $\|\psi_m - \psi\|_{Y_d}^2 \leq \vareps$ for all large $m$. Note that, if $\eta \in C^\infty(\RR)$ such that $\eta = 0$ in $(-1,1)$ and $\eta = 1$ in $(-\infty,-2) \cup (2,\infty)$, then, by Hardy's inequality,
\[
\int_\Omega \frac{1}{x^2} \eta^2 \phi_m^2\,dx\,dy \leq C \int_{\Omega_d} (\partial_x (\eta\phi_m))^2\,dx\,dy \leq C\|\phi_m\|_{X_d}^2
\]
Hence $\{(1 + x^2)^{-1/2} \psi_m\}$ is bounded in $L^2(\Omega_d)$. Now, for any $a > 0$,  we split
\begin{multline*}
\|\psi_m - \psi\|_{Y_d}^2
	= \int_{\{|x| \leq a, |y| \leq d\}} |\psi_m - \psi|^2\,(1 - U_d^2 - V_d^2) dxdy\\
		 + \int_{\{|x| > a, |y| \leq d\}}|\psi_m - \psi|^2\,(1 - U_d^2 - V_d^2) dxdy.
\end{multline*}
Using the fact that $1 - U_d^2 - V_d^2$ decays faster than any polynomials (see Remark \ref{Rem:PD}) and that $\{(1 + x^2)^{-1/2} \psi_m\}$ is bounded in $L^2(\Omega_d)$, we may choose $a$ sufficiently large so that the second integral on the right hand side is no more than $\vareps/2$. Having fixed $a$, by the convergence of $\psi_m$ to $\psi$ in $L^2_{\rm loc}(\Omega_d)$, we can wait until $m$ is large enough so that the first integral on the right hand side is also no more than $\vareps/2$. We then have $\|\psi_m - \psi\|_{Y_d}^2 \leq \vareps$ for all large $m$. The claim is proved.

In view of the compactness embedding $X_d \hookrightarrow Y_d$, it is standard to show that the Neumann eigenvalue problem
\begin{equation}
\begin{cases}
-\Delta \phi = \mu (1 - U_d^2 - V_d^2) \phi & \text{ in } \Omega_d,\\
\partial_\nu \phi = 0 &\text{ on } \partial\Omega_d
\end{cases}
	\label{Eq:muEVP}
\end{equation}
admits a sequence of eigenvalue-eigenfunction pairs $\{\mu_{d,k},\phi_{d,k}\}$ such that
\[
\mu_{d,1} = 0 < \mu_{d,2} \leq \mu_{d,3} \leq \cdots  \nearrow \infty,
\]
$\|\phi_{d,k}\|_{Y_d} = 1$, $\phi_{d,1} > 0$ is a constant, 
\begin{multline*}
\mu_{d,k+1} 
	=  \inf \Big\{\int_{\Omega_d} |\nabla \psi|^2\,dx\,dy: \psi \in X_d, \|\psi\|_{Y_d} = 1, \\\int_{\Omega_d} (1 - U_d^2 - V_d^2) \psi \psi_{d,j}\,dx\,dy = 0 \text{ for } j = 1, \ldots, k\Big\},
\end{multline*}
and $\psi_{d,k+1}$ is chosen as a minimizer of this minimization problem.

We note the following stability property for the eigenvalue problem \eqref{Eq:muEVP}: 
\begin{equation}
\begin{minipage}{.8\textwidth}
If $d_m \rightarrow d_\infty \in [\sqrt{2}\pi/2,\infty)$ and $(\mu_{d_m},\phi_{d_m})$ is a sequence of eigenvalue-eigenfunction pair for \eqref{Eq:muEVP} with $d = d_m$ such that $\mu_{d_m} \rightarrow \mu_\infty \in [0,\infty)$ and $\|\phi_{d_m}\|_{Y_{d_m}} = 1$, then up to extracting a subsequence, $\phi_{d_m}$ converges to a solution to the eigenvalue problem \eqref{Eq:muEVP} with $d = d_\infty$ and $\mu = \mu_\infty$.
\end{minipage}
\label{Eq:EVPStab}
\end{equation}
Indeed, using $\|\nabla \phi_m\|_{L^2(\Omega_{d_m})}^2 = \mu_{d_m}$ and applying elliptic estimates to the equation obtained by differentiating \eqref{Eq:muEVP}, we have $|\nabla \phi_m| \leq C$ in $\Omega_{d_m}$ for some $C$ independent of $m$. On the other hand, since $(U_{d_m}, V_{d_m}) \rightarrow (U_{d_\infty}, V_{d_\infty})$ locally uniformly and $1 - U_{d_\infty}^2 - V_{d_\infty}^2 > 0$ in $\Omega_{d_\infty}$ (by Theorem \ref{Thm1}), we deduce form the normalization $\|\phi_{d_m}\|_{Y_{d_m}} = 1$ that the average of $\phi_{d_m}$ in a disk around the origin of some small radius, say $1 < d_\infty$, is bounded by a constant independent of $m$. It follows that $|\phi_{d_m}| \leq C + C|x|$. By Ascoli-Arzel\`a's theorem and after passing to a subsequence, $\phi_{d_m}$ converges locally uniformly to some limit $\phi_\infty$. By Lebesgue's dominated convergence theorem and the uniform super-polynomial decay estimate \eqref{Eq:UD}, we see that $\|\phi_\infty\|_{Y_{d_\infty}} = 1$. The fact that $(\mu_\infty, \phi_\infty)$ is a non-trivial solution to the eigenvalue problem \eqref{Eq:muEVP} with $d = d_\infty$ also follows.

It is clear that $U_d$ and $V_d$ are solutions of the eigenvalue problem \eqref{Eq:muEVP} corresponding to the eigenvalue $\mu = 1$. Lemma \ref{Lem:EC} follows from the following stronger statement.

\begin{lemma}
\label{Lem:ECS}
For $d \geq \sqrt{2}\pi/2$, the second eigenvalue of \eqref{Eq:muEVP} is $1$ and has multiplicity $2$ (that is, $\mu_{d,2} = \mu_{d,3} = 1 < \mu_{d,4}$).
\end{lemma}

Let us treat first the case $d$ is close to $\sqrt{2}\pi/2$.
\begin{lemma}\label{Lem:ECS-Low}
There exists a maximal $d_* \in (\sqrt{2}\pi/2, \infty]$ such that Lemma \ref{Lem:ECS} holds for $d \in [\sqrt{2}\pi/2,d_*)$.
\end{lemma}

\begin{remark}
For $\sqrt{2}\pi/2 < d < d_*$, the eigenspace corresponding to the eigenfunction $\mu_{d,2} = 1$ is generated by $U_d$ and $V_d$. When $d = \sqrt{2}\pi/2$, $U_d = 0$ and we cannot use it as a basis element. We will see from the proof below that, in this case, this eigenspace is generated by $\mathrm{sech}\frac{x}{\sqrt{2}}  \sin \frac{y}{\sqrt{2}}$ and $V_d = \tanh \frac{x}{\sqrt{2}}$.
\end{remark}

\begin{proof} 
Consider first the case $d = \sqrt{2}\pi/2$. In this case, the eigenvalue problem \eqref{Eq:muEVP} becomes
\begin{equation}
\begin{cases}
-\Delta \phi = \mu\mathrm{sech}^2 \frac{x}{\sqrt{2}}  \phi & \text{ in } \Omega_{\sqrt{2}\pi/2},\\
\partial_y \phi = 0 & \text{ on } \partial\Omega_{\sqrt{2}\pi/2}.
\end{cases}
	\label{Eq:mu*EVP}
\end{equation}

Since $\partial_y \phi = 0$ on $\{y = \pm \sqrt{2}\pi/2\}$, we can write
\[
\phi(x,y) = a_0(x) + \sum_{k=1}^\infty a_k(x) \cos (k\sqrt{2} y) + \sum_{k=1}^\infty b_k(x) \sin(\frac{(2k+1)y}{\sqrt{2}}).
\]
The Fourier coefficients $a_k$ and $b_k$'s satisfy
\begin{equation}
\begin{cases}
-a_k'' + 2k^2 a_k = \mu\mathrm{sech}^2 \frac{x}{\sqrt{2}} a_k,\\
-b_k'' + \frac{1}{2}(2k+1)^2 b_k = \mu\mathrm{sech}^2 \frac{x}{\sqrt{2}} b_k.
\end{cases}
	\label{Eq:mu*EVPk}
\end{equation}
Introducing the change of variable $\theta = \arccos \tanh \frac{x}{\sqrt{2}}$, the equation \eqref{Eq:mu*EVPk} becomes
\begin{equation}
\begin{cases}
- \frac{1}{\sin\theta}\frac{d}{d\theta}\big(\sin\theta \frac{da_k}{d\theta} \big) + \frac{4k^2}{\sin^2\theta} a_k = 2\mu a_k ,\\
- \frac{1}{\sin\theta}\frac{d}{d\theta}\big(\sin\theta \frac{db_k}{d\theta} \big)  + \frac{(2k+1)^2}{\sin^2\theta} b_k= 2\mu b_k \text{ for } \theta \in (0,\pi).
\end{cases}
	\label{Eq:akde}
\end{equation}
Note that the fact that $\phi \in X_{\sqrt{2}\pi/2}$ implies
\[
\int_0^\pi \Big[\Big(\frac{d a_k}{d\theta}\Big)^2 + \frac{4k^2}{\sin^2\theta} a_k^2\Big]\,\sin\theta d\theta < \infty \text{ and }\int_0^\pi \Big[\Big(\frac{d b_k}{d\theta}\Big)^2 + \frac{(2k+1)^2}{\sin^2\theta} a_k^2\Big]\,\sin\theta d\theta < \infty.
\]
Therefore, introducing an artificial variable $\varphi \in [0,2\pi)$, we see that the functions $a_k(\theta) \cos(2k\varphi)$ and $b_k(\theta) \cos((2k+1)\varphi)$ are spherical harmonics of degree $2k$ and $2k+1$, respectively. It readily follows that the eigenvalue $\mu$ takes the form $\frac{1}{2}\ell(\ell+1)$ for $\ell = 0, 1, \dots$. We deduce that, when $d = \sqrt{2}\pi/2$, $\mu_{d,1} = 0$ is simple and the corresponding eigenspace consists of the constant functions, $\mu_{d,2} = 1$ has multiplicity $2$ (hence $\mu_{d,2} = \mu_{d,3}$) and the corresponding eigenspace is generated by $\cos\theta = \tanh \frac{x}{\sqrt{2}}$ and $\sin\theta \sin \frac{y}{\sqrt{2}}= \mathrm{sech}\,\frac{x}{\sqrt{2}}  \sin \frac{y}{\sqrt{2}}$, and $\mu_{d,4} = 3$. This proves that Lemma \ref{Lem:ECS} holds for $d = \sqrt{2}\pi/2$.

Suppose by contradiction that $d_*$ does not exist. Then there exists a sequence $d_m \searrow \sqrt{2}\pi/2$ such that one of the following cases occur: 
\begin{itemize}
\item $\mu_{d_m,2} \neq 1$, or
\item $\mu_{d_m,2} = 1$ but its multiplicity is not $2$. 
\end{itemize}

Consider the first case, that is $\mu_{d_m,2} \neq 1$. Since $1$ is an eigenvalue with eigenfunctions $U_d$ and $V_d$, we have that $\mu_{d_m,2} < 1$ and there exists a minimal $j_m \geq 3$ such that $\mu_{d_m,j_m} = \mu_{d_m,j_m+1} = 1$. Passing to a subsequence, we may assume $\mu_{d_m,2} \rightarrow \mu_\infty \in [0,1]$. By the stability property \eqref{Eq:EVPStab} and since $\mu_{\sqrt{2}\pi/2,1} = 0$ and $\mu_{\sqrt{2}\pi/2,2}  = 1$, we have $\mu_\infty \in \{0,1\}$. Since $\mu_{d,1}$ is simple with constant eigenfunctions, and the eigenfunction corresponding to $\mu_{d_m,2}$ is $Y_{d_m}$-orthogonal to the constants, $\mu_\infty \neq 0$. Hence $\mu_\infty = 1$. This implies that the eigenspace corresponding to $\mu_{\sqrt{2}\pi/2,2} = 1$ contains at least three pair-wise orthogonal eigenfunctions: $V_{\sqrt{2}\pi/2}$, the limit of $\frac{U_{d_m}}{\|U_{d_m}\|_{Y_{d_m}}}$ and the limit of the normalized eigenfunction corresponding to $\mu_{d_m,2}$. This contradicts the fact that this eigenspace has two dimensions.

Consider the second case, that is $\mu_{d_m,2} = 1$ but its multiplicity is not $2$. Since $U_{d_m}, V_{d_m}$ are eigenfunctions corresponding to this eigenfunction, the multiplicity must be at least three. By the stability property \eqref{Eq:EVPStab}, this implies again that the eigenspace corresponding to $\mu_{\sqrt{2}\pi/2,1} = 2$ has at least three dimensions, which is not possible. The proof is complete.
\end{proof}

The next lemma examines the structure of the second eigenspace for $d = d_*$ if $d_*$ is finite.

\begin{lemma}\label{Lem:ECS-d*}
Let $d_*$ be given by Lemma \ref{Lem:ECS-Low}. If $d_*$ is finite, then the second eigenvalue of \eqref{Eq:muEVP} with $d = d_*$ is $\mu_{d_*,2} = 1$ and has multiplicity $3$. Moreover, the corresponding eigenspace is spanned by $U_{d_*}, V_{d_*}$ and another function denoted by $P_{d_*}$ which is even in $x$ and whose zero set 
\[
N(P_{d_*}) = \{(x,y) \in \overline\Omega_{d_*}: P_{d_*}(x,y) = 0\}
\] satisfies simultaneously the following properties:
\begin{enumerate}[(i)]
\item $N(P_{d_*})$ is a closed simple Lipschitz curve meeting $\partial\Omega_d$ at exactly $(0,d_*)$.
\item $N(P_{d_*})$ is smooth away from $(0,d_*)$.
\item Locally near $(0,d_*)$, $N(P_{d_*})$ consists of two curves which are orthogonal at $(0,d_*)$ and each of which forms a $\frac{\pi}{4}$-angle with $\{y = d_*\}$.
\item $N(P_{d_*}) \cap \{0 < y < d_*\}$ consists of two simple curves which are symmetric about the $y$-axis and which split $\{0 < y < d_*\}$ into three components.
\end{enumerate}
\end{lemma}

\begin{proof} We will make use of some arguments appearing previously in \cite{Cheng76, CSLin87-CMP, Liqun95} in the context of eigenvalues of the Laplacian.

By the stability property \eqref{Eq:EVPStab} and the fact that $\mu_{d,2} = \mu_{d,3} = 1$ for $d < d_*$ (by Lemma \ref{Lem:ECS-Low}, we have that $\mu_{d_*,2} = \mu_{d_*,3} = 1$. If $\mu_{d*,4} > 1$, we may argue as in the proof of Lemma \ref{Lem:ECS-Low} to find $\hat d_{*} > d_*$, so that $\mu_{d,2} = \mu_{d,3} = 1 < \mu_{d,4}$ for $d_* \leq d < \hat d_{*}$, contradicting the maximality of $d_*$. Thus, $\mu_{d_*,4} = 1$.

Let $Z$ denote the eigenspace corresponding to $\mu_{d*,2} = 1$. The above shows that $\dim Z \geq 3$. Since $U_{d_*}, V_{d_*} \in Z$, we may find another $\phi \in Z\setminus \{0\}$ which is independent from $U_{d_*}, V_{d_*}$. We claim that we may assume that $\phi$ is even in both $x$ and $y$. To see this, suppose for example that $\phi$ is not even in $x$. In this case, $\phi$ can be split as $\phi_1 + \phi_2$ where $\phi_1$ is even in $x$ and $\phi_2 \neq 0$ is odd in $x$. Since the weight $1 - U_{d_*}^2 - V_{d_*}^2$ is even in $x$, we see that $\phi_2$ also belongs to $Z$. Noting that $\phi_2 = 0$ on $\{x = 0\}$, we may use \eqref{Eq:UV-1b} to get
\[
0 = \int_{\Omega_{d_*}} \big [|\nabla \phi_2|^2 - (1 - U_{d_*}^2 - V_{d_*}^2) \phi_2^2\big]\,dx\,dy \geq \int_{\Omega_{d_*}} V_{d_*}^2 \big|\nabla\big(\frac{\phi_2}{V_{d_*}}\big)\big|^2\,dx\,dy.
\]
This implies that $\frac{\phi_2}{V_{d_*}}$ is constant, that is $\phi_2$ is a multiple of $V_{d_*}$. Since $\phi$ and $V_{d_*}$ are linear independent, we deduce that $\phi_1 \neq 0$. Therefore, we can assume from the beginning that $\phi$ is even in $x$. Similarly, using \eqref{Eq:UV-1a}, we may also assume that $\phi$ is even in $y$. The claim is proved.

We will use a few facts on the zero set of functions in $Z$.

\medskip
\noindent\underline{Fact 1:} The number of nodal domains of any non-zero function $\psi \in Z$ is $2$. This follows from Courant's nodal domain theorem.

\medskip
\noindent\underline{Fact 2:} If $p \in \Omega_d$ is an interior zero of a non-zero function $\psi \in Z$, then $\nabla \psi(p) \neq 0$. Indeed, by \eqref{Eq:muEVP} and unique continuation, the leading non-trivial homogeneous Taylor polynomial of $\psi$ at $p$ is a non-zero harmonic polynomial of some degree $k \geq 1$. Thus, locally around $p$, the zero set of $\psi$ consists or $k$ transversal smooth curves (which meet at $p$ equiangularly). Topological consideration shows that the number of nodal domains of $\psi$ will be more than $2$ if $k \geq 2$. Hence $k = 1$ and so $\nabla \psi(p) \neq 0$.

\medskip
\noindent\underline{Fact 3:} If $p \in \partial\Omega_d$ is a boundary zero of a non-zero function $\psi \in Z$, then exactly one of the two following situations occurs:
\begin{enumerate}[(a)]
\item $\nabla \psi(p) \neq 0$ and, locally near $p$, the zero set of $\psi$ is a smooth curve meeting $\partial \Omega_d$ at $p$ orthogonally.
\item $\nabla \psi(p) = 0$, $\nabla^2 \psi(p) \neq 0$, and, locally near $p$, the zero set of $\psi$ consists of two orthogonal curves each of which forms a $\pi/4$ angle with $\partial\Omega_d$ at $p$.
\end{enumerate}
The proof of this is similar to that of Fact 2, although in this case the degree $k$ of the leading homogeneous Taylor polynomial is either $1$ or $2$, leading to the two indicated possibilities. 

Let us now see that $\dim Z = 3$. Indeed, if this is not true, then we would have $\dim Z \geq 4$ and hence, for any fixed $p \in \Omega_{d_*}$, we may select $\psi \in Z$ such that $\psi(p) = \partial_x \psi(p) = \partial_y\psi(p) = 0$, which contradicts Fact 2.

It remains to show the existence of $P_{d_*}$. Before doing this, let us show that the zero set $N(\phi)$ of $\phi$ is a closed simple curve in $\Omega_{d_*}$. We start by showing that $N(\phi)$ contains no boundary point. Arguing by contradiction, assume that  $N(\phi)$ contains a boundary point. In this case, $N(\phi)$ contains a point on each component of $\partial\Omega_{d_*}$ due to the evenness of $\phi$ with respect to $y$. By Fact 1, at each such boundary point, situation (a) in Fact 3 must hold. Since $N(\phi)$ is symmetric with respect to $x$, Fact 1 again implies that $N(\phi)$ is exactly $\{x = 0\} \cap \Omega_d$. In particular, $\phi$ satisfies in addition to \eqref{Eq:muEVP} the identity
\[
\phi = \partial_x \phi = 0 \text{ on } \{x = 0\} \cap \Omega_d.
\]
By unique continuation, $\phi \equiv 0$, contradicting the choice of $\phi$. Now, by Facts 1 and 2 and since $N(\phi)$ contains no boundary point, $N(\phi)$ is a single closed simple curve or an open curve contained in $\Omega_d$. By the symmetry of $N(\phi)$ with respect to $y$, $N(\phi)$ must intersect the $x$-axis, as otherwise it would have at least two components which are mirror image of one another about the $x$-axis. Let $(x_0,0)$ be a point of intersection of $N(\phi)$ and the $x$-axis. By Fact 2, $\nabla \phi(x_0,0) \neq 0$, and so $x_0 \neq 0$. By the symmetry of $N(\phi)$ with respect to $x$, $N(\phi)$ also contains $(-x_0,0)$. Let $\Gamma_0$ denote an arc connecting $(x_0,0)$ to $(-x_0,0)$ along $N(\phi)$. By the symmetry of $N(\phi)$ with respect to $x$ and its regularity, $N(\phi)$ is then the union of $\Gamma_0$ and its mirror image about the $x$-axis. We conclude that $N(\phi)$ is a closed simple curve in $\Omega_d$.

Since $N(\phi)$ is a closed curve in $\Omega_{d_*}$, we have $\phi(0,d_*) \neq 0$. Let
\[
P_{d_*} = U_{d_*} - \frac{U_{d_*}(0,d_*)}{\phi(0,d_*)} \phi(0,d_*).
\]
It is clear that $P_{d_*}(0,d_*) = 0$. Since $\partial_y P_{d_*} = 0$ on $\partial \Omega_{d_*}$ and $\partial_x P_{d_*} = 0$ on $\{x = 0\}$, we have $\nabla P_{d_*}(0,d_*) = 0$. Hence situation (b) in Fact 3 holds, that is conclusion (iii) holds. This together with Fact 1 implies conclusion (i). In view of Fact 2, conclusion (ii) holds. Next, since both $U_{d_*}$ and $\phi$ vanishes at $(\pm x_0,0)$, we have that $(\pm x_0,0) \in N(P_{d_*})$. Note that as the curve $N(P_{d_*})$ leaves $(0,d_*)$ from the right half-plane, it must first connect to $(x_0,0)$ before reaching $(-x_0,0)$ (since otherwise, before reaching $(-x_0,0)$, $N(P_{d_*})$ would cross the $y$-axis perpendicularly at some point $(0,y_0)$, which, in view of the symmetry about the $y$-axis and conclusion (ii), would imply that $N(P_{d_*})$ would be the union of the arc from $(0,d_*)$ to $(0,y_0)$ on the right half-plane with its mirror image on the left half-plane without ever reaching $(\pm x_0,0)$). By the same reasoning, before reaching $(x_0,0)$, $N(P_{d_*})$ cannot meet the $y$-axis again. Conclusion (iv) follows.
\end{proof}

The goal now is to show that the very existence of $P_{d_*}$ in Lemma \ref{Lem:ECS-d*} will lead to a contradiction. For this, we detour through a related eigenvalue problem. (Our consideration of a different eigenvalue problem is inspired by a line of argument in \cite{ChenGuiYao, JM20-AnnM, JM22-Err-AnnM}, although the details of our proof are independent from these works.) Let $\tilde \Omega_d = \RR \times (0,2d)$, and extend $U_d$ and $V_d$ evenly across $y = d$. Let $\tilde X_d$ and $\tilde Y_d$ denote the completions of the space of bounded smooth functions on $\overline{\tilde \Omega}_d$ with respect to the norms
\begin{align*}
\|\tilde\psi\|_{\tilde X_d}^2 &= \int_{\tilde\Omega_d} [|\nabla \tilde \psi|^2 + (1 - U_d^2 - V_d^2)\tilde\psi^2]\,dx\,dy,\\
\|\tilde\psi\|_{\tilde Y_d}^2 &= \int_{\tilde \Omega_d}  (1 - U_d^2 - V_d^2)\tilde\psi^2 \,dx\,dy.
\end{align*}
 As before, the eigenvalue problem
\begin{equation}
\begin{cases}
-\Delta \tilde\phi = \tilde\mu (1 - U_d^2 - V_d^2) \tilde\phi & \text{ in } \tilde\Omega_d,\\
\partial_\nu \tilde\phi = 0 &\text{ on } \partial\tilde\Omega_d
\end{cases}
	\label{Eq:tmuEVP}
\end{equation}
admits a sequence of eigenvalue-eigenfunction pairs $\{\tilde\mu_{d,k},\tilde\phi_{d,k}\}$ such that
\[
\tilde\mu_{d,1} = 0 < \tilde\mu_{d,2} \leq \tilde\mu_{d,3} \leq \cdots  \nearrow \infty,
\]
$\|\tilde \phi_{d,k}\|_{Y_d} = 1$, $\tilde\phi_{d,1} > 0$ is a constant, 
\begin{multline*}
\tilde\mu_{d,k+1} 
	=  \inf \Big\{\int_{\Omega_d} |\nabla \tilde\psi|^2\,dx\,dy: \tilde\psi \in \tilde X_d, \|\tilde\psi\|_{\tilde Y_d} = 1, \\\int_{\tilde\Omega_d} (1 - U_d^2 - V_d^2) \tilde\psi \tilde\psi_{d,j}\,dx\,dy = 0 \text{ for } j = 1, \ldots, k\Big\},
\end{multline*}
and $\tilde\psi_{d,k+1}$ is chosen as a minimizer of this minimization problem. Like \eqref{Eq:muEVP}, the eigenvalue problem \eqref{Eq:tmuEVP} also has similar a stability property:
\begin{equation}
\begin{minipage}{.8\textwidth}
If $d_m \rightarrow d_\infty \in [\sqrt{2}\pi/2,\infty)$ and $(\tilde\mu_{d_m},\tilde\phi_{d_m})$ is a sequence of eigenvalue-eigenfunction pair for \eqref{Eq:tmuEVP} with $d = d_m$ such that $\tilde\mu_{d_m} \rightarrow \tilde\mu_\infty \in [0,\infty)$ and $\|\tilde\phi_{d_m}\|_{\tilde Y_{d_m}} = 1$, then up to extracting a subsequence, $\tilde\phi_{d_m}$ converges to a solution to the eigenvalue problem \eqref{Eq:muEVP} with $d = d_\infty$ and $\tilde\mu = \tilde\mu_\infty$.
\end{minipage}
\label{Eq:tEVPStab}
\end{equation}

It is clear that, when $d = \sqrt{2}\pi/2$, the eigenvalues to \eqref{Eq:tmuEVP} are identical to those to \eqref{Eq:muEVP}. In particular, by Lemma \ref{Lem:ECS-Low}, $\tilde\mu_{\sqrt{2}\pi/2, 1} = 0 < \tilde \mu_{\sqrt{2}\pi/2,2} = \tilde \mu_{\sqrt{2}\pi/2, 3} = 1 < \tilde \mu_{\sqrt{2}\pi/2,4}$.

For $d > \sqrt{2}\pi/2$, it is clear that $V_d$ is an eigenfunction to \eqref{Eq:tmuEVP} corresponding to eigenvalue $\tilde \mu = 1$. The following lemma asserts that $\tilde \mu = 1$ cannot be the second eigenvalue for $d > \sqrt{2}\pi/2$.

\begin{lemma}\label{Lem:Sib2nd}
For $d > \sqrt{2}\pi/2$, it holds that $\tilde \mu_{d,2} < 1$.
\end{lemma}

\begin{proof}
We only need to show that there exists a non-trivial solution to \eqref{Eq:tmuEVP} with $\tilde \mu < 1$. Consider the minimization problem
\begin{equation}
\lambda_{d,1} := \inf\Big\{\int_{\tilde\Omega_{d/2}} |\nabla \psi|^2\,dx\,dy: \psi \in \tilde X_{d/2}, \|\psi\|_{\tilde Y_{d/2}} = 1, \psi = 0 \text{ on } \{y = d\}\Big\}.
	\label{Eq:S2-0}
\end{equation}
Since $\tilde X_{d/2} \hookrightarrow \tilde Y_{d/2}$ compactly, there exists a solution to this minimization problem. Denote such solution as $\psi_d$. Extend $\psi_d$ across $y = d$ as an odd function. Then $\psi_d$ solves \eqref{Eq:tmuEVP} with $\mu = \lambda_{d,1}$. Therefore, we only need to prove that $\lambda_{d,1} < 1$. Equivalently, we need to show that there exists a function $\psi$ such that
\begin{equation}
\int_{\tilde\Omega_{d/2}} \Big[|\nabla \psi|^2 - (1 - U_d^2 - V_d^2) \psi^2\Big]\,dx\,dy < 0.
	\label{Eq:S2-1}
\end{equation}
More explicitly, we will show that $\psi = \partial_y U_d$ satisfies \eqref{Eq:S2-1}. 

We start by recalling the fact that $T_1 = 0$ in the proof of Remark \ref{Rem:uvUnstable}. This means
\begin{multline}
\int_{\Omega_d} \Big[(|\nabla \partial_y U_d|^2 + |\nabla \partial_y V_d|^2) - (1 - U_d^2 - V_d^2) ((\partial_y U_d)^2 + (\partial_y V_d)^2)\\
	 + \frac{1}{2}(\partial_y (U_d^2 + V_d^2))^2\Big]\,dx\,dy = 0.
	 \label{Eq:S2-2}
\end{multline}
Noting that, since $\partial_y V_d = 0$ on $\{y = 0\}$, we may apply \eqref{Eq:UV-1a} to get
\begin{equation}
\int_{\Omega_d} \Big[ |\nabla \partial_y V_d|^2 - (1 - U_d^2 - V_d^2)  (\partial_y V_d)^2\Big]\,dx\,dy \geq \int_{\Omega_d} U_d^2 \big|\nabla\big(\frac{\partial_y V_d}{U_d}\big)\big|^2\,dx\,dy.
	\label{Eq:S2-3}
\end{equation}
The right hand side of \eqref{Eq:S2-3} must be positive, since $\partial_y V_d \not\equiv 0$, $\partial_y V_d = 0$ on $\{x = 0\}$ while $U_d < 0$ on $\{x = 0, y > 0\}$. Combining \eqref{Eq:S2-2} and \eqref{Eq:S2-3} and recalling the symmetry of $U_d$ and $V_d$, we get
\[
\int_{\tilde\Omega_{d/2}} \Big[|\nabla (\partial_y U_d)|^2 - (1 - U_d^2 - V_d^2) (\partial_y U_d)^2\Big]\,dx\,dy
	\leq -\frac{1}{2}\int_{\Omega_d}U_d^2 \big|\nabla\big(\frac{\partial_y V_d}{U_d}\big)\big|^2 \,dx\,dy < 0.
\]
This shows that $\psi = \partial_y U_d$ satisfies \eqref{Eq:S2-1} as declared. 
\end{proof}

The next lemma concerns the third eigenvalue $\tilde \mu_{d,3}$ and the fourth eigenvalue $\tilde \mu_{d,4}$ for $\sqrt{2}\pi/2 < d < d_*$.

\begin{lemma}\label{Lem:Sib3rd}
Let $d_*$ be as in Lemma \ref{Lem:ECS-Low}. Then for $\sqrt{2}\pi/2 < d < d_*$, the third eigenvalue of \eqref{Eq:tmuEVP} is $1$ and is simple (that is, $\tilde \mu_{d,3} = 1 < \tilde \mu_{d,4}$). Moreover, if $d_*$ is finite, then $\tilde\mu_{d_*,3} = 1$.
\end{lemma}

\begin{proof}
We only need to show that $\tilde \mu_{d,3} = 1 < \tilde \mu_{d,4}$ for $\sqrt{2}\pi/2 < d < d_*$, as the last assertion follows from this and the stability property \eqref{Eq:tEVPStab}.

Arguing as in the proof of Lemma \ref{Lem:ECS-Low} and using Lemma \ref{Lem:Sib2nd}, we find a maximal $\tilde d_* \in (\sqrt{2}\pi/2, d_*]$ such that, for $\sqrt{2}\pi/2 < d < \tilde d_*$, it holds that $\tilde\mu_{d,1} = 0 < \tilde \mu_{d,2} < \tilde \mu_{d,3} = 1 < \tilde \mu_{d,4}$. 

Arguing by contradiction, assume that the conclusion of the lemma does not hold, so that $\tilde d_* < d_*$ and $\tilde \mu_{\tilde d_*,4} = 1$. 

Let $\tilde Z$ be the eigenspace corresponding to the eigenvalue $\tilde \mu_{\tilde d_*,3} = 1$. We know that $\tilde Z$ contains $V_{\tilde d_*}$ and $\dim \tilde Z \geq 2$. Pick $\tilde \phi \in \tilde Z\setminus \{0\} \not\equiv 0$ which linearly independent from $V_{\tilde d_*}$. Arguing as in the proof of Lemma \ref{Lem:ECS-d*}, if we decompose $\tilde \phi$ as a sum of a function which is even in $x$ and a function which is odd in $x$, then the odd part is a multiple of $V_{\tilde d_*}$. Thus, we may assume without loss of generality that $\tilde \phi$ is even in $x$.

Write $\tilde\phi = \tilde \phi_1 + \tilde\phi_2$ where $\tilde\phi_1$ is even about $y =  \tilde d_*$ and $\tilde \phi_2$ is odd about $y =  \tilde d_*$. Note that both $\tilde \phi_1$ and $\tilde \phi_2$ are even in $x$. Extend $\tilde \phi_2$ evenly across $y = 0$. Then $\tilde \phi_2$ solves the eigenvalue problem \eqref{Eq:muEVP} for $d = \tilde d_*$ and $\mu = 1$. Since $\tilde \phi_2$ are even both in $x$ and $y$, it is $Y_{\tilde d_*}$-orthogonal to both $U_{\tilde d_*}$ and $V_{\tilde d_*}$. As the eigenvalue $\mu = 1$ of \eqref{Eq:muEVP} has multiplicity two (by Lemma \ref{Lem:ECS-Low} and since $\tilde d_* < d_*$), this implies that $\tilde \phi_2 \equiv 0$. In other words, $\tilde \phi$ is odd about $y =  \tilde d_*$.

Recall that, by Courant's nodal domain theorem, the number of nodal domain of $\tilde\phi$ is at most $3$ (as $\tilde \mu = 1$ is the third eigenvalue for \eqref{Eq:tmuEVP}). Since the zero set of $\tilde \phi$ contains the line $y = \tilde d_*$ and is symmetric about this line, we deduce that $\tilde \phi$ does not change sign in $\{0 < y < \tilde d_*\}$ (as well as in $\{ \tilde d_* < y < 2 \tilde d_*\}$). It is well-known that this implies that the restriction of $\tilde\phi$ to $\{0 < y <  \tilde d_*\}$ solves the minimization problem \eqref{Eq:S2-0} with $d = \tilde d_*$, and so $1 = \mu_{\tilde d_*, 3} = \lambda_{\tilde d_*,1} < 1$, which is absurd. This completes the proof.
\end{proof}

We are now ready to prove Lemma \ref{Lem:ECS}.

\begin{proof}[Proof of Lemma \ref{Lem:ECS}]
Suppose by contradiction that the conclusion does not hold. Then the constant $d_*$ in Lemma \ref{Lem:ECS-Low} is finite. Let $P_{d_*}$ be given by Lemma \ref{Lem:ECS-d*}. Extend $P_{d_*}$ evenly across $y = d_*$. Then $P_{d_*}$ solves the eigenvalue problem \eqref{Eq:tmuEVP} with $d = d_*$ and $\tilde\mu = 1$. Since $\tilde \mu_{d_*,3} = 1$ (by Lemma \ref{Lem:Sib3rd}), we have by Courant's nodal domain theorem that the number of nodal domain of $P_{d_*}$ in $\tilde \Omega_{d_*}$ is at most $3$. However, by Lemma \ref{Lem:ECS-d*}(iv), the number of nodal domains of $P_{d_*}$ in $\tilde \Omega_{d_*}$ is in fact $4$. This contradiction concludes the proof.
\end{proof}

\section{$3D$ equilibrium configurations}\label{Sec:3D}

We turn to study the three-dimensional case. In Subsection \ref{SSec3.1}, we prove Theorem \ref{Thm1Ext2} on the existence of the $3D$ solitonic vortex and spoke wheel solutions and indicate how it can be extended to include the case of more general domains as well as higher dimensions. In Subsection \ref{SSec3.2}, we prove Theorem \ref{Thm:MPC-3D} on the mountain pass energy. In Subsection \ref{SSec3.3}, we prove Theorem \ref{Thm:RingCand} on the existence of a candidate vortex ring solution.

\subsection{Minimization under symmetry assumptions}\label{SSec3.1}

Instead of proving Theorem \ref{Thm1Ext2} on the existence of the solitonic vortex solution, we prove a more general result. The conditions \eqref{Eq:Imparity} and \eqref{Eq:kSym} are replaced by the following condition, defined for a given $\ell \in \NN^*$:
\begin{equation}
\begin{cases}
u_1\text{ is even about the planes }  z\cos \frac{q\pi}{\ell}= y \sin \frac{q\pi}{\ell} ,  0 \leq q \leq \ell-1,\\
u_1 \text{ is odd about the planes }  z\cos \frac{(2q+1)\pi}{2\ell}= y \sin \frac{(2q+1)\pi}{2\ell} ,  0 \leq q \leq \ell-1,\\
u_2 \text{ is even about the planes }  z\cos \frac{q\pi}{2\ell}= y \sin \frac{q\pi}{2\ell} ,  0 \leq q \leq 2\ell-1.
\end{cases}
	\label{Eq:2ellSym}
\end{equation}
Here we have written $x' = (y,z)$. 

Let $\mcA_{d,\ell}$ be the set of maps $u$ in $u_s + H^1(\Omega_d,\RR^2)$ satisfying the phase imprinting condition \eqref{Eq:Imprinting} as well as \eqref{Eq:2ellSym}. Let $j_{\ell,1}'$ be the smallest positive zero of $J_\ell'$, the derivative of the Bessel function $J_\ell$ of the first kind and of order $\ell$.

\begin{theorem}
\label{Thm1Ext2Gen}
Let $N = 3$ and $\ell \in \NN^*$.
\begin{enumerate}[(i)]
\item If $0 < d \leq \sqrt{2}j_{\ell,1}'$, then the soliton solution $u_s$ is the unique minimizer of $\mcF_d$ in $\mcA_{d,\ell}$.

\item If $d > \sqrt{2}j_{\ell,1}'$, then the soliton solution $u_s$ is an unstable critical point of $\mcF_d$ in $\mcA_{d,\ell}$ and $\mcF_d$ has exactly two minimizers in $\mcA_{d,\ell}$ of the form $(\pm u_1, u_2)$ with $u_1 \neq 0$. Moreover, the zero set of these minimizers is the union of $\ell$ line segments which are on the $(y,z)$-plane and intersect equiangularly at the origin. They are the intersection of the planes $z\cos \frac{(2q+1)\pi}{2\ell}= y \sin \frac{(2q+1)\pi}{2\ell}$,  $0 \leq q \leq \ell-1$, with the $(y,z)$-plane.
\end{enumerate}
\end{theorem}

\begin{proof}
First, the proof of Theorem \ref{Thm:S=>M} applies verbatim yielding that $u_s$ is a minimizer of $\mcF_d$ in $\mcA_{d,\ell}$ if and only if $u_s$ is stable in the first direction within $\mcA_{d,\ell}$. The threshold value $d = \sqrt{2}j'_{\ell,1}$ at which $u_s$ loses it stability is computed in the same way as in the proof of Corollary \ref{Cor:d1}, but now using the Poincar\'e inequality
\begin{equation*}
\int_{|x'| < 1} |\nabla_{x'}\varphi|^2\,dx' \geq (j'_{\ell,1})^2\int_{|x'| < 1}  \varphi ^2\,dx'
\end{equation*}
for all $\varphi \in H^1(\{|x'| < 1\})$ satisfying
\begin{equation}
\begin{cases}
\varphi \text{ is even about the lines }  z\cos \frac{q\pi}{\ell}= y \sin \frac{q\pi}{\ell} ,  0 \leq q \leq \ell-1,\\
\varphi  \text{ is odd about the lines }  z\cos \frac{(2q+1)\pi}{2\ell}= y \sin \frac{(2q+1)\pi}{2\ell} ,  0 \leq q \leq \ell-1.
\end{cases}
	\label{Eq:varphiSym}
\end{equation}

Consider the case $d > \sqrt{2}j'_{\ell,1}$. By the above, $\mcF_d$ has a minimizer $u$ in $\mcA_{d,\ell}$ which is different from $u_s$. Let
\begin{align*}
S  &= \Big\{(r\cos\theta,r\sin\theta): 0 < r < 1, 0 < \theta < \frac{\pi}{2p}\Big\}.
\end{align*}
We may then construct the functions $(U,V)$  with $U \leq 0$ and $V \leq 0$ as in Step 1 of the proof of Theorem \ref{Thm:UV} with $Q_d = (0,\infty) \times dS$. The rest of the proof remains unchanged yielding that the restriction to $Q_d$ of any minimizers of $\mcF_d$ in $\mcA_{d,\ell}$ must take the form $(\pm U,V)$. Since every map in $\mcA_{d,\ell}$ is uniquely determined by its restriction to $Q_d$, we deduce that $\mcF_d$ has exactly two minimizers in $\mcA_{d,\ell}$. Furthermore, any such minimizers vanish along the lines $x = 0$, $z\cos \frac{(2q+1)\pi}{2\ell}= y \sin \frac{(2q+1)\pi}{2\ell}$,  $0 \leq q \leq \ell-1$. These form $\ell$ vortex lines passing through the origin. 
 The proof is complete.
\end{proof}

We note that Theorem \ref{Thm1Ext2} remains valid if the circular cross-section of $\Omega_d$ is replaced by a dilation of a domain $\Sigma$ which satisfies a symmetry coherent with \eqref{Eq:2ellSym}, namely $\Sigma$ is symmetric about the lines $z\cos \frac{q\pi}{2\ell}= y \sin \frac{q\pi}{2\ell}  $, $0 \leq q \leq 2\ell-1$. See Figure \ref{Fig1} for an illustration. In this case, the threshold value $\sqrt{2}j_{\ell,1}'$ is replaced by $\sqrt{2\lambda_{\Sigma,\ell}}$ where $\lambda_{\Sigma,\ell}$ be largest positive number such that
\begin{equation*}
\int_\Sigma |\nabla_{x'}\varphi|^2\,dx' \geq \lambda_{\Sigma,\ell}\int_\Sigma  \varphi ^2\,dx' \text{ for all for all $\varphi \in H^1(\Sigma)$ satisfying \eqref{Eq:varphiSym}}.
\end{equation*}

\begin{figure}[h]
\begin{center}
\begin{tikzpicture}
\draw[smooth,samples=100,domain = 0:360] plot ({(1.5 + .5*cos(6*\x))*cos(\x)},{(1.5 + .5*cos(6*\x))*sin(\x)});
\draw (-2,0)--(2,0);
\draw (1,1.73)--(-1,-1.73);
\draw (1,-1.73)--(-1,1.73);
\draw[dashed] (0,-1)--(0,1);
\draw[dashed] (.87,-.5)--(-.87,.5);
\draw[dashed] (-.87,-.5)--(.87,.5);
\end{tikzpicture}

\caption{A $2D$ domain $\Sigma$ for which one can use the proof of Theorem \ref{Thm1Ext2} to construct critical point on the domain $\RR \times d\Sigma$ where vortex lines on the central plane $x = 0$ can be arranged at the three solid lines or the three dashed lines.}
\label{Fig1}
\end{center}
\end{figure}
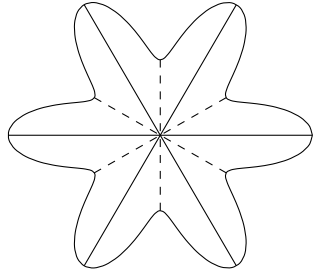

As a final remark for this subsection, we note that our treatment is applicable for higher dimensional cross-section $\Sigma \subset \RR^{N-1}$, $N \geq 4$. We make the following assumptions on $\Sigma$:
\begin{equation}
\begin{minipage}{.85\textwidth}
There is a discrete subgroup $G \leq O(N-1)$ which leaves $\Sigma$ invariant, i.e. $g(\Sigma) = \Sigma$ for every $g \in G$,
\end{minipage}
\label{Eq:S1}
\end{equation}
and
\begin{equation}
\begin{minipage}{.85\textwidth}
There is a reflection matrix $j \in O(N-1) \setminus G$ which leaves $\Sigma$ invariant and the $G$-orbits invariant, i.e. $j(\Sigma) = \Sigma$ and $jG(x') = G(x')$ for every $x' \in \RR^{N-1}$.
\end{minipage}
\label{Eq:S2}
\end{equation}
Let $\mcA_{d,G,j}$ be the set of maps $u$ in $u_s + H^1(\RR \times d\Sigma,\RR^2)$ satisfying the phase imprinting condition \eqref{Eq:Imprinting} as well as
\[
\begin{cases}
u(x,x') = u(x,gx') \text{ for all } g \in G, x' \in d\Sigma,\\
u_1(x,x') = -u_1(x,jx'), u_2(x,x') =  u_2(x,jx') \text{ for all } x' \in d\Sigma.
\end{cases}
\]
Let $\lambda_{\Sigma,G,j}$ be largest positive number such that
\begin{equation*}
\int_\Sigma |\nabla_{x'}\varphi|^2\,dx' \geq \lambda_{\Sigma,G,j}\int_\Sigma  \varphi ^2\,dx'
\end{equation*}
for all $\varphi \in H^1(\Sigma)$ satisfying
\[
\begin{cases}
\varphi(x') = \varphi(gx') \text{ for all } g \in G, x' \in \Sigma,\\
\varphi(x') = -\varphi(x,x') \text{ for all } x' \in  \Sigma.
\end{cases}
\]
(Note that $\lambda_{\Sigma,G,j} > 0$ due to the fact that $\int_\Sigma \varphi\,dx' = 0$.)

\begin{theorem}
\label{Thm1ExtG}
Let $N \geq 4$ and $\Sigma \subset \RR^{N-1}$ be a bounded domain with smooth boundary satisfying \eqref{Eq:S1} and \eqref{Eq:S2}. 
\begin{enumerate}[(i)]
\item If $0 < d \leq \sqrt{2\lambda_{\Sigma,G,j}}$, then the soliton solution $u_s$ is the unique minimizer of $\mcF_d$ in $\mcA_{d,G,j}$.

\item If $d > \sqrt{2\lambda_{\Sigma,G,j}}$, then the soliton solution $u_s$ is an unstable critical point of $\mcF_d$ in $\mcA_{d,G,j}$ and $\mcF_d$ has at least two minimizers in $\mcA_{d,G,j}$ of the form $(\pm u_1, u_2)$ with $u_1 \neq 0$. 
\end{enumerate}
\end{theorem}

In some situation, we may be able to show that $\mcF_d$ has exactly two minimizers in $\mcA_{d,G,j}$, but due to a lack of physical contexts, we have not pursued this further.

\begin{proof}
The proof is the same except for the existence of at least two minimizers when $d > \sqrt{2\lambda_{\Sigma,G,j}}$. In this case, note that $\mcF_d$ has a minimizer $u$ in $\mcA_{d,G,j}$ which is different from $u_s$.  To conclude, it suffices to show that $u_1 \neq 0$, since this implies that $(-u_1, u_2)$ is then a different minimizer of $\mcF_d$ in $\mcA_{d,G,j}$.

Arguing by contradiction, we assume that $u_1 \equiv 0$. We then have
\begin{multline}
\int_{d\Sigma} \int_{-\infty}^\infty \Big[\frac{1}{2}(\partial_x u_2)^2 + \frac{1}{4}(1 - u_2^2)^2\Big]\,dx\,dx' \leq \mcF_d[u]\\
	 \leq \mcF_d[u_s] = d^{N-1} |\Sigma|\int_{-\infty}^\infty \Big[\frac{1}{2}(\partial_x u_s)^2 + \frac{1}{4}(1 - u_s^2)^2\Big]\,dx.
	 \label{Eq:Pretest}
\end{multline}
In particular, there exists $x_0' \in d\Sigma$ such that
\begin{equation}
 \int_{-\infty}^\infty \Big[\frac{1}{2}(\partial_x u_2(x,x_0'))^2 + \frac{1}{4}(1 - u_2(x,x_0')^2)^2\Big]\,dx \leq \int_{-\infty}^\infty \Big[\frac{1}{2}(\partial_x u_s)^2 + \frac{1}{4}(1 - u_s^2)^2\Big]\,dx.
 	\label{Eq:testx0'}
\end{equation}
Consequently, for any $\tilde d > 0$, the map $\tilde u(x,x') := u_2(x,x_0') e_2 \in \mcA_{\tilde d,G,j}$ satisfies
\[
\mcF_{\tilde d}[\tilde u] \leq \mcF_{\tilde d} [u_s].
\]
Since $u_s$ is the unique critical point of $\mcF_{\tilde d}$ in $\mcA_{\tilde d, G,j}$ for small $\tilde d$, we deduce that $u_2(x,x_0') = u_s(x)$. Since this holds for all $x_0'$ such that \eqref{Eq:testx0'} holds, we infer that the chain of inequalities in \eqref{Eq:Pretest} are all saturated, and hence $\mcF_d[u] = \mcF_d[u_s]$, which contradicts the fact that $u_s$ is not a minimizer for $\mcF_d$ in $\mcA_{d,G,j}$. The proof is complete.
\end{proof}

\subsection{Estimate of $3D$ mountain-pass energy in $\mcA_d^{\rm x}$}\label{SSec3.2}

\begin{proof}[Proof of Theorem \ref{Thm:MPC-3D}]
\underline{Step 1:} We show that $c_d \geq \mcF_d[u_s]$ when $d \leq \sqrt{2}j'_{1,1}$.

This is more or less identical to Step 1 of the proof of Lemma \ref{Lem:M1}. We thus only give the key points. For a given $h \in \Gamma_d$, we select $t_0 \in (-1,1)$ such that
\begin{equation}
\int_{d\Sigma} h(t_0)_1(0,y,z)\,dy\,dz = 0
	\label{Eq:ZA-1-3D}
\end{equation}
and proceed to prove $\mcF_d[h(t_0)] \geq \mcF_d[u_s]$. With $w = h(t_0) - u_s$, this reduces to proving that
\begin{equation}
\int_{\Omega_d} \big[|\nabla w_1|^2 - (1 - |u_s|^2)w_1^2\big]\,dx\,dx' \geq 0.
	\label{Eq:ZA-2-3D}
\end{equation}

Let $\lambda_1 = 0 < \lambda_2 = (j'_{1,1})^2 \leq \ldots$ be the Neumann eigenvalues of $-\Delta$ in the unit disk $\mathbb{D}$ and let $\phi_1, \phi_2, \ldots$ be the corresponding orthonormal basis of eigenfunctions in $L^2(\mathbb{D})$. Write
\[
w_1(x,y,z) =\sum_{k=0}^\infty a_k(x) \phi_k(x'/d).
\]
Inequality \eqref{Eq:ZA-2-3D} becomes
\[
\sum_{k=1}^\infty \int_{-\infty}^\infty \Big[(a_k')^2  - \Big(1 - \frac{\lambda_k}{d^2} - |u_s|^2\Big) a_k^2 \Big]\,dx \geq 0.
\]
We then use \eqref{Eq:usM-1} to treat the terms involving $a_0$ and use the proof of \eqref{Eq:ZA-4} to treat the remaining terms (while keeping in mind that $d \leq \sqrt{2}j'_{1,1} = \sqrt{2\lambda_2}$). We omit the details.

\bigskip
\noindent
\underline{Step 2:} We show that $c_d \leq \mcF_d[u_s]$ when $d \leq \sqrt{2}j'_{1,1}$.

This is identical to Step 2 of the proof of Lemma \ref{Lem:M1}.

\bigskip
\noindent
\underline{Step 3:} We show that $c_d < \mcF_d[u_s]$ when $d > \sqrt{2}j'_{1,1}$.

To this end, we construct $h \in \Gamma_d$ such that $\mcF_d[h(t)] < \mcF_d[u_s]$ for all $t \in [-1,1]$.

Consider the function $g: \RR^2 \rightarrow \RR$ given by.
\[
g(a_1,a_2) = \mcF_d\Big[u_s + \Big(a_1 A + a_2 A \phi_2(x'/d),0\Big)\Big],
\]
where $A = \mathrm{sech}\frac{x}{\sqrt{2}}$ and $\phi_2$ is a normalized eigenfunction of $-\Delta$ in $\Sigma$ corresponding to the eigenvalue $\lambda_2 = (j'_{1,1})^2$. Note that $g(0,0) = \mcF_d[u_s]$, $\nabla g(0,0) = 0$, $\frac{\partial^2 g}{\partial a_1 \partial a_2}(0,0) = 0$. By Remark \ref{Rem:usUnstable}, 
\[
 \frac{\partial^2 g}{\partial a_1^2}(0,0) < 0.
 \]
We claim that
\[
\frac{\partial^2 g}{\partial a_2^2}(0,0) < 0.
\]
Indeed, we have
\begin{align*}
\frac{\partial^2 g}{\partial a_2^2}(0,0) 
	&= \int_{\Omega_d} [|\nabla(A \phi_2(x'/d))|^2 - \mathrm{sech}^2\frac{x}{\sqrt{2}} A^2 \phi_2(x'/d)^2]\,dx\,dx'\\
	&= d^2\int_{-\infty}^\infty [(A')^2 + \frac{(j'_{1,1})^2}{d^2} A^2 - \mathrm{sech}^2\frac{x}{\sqrt{2}} A^2  ]\,dx\\
	&= d^2\int_{-\infty}^\infty \Big(\frac{3}{2}\tanh^2\frac{x}{\sqrt{2}} + \frac{(j'_{1,1})^2}{d^2}   - 1 \Big)\mathrm{sech}^2\frac{x}{\sqrt{2}}\,dx\\
	&= d^2\int_{-1}^1 \Big(\frac{3}{2} \tau^2 + \frac{(j'_{1,1})^2}{d^2}   - 1\Big)d\tau\\
	&= d^2\Big(  \frac{2(j'_{1,1})^2}{d^2}   - 1\Big) < 0.
\end{align*}
Therefore, there exists $\delta > 0$ such that
\[
g(a_1, a_2) < \mcF_d[u_s] \text{ on } \{a_1^2 + a_2^2 = \delta^2\}.
\]
We now define the path $h$ is by first connecting $\psi_n^\pm$ to $u_s + (\pm \delta A, 0)$ as in Step 2  of the proof of Lemma \ref{Lem:M1}, and then connecting $u_s + (\delta A, 0)$ to $u_s + (-\delta A, 0)$ by $u_s + (a_1 A + a_2 A \phi_2(x'/d),0)$ along the circle $a_1^2 + a_2^2 = \delta^2$. Clearly, $\mcF_d[h(t)] < \mcF_d[u_s]$ along this path and hence $c_d < \mcF_d[u_s]$.

\bigskip
\noindent
\underline{Step 4:} We show that there exists $\underline{c}(\Sigma) > 0$ such that $c_d \geq \underline{c}(\Sigma)$ when $d > \sqrt{2} j'_{1,1}$.

This is similar to the proof of Lemma \ref{Lem:M3X}. We omit the details.
\end{proof}

\subsection{Candidate vortex ring solution}\label{SSec3.3}

We turn to the question of the existence of a critical point of $\mcF_d$ in $\mcA_d^{\rm rad}$ which is different from the vortex solution when $d$ is larger than $\sqrt{2}j'_{0,1}$. As pointed out in the introduction, the argument in Subsection \ref{SSec3.2} proves that the mountain pass energy $c_d^{\rm rad}$ satisfies the bounds \eqref{Eq:3DMPRad-1} and \eqref{Eq:3DMPRad-2}. However, it is not clear if $\mcF_d$ (as a function on $\mcA_d^{\rm x}$) satisfies the Palais-Small condition at energy level $c_d^{\rm rad}$. We therefore cannot directly apply the mountain pass theorem to obtain a critical point at this energy level. Instead, we consider solutions of the Euler-Lagrange equation on cylinders $\Omega_{R,d} :=[-R,R] \times d\Sigma$ for finite $R$ and attempt to obtain a candidate solution by sending $R \rightarrow \infty$. The outcome of this is Theorem \ref{Thm:RingCand}.

For $R \in (0,\infty)$, consider the functional
\[
\mcF_{R,d}[u] = \int_{\Omega_{R,d}} \Big(\frac{1}{2} |\nabla u|^2 + \frac{1}{4} (1 - |u|^2)^2\Big)\,dx\,dx'
\]
for maps $u$ belonging to the set $\mcA_{R,d}^{\rm rad}$ consisting of maps in $H^1(\Omega_{R,d},\RR^2)$ satisfying the phase imprinting condition \eqref{Eq:Imprinting}, the symmetry condition \eqref{Eq:RadSym} and the boundary condition $u(\pm R,\cdot) = (0,\pm 1)$. Sometimes, we also consider the set $\mcA_{R,d}^{\rm x}$ consisting of maps in $H^1(\Omega_{R,d},\RR^2)$ satisfying the phase imprinting condition \eqref{Eq:Imprinting} and the boundary condition $u(\pm R,\cdot) = (0,\pm 1)$. By a slight abuse of notation, we let $\Omega_{\infty,d} = \Omega_d, \mcF_{\infty,d} = \mcF_d$, $\mcA_{\infty,d}^{\rm rad} = \mcA_d^{\rm rad}$, and $\mcA_{\infty,d}^{\rm x} = \mcA_d^{\rm x}$.

There are three special critical points of $\mcF_{R,d}$ in $\mcA_{R,d}^{\rm rad}$ which depend only on $x$. The first one is the counterpart of the soliton which is given by
\[
u_{s,R}(x) = \Xi_R(x)e_2
\]
where $\Xi_R$ is given implicitly by
\[
\int_0^{\Xi_R(x)} (q_R + \frac{1}{2}(1 - s^2)^2)^{-1/2}\,ds = x \text{ for } x \in [-R,R],
\]
where $q_R$ is the unique positive constant such that 
\[
\int_0^{1} (q_R + \frac{1}{2}(1 - s^2)^2)^{-1/2}\,ds = R.
\]
As $R \nearrow \infty$, we have $q_R \searrow 0$, $\Xi_R(x) \searrow \tanh\frac{x}{\sqrt{2}}$ locally uniformly in $(-\infty,\infty)$ and
\begin{equation}
\mcF_{R,d}[u_{s,R}] \searrow \mcF_d[u_s] = \frac{2\sqrt{2}|\Sigma|d^2}{3}.
	\label{Eq:RSolEnergy}
\end{equation}
The other two critical points are given by
\begin{equation}
\underline{u}_{R,d}^\pm(x,x') = \rho_{R,d}(x) \Big(\pm \cos \chi_{R,d}(x), -\sin \chi_{R,d}(x) \Big),
	\label{Eq:uuRd}
\end{equation}
where $\rho_{R,d} > 0$, $-\frac{\pi}{2} \leq \chi_{R,d} \leq \frac{\pi}{2}$ and $(\rho_{R,d},\chi_{R,d})$ minimizes the functional
\[
\int_{-R}^R \Big[\frac{1}{2} (\rho')^2 +  \frac{1}{2}\rho^2 (\chi')^2 + \frac{1}{4} (1 - \rho^2)^2\Big]\,dx
\]
defined for even $\rho \in H^1((-R,R))$ and odd $\chi \in H^1((-R,R))$ with boundary conditions $\rho(\pm R) = 1$ and $\chi(\pm R) = \pm \frac{\pi}{2}$. Note that
\begin{equation}
\underline{c}_{R,d} := \mcF_{R,d}[\underline{u}_{R,d}^\pm] \leq \mcF_{R,d}\Big[\Big(\pm \cos \frac{\pi x}{2R}, -\sin \frac{\pi x}{2R} \Big)\Big] = \frac{\pi^2 |\Sigma| d^2}{4R^2}.
	\label{Eq:RMinEnergy}
\end{equation}

It turns out that $\underline{u}_{R,d}^\pm$ are in fact the only two minimizers of $\mcF_{R,d}$ in $\mcA_{R,d}^{\rm x}$. This is an immediate consequence of the following result, a version of which appeared earlier in \cite{INSZ-ENS}.

\begin{lemma}\label{Lem:Pos=>Min}
Let $d > 0$, $R \in (0,\infty]$ and $u \in \mcA_{R,d}^{\rm x}$ be a critical point of $\mcF_{R,d}$. If $u_1 > 0$ in $\Omega_{R,d}$, then $\mcF_{R,d}$ has exactly two minimizers in $\mcA_{R,d}^{\rm x}$ which are $(\pm u_1, u_2)$.
\end{lemma}

\begin{corollary}
Let $d > 0$ and $R \in (0,\infty)$. Then $\underline{u}_{R,d}^\pm$ are the only two minimizers of $\mcF_{R,d}$ in $\mcA_{R,d}^{\rm x}$. Consequently, $\underline{c}_{R,d}$ is the minimal energy of $\mcF_{R,d}$ in $\mcA_{R,d}^{\rm x}$.
\end{corollary}

\begin{proof}
Let $W$ be the space of maps in $H^1(\Omega_{R,d})$ such that $u + w \in \mcA_{R,d}^{\rm x}$. As in the proof of Theorem \ref{Thm:S=>M}, we have for any $w \in W$ that
\begin{align}
\mcF_{R,d}[u + w] - \mcF_{R,d}[u]
	&= \frac{1}{2}\int_{\Omega_{R,d}} \Big(|\nabla w|^2 - (1 - |u|^2) |w|^2\Big) \,dx\,dx' \nonumber\\
		&\qquad+ \frac{1}{4}\int_{\Omega_{R,d}} (2u \cdot w + |w|^2)^2 \,dx\,dx'.
	\label{Eq:ABC-1}
\end{align}
Now, note that $u_1 > 0$ in $\Omega_{R,d}$, $u_1 \in W$ and $-\Delta u_1 - (1 - |u|^2) u_1 = 0$ in $\Omega_{R,d}$. Thus, $-\Delta - (1 - |u|^2)$ is positive semi-definite on $W$ and its kernel is generated by $u_1$. This implies that the right hand side of \eqref{Eq:ABC-1} is non-negative and thus $u$ is a minimizer of $\mcF_{R,d}$ in $\mcA_{R,d}^{\rm x}$. Suppose that $u + w$ is another minimizer of $\mcF_{R,d}$ in $\mcA_{R,d}^{\rm x}$. Then we must have that $w = u_1 c$ for some constant vector $c \in \RR^2$, and
\[
0 
	= 2u \cdot w + |w|^2 
	= 2u \cdot c u_1  + |c|^2 u_1^2 = (2c_1 + |c|^2)u_1^2 + 2c_2u_1 u_2 \text{ in } \Omega_{R,d}
\]
Since $u_1 > 0$, this implies that
\[
(2c_1 + |c|^2)u_1 + 2c_2 u_2 = 0 \text{ in } \Omega_{R,d}.
\]
By the boundary condition at $x = \pm R$, this implies further that $c_2 = 0$, which then implies further that $c_1 \in \{0,-2\}$. This means that $u+w \in \{(\pm u_1, u_2)\}$. As $\mcF_{R,d}[(-u_1,u_2)] = \mcF_{R,d}[u]$, this completes the proof.
\end{proof}

\begin{lemma}\label{Lem:FRdStrictMin}
Let $d > 0$ and $R \in (0,\infty)$. The second variation of $\mcF_{R,d}$ at $\underline{u}_{R,d}^\pm$ is positive definite. More precisely, there exists $\delta_{R,d} > 0$ such that
\[
\frac{d^2}{dt^2}\Big|_{t = 0} \mcF_{R,d}[\underline{u}_{R,d}^\pm + tw] 
	\geq \delta_{R,d}\|w\|_{H^1(\Omega_{R,d})}^2
\]
for all $w \in H^1(\Omega_{R,d},\RR^2)$ satisfying the phase imprinting condition \eqref{Eq:Imprinting} and $w(\pm R, \cdot) = 0$.
\end{lemma}

\begin{proof}
Let $W$ denote the space of $w \in H^1(\Omega_{R,d},\RR^2)$ satisfying the phase imprinting condition \eqref{Eq:Imprinting} and $w(\pm R, \cdot) = 0$. For $w \in W$, define
\[
J[w] = \frac{d^2}{dt^2}\Big|_{t = 0} \mcF_{R,d}[\underline{u}_{R,d}^\pm + tw] 
	= \int_{\Omega_{R,d}} \Big[|\nabla w|^2 - (1 - |\underline{u}_{R,d}^\pm|^2)|w|^2 + 4 (\underline{u}_{R,d}^\pm \cdot w)^2\Big]\,dx\,dx'.
\]
We saw in the proof of Lemma \ref{Lem:Pos=>Min} that $-\Delta + (1 - |\underline{u}_{R,d}^\pm|^2)$ is positive semi-definite on $W$. Thus, the direct method shows that the minimization problem
\[
\min \Big\{J[w]: w \in W, \|w\|_{L^2(\Omega_{R,d})} = 1\Big\}
\]
is attained by some $w_* \in W$ with $\|w_*\|_{L^2(\Omega_{R,d})} = 1$. Note that $J[w_*] > 0$, for if $J[w_*] = 0$, then $w_* = (\underline{u}_{R,d}^+)_1 c$ for some constant vector $c \in \RR^2$, and
\[
0 
	= \underline{u}_{R,d}^+ \cdot w_*
	= \underline{u}_{R,d}^+ \cdot c (\underline{u}_{R,d}^+)_1
	= \rho_{R,d}^2 \Big(c_1 \cos^2 \chi  - c_2\sin \chi\cos  \chi\Big),
\]
which implies $c = 0$ and $w_* = 0$ contradicting the fact that $\|w_*\|_{L^2(\Omega_{R,d})} = 1$.

Now, we have
\[
J[w] \geq J[w_*] \int_{\Omega_{R,d}} |w|^2\,dx\,dx' \quad \text{ for all } w \in W.
\]
As the definition of $J$ also gives
\[
J[w] + \int_{\Omega_{R,d}} |w|^2\,dx\,dx' \geq \int_{\Omega_{R,d}} |\nabla w|^2\,dx\,dx'  \quad \text{ for all } w \in W,
\]
we deduce that
\[
(1 + 2J[w_*]^{-1}) J[w] \geq J[w] + \int_{\Omega_{R,d}} 2|w|^2\,dx\,dx' \geq \|w\|_{H^1(\Omega)}^2  \quad \text{ for all } w \in W.
\]
The conclusion follows with $\delta_{R,d} = (1 + 2J[w_*]^{-1})^{-1}$.
\end{proof}

In view of Lemma \ref{Lem:FRdStrictMin} and the fact that $\mcF_{R,d}[\underline{u}_{R,d}^+] = \mcF_{R,d}[\underline{u}_{R,d}^-] = \underline{c}_{R,d}$, the mountain pass theorem 
\cite{Rabinowitz} implies the existence of a critical point of $\mcF_{R,d}$ in $\mcA_{R,d}^{\rm rad}$ which is different from $\underline{u}_{R,d}^\pm$. More precisely, if we define
 \begin{align}
\Gamma_{R,d}^{\rm rad}
	&= \Big\{h \in C([-1,1],\mcA_{R,d}^{\rm rad}): h(\pm 1)  = \underline{u}_{R,d}^\pm\Big\}, \label{Eq:GamRdDef}\\
c_{R,d} ^{\rm rad}
	&= \inf_{h \in \Gamma_{R,d}} \max_{t \in [-1,1]} \mcF_{R,d}[h(t)],
\label{Eq:MPRdEnergy}
\end{align}
then $c_{R,d}^{\rm rad} > \underline{c}_{R,d}$ and is a critical value of $\mcF_{R,d}$ in $\mcA_{R,d}^{\rm rad}$. We would like to show that, along a sequence $R \rightarrow \infty$, a sequence of corresponding critical points converges to a critical point of $\mcF_d$ which is different from the soliton solution when $d > \sqrt{2} j'_{0,1}$. To this end, we need some preparation.

Note that, the proof of Lemma \ref{Lem:M3X} can be applied yielding 
\begin{equation}
c_{R,d}^{\rm rad} \geq \underline{c}^{\rm rad}\text{ for } R \geq d
	\label{Eq:cRdLB}
\end{equation}
for the same positive constant $\underline{c}^{\rm rad}$ as in \eqref{Eq:3DMPRad-2}. 

\begin{remark}
Let $d > 0$. There exists $\underline{n}_{d}$ such that provided $n \geq \underline{n}_{d}$ in the definition \eqref{Eq:GamDef} and \eqref{Eq:MPEnergy} of $\Gamma_d$ and $c_d$, we have 
\[
\liminf_{R \rightarrow \infty} c_{R,d}^{\rm rad} \geq c_d^{\rm rad}.
\]
\end{remark}

\begin{proof} By \eqref{Eq:3DMPRad-2}, when $n$ is sufficiently large, $c_d^{\rm rad}$ is independent of $n$. We make two observations:
\begin{itemize}
\item Every map in $\mcA_{R,d}^{\rm rad}$ extends to a map in $\mcA_d^{\rm rad}$ with the same energy by simply assigning the value $(0,1)$ for $x > R$ and $(0,-1)$ for $x < -R$.
\item The map $\underline{u}_{R,d}^+$ can be connected to $\psi_n^+$ by a path whose maximal energy is of order $O(n^{-2} + R^{-2})$: First, the energy of $\underline{u}_{R,d}^+$ is $O(R^{-2})$. As in Step 2 of the proof of Lemma \ref{Lem:M1}, one starts by deforming the modulus $\rho_{R,d}$ to $1$ via a linear convex combination while keeping the phase fixed, gaining at most $O(R^{-2})$ in energy along the way. One then deforms the phase via a linear convex combination to $\chi(x/n)$ while keeping the modulus fixed, gaining at most $O(n^{-2} + R^{-2})$ in energy along the way.
\end{itemize} 
Thus, for any path in $\Gamma_{R,d}$, there corresponds a path in $\Gamma_d$ whose maximal energy is at most that of the original path plus $O(n^{-2} + R^{-2})$. The assertion follows.
\end{proof}

Our next lemma shows that the obtained critical point at the mountain pass critical value $c_{R,d}$ is different from the soliton $u_{s,R}$ for $d > \sqrt{2}j'_{0,1}$ and sufficiently large $R$.

\begin{lemma}\label{Lem:MPR}
Suppose $d > \sqrt{2}j'_{0,1}$. There exists $\underline{R}_d > 0$ such that the mountain pass energy $c_{R,d}$ defined by \eqref{Eq:MPRdEnergy} satisfies
\begin{equation}
	c_{R,d}^{\rm rad} < \mcF_{R,d}[u_{s,R}] \text{ for } R > \underline{R}_d.
	\label{Eq:3DMPR}
\end{equation}
\end{lemma}

\begin{proof} As in Step 3 of the proof of Theorem \ref{Thm:MPC-3D}, it suffices to show that, when $R$ is sufficiently large, we can adjust the function $A(x) = \mathrm{sech} \frac{x}{\sqrt{2}}$ to a function $A_R: [-R,R] \rightarrow (0,\infty)$ with $A_R(\pm R) = 0$ such that the function $g_R: \RR^2 \rightarrow \RR$ given by
\[
g_R(a_1,a_2) = \mcF_{R,d}\Big[u_{s,R} + \Big(a_1 A_R + a_2 A_R \phi_2(|x'|/d),0\Big)\Big]
\]
satisfies $\frac{\partial^2 g_R}{\partial a_1^2}(0,0) < 0$ and $\frac{\partial^2 g_R}{\partial a_2^2}(0,0) < 0$. Here $\phi_2$ is a normalized radially symmetric eigenfunction of $-\Delta$ in the unit disk corresponding to the eigenvalue $(j'_{0,1})^2$. (We continue to have $g_R(0,0) = \mcF_{R,d}[u_{s,R}]$, $\nabla g_R(0,0) = 0$ and $\frac{\partial^2 g_R}{\partial a_1 \partial a_2}(0,0) = 0$.)

We compute
\begin{align*}
\frac{\partial^2 g}{\partial a_1^2}(0,0) 
	&= \int_{\Omega_{R,d}} [(A_R')^2 - (1 - \Xi_R^2) A_R^2  ]\,dx\,dx'\\
	&= \pi d^2 \int_{-R}^R [(A_R')^2   - (1 - \Xi_R^2) A_R^2  ]\,dx,\\		 
\frac{\partial^2 g}{\partial a_2^2}(0,0) 
	&= \int_{\Omega_{R,d}} \big[|\nabla(A_R \phi_2(x'/d))|^2 - (1 - \Xi_R^2) A_R^2 \phi_2(x'/d)^2\big]\,dx\,dx' \\
	&= d^2\int_{-R}^R \Big[(A_R')^2 + \frac{(j'_{0,1})^2}{d^2} A_R^2 - (1 - \Xi_R^2) A_R^2 \Big ]\,dx. 
\end{align*}
Therefore, we only need to find $A_R \in H_0^1((-R,R))$ such that 
\begin{equation}
\int_{-R}^R \Big[(A_R')^2 + \frac{(j'_{0,1})^2}{d^2} A_R^2 - (1 - \Xi_R^2) A_R^2  \Big]\,dx
	< 0.
	\label{Eq:MPR-1}
\end{equation}

We assume that ansatz
\[
A_R = (1 - \Xi_R^2)^{\alpha} \in H_0^1((-R,R)),
\]
where $\alpha \in ( \frac{1}{2}, 1)$ is to be fixed. Using the identities
\[
(\Xi_R')^2 = q_R  + \frac{1}{2}(1 - \Xi_R^2)^2 = q_R + \frac{1}{2} A_R^{\frac{2}{\alpha}} \text{ and } \Xi_R'' = - (1 - \Xi_R^2)\Xi_R = - A_R^{\frac{1}{\alpha}} \Xi_R,
\]
we compute
\begin{align*}
-A_R' 
	&=  2\alpha \Xi_R \Xi_R' A_R^{1-\frac{1}{\alpha}},\\
-A_R''
	&= 2(1-\alpha) \Xi_R^2 (\Xi_R')^2 A_R^{1-\frac{2}{\alpha}}  + 2\alpha (\Xi_R')^2 A_R^{1-\frac{1}{\alpha}} + 2\alpha \Xi_R \Xi_R'' A_R^{1-\frac{1}{\alpha}}\\
	&=  2[\alpha + (1-2\alpha) \Xi_R^2] (\Xi_R')^2 A_R^{1-\frac{2}{\alpha}}  + 2\alpha \Xi_R \Xi_R'' A_R^{1-\frac{1}{\alpha}} \\
	&=  \Big\{\frac{2q_R [\alpha + (1-2\alpha) \Xi_R^2]}{(1 - \Xi_R^2)^2}    +    \alpha + (1-4\alpha) \Xi_R^2 \Big\}  A_R .
\end{align*}
This implies that
\[
\int_{-R}^R (A_R')^2\,dx
	= \int_{-R}^R \Big\{\frac{2q_R [\alpha + (1-2\alpha) \Xi_R^2]}{(1 - \Xi_R^2)^2}    +   \alpha + (1-4\alpha) \Xi_R^2 \Big\}  A_R^2\,dx,
\]
and hence
\begin{align}
&\int_{-R}^R \Big[(A_R')^2 + \frac{(j'_{0,1})^2}{d^2} A_R^2 - (1 - \Xi_R^2) A_R^2  \Big]\,dx\nonumber\\
	&\qquad = \int_{-R}^R \Big\{\frac{2q_R [\alpha + (1-2\alpha) \Xi_R^2]}{(1 - \Xi_R^2)^2}   + \frac{(j'_{0,1})^2}{d^2} +  (\alpha - 1) + (2-4\alpha) \Xi_R^2 \Big\}  A_R^2\,dx.\label{Eq:MPR-2}
\end{align}
Note that, since $\alpha \in (\frac{1}{2},1)$, the terms $\alpha - 1$ and $(2-4\alpha) \Xi_R^2$ are negative and helpful in achieving \eqref{Eq:MPR-1}. (This is the reason why we used $A_R''$ rather than the exact formula for $A_R'$ in evaluating the integral of $(A_R')^2$.) Particularly, we will discard the term $(2-4\alpha) \Xi_R^2$ in subsequent discussion. For the other terms, observe that, on the one hand,
\begin{align*}
\int_{-R}^R   \frac{ q_R^{1/2} [\alpha + (1-2\alpha) \Xi_R^2]}{(1 - \Xi_R^2)^2}   A_R^2\,dx
	&= \int_{-R}^R     q_R^{1/2} [\alpha + (1-2\alpha) \Xi_R^2]  (1 - \Xi_R^2)^{2\alpha - 2} \,dx\nonumber\\
	&\leq \int_{-R}^R   \Xi_R' [\alpha + (1-2\alpha) \Xi_R^2]  (1 - \Xi_R^2)^{2\alpha - 2} \,dx\nonumber\\
	&= \int_{-1}^1 [\alpha + (1-2\alpha) \tau^2]  (1 - \tau^2)^{2\alpha - 2} \,d\tau = c_1(\alpha) < \infty.
\end{align*}
On the other hand, provided $R$ is sufficiently large such that $\Xi_R' = (q_R  + \frac{1}{2}(1 - \Xi_R^2)^2)^{1/2} \leq 1$ (recall that $q_R \searrow 0$ as $R \nearrow \infty$),
\[
\int_{-R}^R A_R^2\,dx \geq \int_{-R}^R (1 - \Xi_R^2)^{2\alpha}\Xi_R' \,dx = \int_{-1}^1 (1 - \tau^2)^{2\alpha}\,d\tau = c_2(\alpha) > 0.
\]
Therefore, by fixing $\alpha \in (\frac{1}{2},1 - \frac{(j'_{0,1})^2}{d^2})$ (which is possible since $d > \sqrt{2}j'_{0,1} $), we have from \eqref{Eq:MPR-2} that
\[
\int_{-R}^R \Big[(A_R')^2 + \frac{(j'_{0,1})^2}{d^2} A_R^2 - (1 - \Xi_R^2) A_R^2  \Big]\,dx
	\leq 2q_R^{1/2}c_1(\alpha) - \Big(\frac{(j'_{0,1})^2}{d^2} +   \alpha - 1\Big) c_2(\alpha)
\]
and so
\[
\lim_{R \rightarrow \infty} \int_{-R}^R \Big[(A_R')^2 + \frac{(j'_{0,1})^2}{d^2} A_R^2 - (1 - \Xi_R^2) A_R^2  \Big]\,dx \leq - \Big(\frac{(j'_{0,1})^2}{d^2} +   \alpha - 1\Big) c_2(\alpha) < 0.
\]
This proves that \eqref{Eq:MPR-1} holds for large $R$ and concludes the proof.
\end{proof}

\begin{lemma}\label{Lem:MPRChoice}
Suppose $d > \sqrt{2}j'_{0,1}$. There exists $\underline{R}_d > 0$ such that, for each $R \in (\underline{R}_d,\infty)$, there are two critical points $\hat u^{\pm }_{R,d}  = (\pm \hat U_{R,d}, \hat V_{R,d})$ of $\mcF_{R,d}$ in $\mcA_{R,d}^{\rm rad}$ with energy $c_{R,d}^{
\rm rad}$ satisfying 
\[
\begin{cases}
\hat U_{R,d}^2 + \hat V_{R,d}^2 \leq 1 \text{ and } \mathrm{sign}(x) \hat V_{R,d} \geq 0 \text{ in }\Omega_{R,d},\\
\hat U_{R,d} \text{ changes sign } on \{x = 0\} \cap \Omega_{R,d}.
\end{cases}
\]
In particular $\hat u^{\pm}_{R,d}$ has a zero on $\{x = 0\} \cap \Omega_{R,d}$ but is not identically zero therein, and it is different from the minimizers $u^\pm_{R,d}$ and the soliton $u_{s,R}$.
\end{lemma}

\begin{proof}
We claim that, for every given $v \in \mcA_{R,d}^{\rm rad}$, there exists a unique $j(v) \in \mcA_{R,d}^{\rm rad}$ such that
\[
\mcF_{R,d}[j(v)] \leq \mcF_{R,d}[\tilde v] \text{ for all } \tilde v \in \mcA_{R,d}^{\rm x} \text{ with } \tilde v|_{\{x = 0\}} = v|_{\{x = 0\}}.
\] 
Indeed, if we let $\Omega_{R,d}^+ = \{x > 0\} \cap \Omega_{R,d}$, then a candidate $j(v)$ is obtained by first minimizing
\[
\int_{\Omega_{R,d}^+} \Big[\frac{1}{2} |\nabla \tilde v|^2 + \frac{1}{4}(1 - |\tilde v|^2)^2\Big]\,dx
\]
among maps in $H^1(\Omega_{R,d}^+,\RR^2)$ which agree with $v$ on $\{x = 0\}$ and $\{x = R\}$, and then reflecting evenly reflecting the first component of $j(v)$ and oddly reflecting the second component of $j(v)$ across $\{x = 0\}$. To prove the claim, we only need to show that such a minimizer is unique. Replacing $j_2(v)$ by $|j_2(v)|$ in $\Omega_{R,d}^+$, we may assume $j_2(v) \geq 0$ in $\Omega_{R,d}^+$. Since $-\Delta j_2(v) = (1 - |j(v)|^2)j_2(v)$ in $\Omega_{R,d}^+$ and $j_2(v) = 1$ on $\{x = R\}$, the strong maximum principle implies that $j_2(v) > 0$ in $\Omega_{R,d}^+$. Thus, the kernel of $-\Delta - (1 - |j(v)|^2)$ in the space of functions $H^1(\Omega_{R,d}^+)$ which vanish at $\{x = 0\} \cup \{x = R\}$ is trivial. The proof of Lemma \ref{Lem:Pos=>Min} can then be applied yielding that $j(v)$ is the unique minimizer.

We note that the above argument also shows that
\begin{itemize}
\item $j_2(v) > 0$ in $\Omega_{R,d}^+$; and
\item $\mcF_{R,d}[u_{s,R}] \leq \mcF_{R,d}[w]$ for all $w \in \mcA_{R,d}^{\rm rad}$ with $w|_{\{x = 0\}} = 0$.
\end{itemize}

Let $\{h_{(m)}\} \subset \Gamma_{R,d}$ be such that $\sup_{t \in [-1,1]} \mcF_{R,d}[h_{(m)}(t)] \rightarrow c_{R,d}^{\rm rad}$ as $m \rightarrow \infty$. Using a standard truncation argument we may assume that $|h_{(m)}(t)| \leq 1$ in $\Omega_{R,d}$. Note that, by the uniqueness and minimizing properties of $j$, we have that $j(h_m) \in \Gamma_{R,d}$ and $|j(h_m(t)| \leq 1$ in $\Omega_{R,d}$. It follows that $\sup_{t \in [-1,1]} \mcF_{R,d}[j(h_{(m)}(t))] \rightarrow c_{R,d}$ as $m \rightarrow \infty$. An application of Ekeland's principle (see e.g. \cite[Corollary 4.3]{MawhinWillem}) then gives a sequence $\{u_m\} \subset \mcA_{R,d}^{\rm rad}$ and a sequence $\{t_m\} \subset [-1,1]$ such that
\[
\begin{cases}
\mcF_{R,d}[u_m] \rightarrow c_{R,d},\\
- \Delta u_m - (1 - |u_m|^2)u_m \rightarrow 0 \text{ in } H^{-1}(\Omega_{R,d},\RR^2),\\
u_m - j(h_m(t_m)) \rightarrow 0 \text{ in } H^1(\Omega_{R,d},\RR^2).
\end{cases}
\]
A routine argument then shows that, after possibly passing to a subsequence, $u_m$ as well as $j(h_m(t_m))$ converge in $H^1(\Omega_{R,d},\RR^2)$ to a critical point $\hat u_{R,d} = \hat u_{R,d}^+ = (\hat U_{R,d},\hat V_{R,d})$ of $\mcF_{R,d}$ in $\mcA_{R,d}^{\rm rad}$ and $\mcF_{R,d}[\hat u_{R,d}] = c_{R,d}^{\rm rad}$. Since $|j(h_m(t_m)| \leq 1$ and $j_2(h_m(t_m)) < 0$ in $\Omega_{R,d}^+$, we have that $\hat U^2_{R,d} + \hat V_{R,d}^2 \leq 1$ and $\textrm{sign}(x) \hat V_{R,d} \geq 0$ in $\Omega_{R,d}$. Since $j^2 = j$, we also have that 
\begin{equation}
j(\hat u_{R,d}) = j(\lim_{m\rightarrow \infty} j(h_m(t_m))) = \lim_{m\rightarrow \infty} j^2(h_m(t_m)) = \lim_{m\rightarrow \infty} j(h_m(t_m)) = \hat u_{R,d}.
	\label{Eq:jMPR}
\end{equation}
It is clear that $\hat u_{R,d}^- := (-\hat U_{R,d},\hat V_{R,d})$ is also a critical point of $\mcF_{R,d}$ in $\mcA_{R,d}^{\rm rad}$. 

So far, everything works for any positive finite $R$. We now impose that $R$ is sufficiently large so that $c_{R,d}^{\rm rad} < \mcF_{R,d}[u_{s,R}]$ (by Lemma \ref{Lem:MPR}) and $c_{R,d}^{\rm rad} > \underline{c}_{R,d}$ (by \eqref{Eq:RMinEnergy} and \eqref{Eq:cRdLB}). Then $\hat u_{R,d}^\pm$ are different from the minimizers $\underline{u}_{R,d}$ and the soliton $u_{s,R}$. To conclude, it remains to show that $\hat U_{R,d}$ changes sign on $\{x = 0\}$. Arguing by contradiction, assume that $\hat U_{R,d}$ does not change sign on $\{x = 0\}$. Without loss of generality, we may assume $\hat U_{R,d} \geq 0$ on $\{x = 0\}$. Using the fact that $j(\hat u_{R,d}) = \hat u_{R,d}$, the uniqueness and minimizing properties of $j(\hat u_{R,d})$ and the fact that $\hat u_{R,d}$ and $(|\hat U_{R,d}|, \hat V_{R,d})$ have the same energy relative to $\Omega_{R,d}^+$, we have that $\hat U_{R,d} \geq 0$ in $\Omega_{R,d}^+$ and hence in $\Omega_{R,d}$. Since $-\Delta \hat U_{R,d} = (1 - |\hat u_{R,d}|^2)\hat U_{R,d}$ in $\Omega_{R,d}$, the strong maximum principle implies that either $\hat U_{R,d} > 0$ or $\hat U_{R,d} \equiv 0$ in $\Omega_{R,d}$. The former case is not possible as Lemma \ref{Lem:Pos=>Min} would then imply $\hat u_{R,d}$ is a minimizer of $\mcF_{R,d}$ in $\mcA_{R,d}^{\rm x}$. The latter case is also not possible as the minimizing property of $\hat u_{R,d}$ and $u_{s,R}$ would imply that $\hat u_{R,d} = u_{s,R}$. The proof is complete.
\end{proof}

\begin{lemma}\label{Lem:MPLimit}
Suppose $d > \sqrt{2}j'_{0,1}$. Let $\hat u_{R,d}^\pm = (\pm \hat U_{R,d},\hat V_{R,d})$ be as in Lemma \ref{Lem:MPRChoice}. Then, along a sequence $R \rightarrow \infty$, $\hat u_{R,d}^\pm $ converge in $C_{\rm loc}^\infty(\bar \Omega_d)$ to $\hat u_d^\pm = (\pm \hat U_d,\hat V_d)$ which satisfies \eqref{Eq:CritNoBC}, the phase imprinting condition \eqref{Eq:Imprinting}, the symmetry condition \eqref{Eq:RadSym} and the property \eqref{Eq:hatUVSign}. Moreover, $\hat U_d$ changes sign on $\{x = 0\} \cap \Omega_d$ and hence $\hat u_d^\pm$ are different from the soliton solution $u_s$. Finally, as $x \rightarrow \pm\infty$, $\hat u_d^\pm \rightarrow (\pm a_d, b_d) \in \mathbb{S}^1$ with $b_d \in (0,1]$.
\end{lemma}

To dispel confusion, we are not able to assert at this point that $u_d^\pm \in \mcA_d^{\rm rad}$ (which is equivalent to $(a_d, b_d) = (0,1)$).

\begin{proof}
By elliptic estimates, along a sequence $R \rightarrow \infty$, $(\hat U_{R,d}, \hat V_{R,d})$ converges in $C^\infty_{\rm loc}(\bar \Omega_d)$ to a smooth map $\hat u_d^+ = \hat u_d = (\hat U_d,\hat V_d) \in C^\infty(\bar\Omega_d)$ which satisfies 
\begin{equation}
\begin{cases}
-\Delta \hat u_d = (1 - |\hat u_d|^2)\hat u_d \text{ in } \Omega_d,\\
\partial_\nu \hat u_d = 0 \text{ on } \partial \Omega_d,\\
\hat U_d^2 + \hat V_d^2 \leq 1 \text{ and } \mathrm{sign}(x) \hat V_d \geq 0 \text{ in } \Omega_d,\\
\hat U_d \text{ has a zero on } \{x = 0\} \cap \bar \Omega_d,\\
\mcF_d[\hat u_d] \leq \liminf_{R \rightarrow \infty} c_{R,d}^{\rm rad} \leq \liminf_{R \rightarrow \infty} \mcF_{R,d}[u_{s,R}] = \mcF_d[u_s].
\end{cases}
	\label{Eq:udTC}
\end{equation}
Moreover, referring to \eqref{Eq:jMPR} in the proof of Lemma \ref{Lem:MPRChoice}, we have that
\[
\mcF_d[\hat u_d] \leq \mcF_d[\hat u_d + w] \text{ for all } w \in H_0^1(\{x \geq 0\} \cap \Omega_d,\RR^2)
\]
where equality holds if and only if $w \equiv 0$. In addition, by elliptic estimates and finiteness of $\mcF_d[\hat u_d]$, we have that $|\hat u_d| \rightarrow 1$ and $|\nabla \hat u_d| \rightarrow 0$ as $|x| \rightarrow \infty$. In addition, using the third line of \eqref{Eq:udTC}, the strong maximum principle and the Hopf lemma, it is routine to show that the first line of \eqref{Eq:hatUVSign} holds, and either the second line of \eqref{Eq:hatUVSign} holds or $\hat V_d \equiv 0$. It remains to show that $\hat u_d$ has the stated limiting behavior as $x \rightarrow \pm \infty$ (which rules out the case $\hat V_d \equiv 0$) and $\hat U_d$ changes sign on $\{x = 0\} \cap \Omega_d$.

\medskip
\noindent\underline{Step 1:} We show that there exists $(a_d,b_d) \in \mathbb{S}^1$ with $b_d \geq 0$ such that $U_d \rightarrow a_d$ and $V_d \rightarrow \pm b_d$ as $x\rightarrow \pm \infty$.

We only consider the limit as $x \rightarrow \infty$. The other side follows from the phase imprinting condition \eqref{Eq:Imprinting}.

Let 
\[
\bar V(x) = \frac{1}{\pi d^2}\int_{\{|x'| < d\}} \hat V_d(x,x')\,dx'.
\]
Since $\hat V_d \geq 0$ for $x > 0$, $\hat U_d^2 + \hat V_d^2 \leq 1$ and $-\Delta \hat V_d = (1 - \hat U_d^2 - \hat V_d^2)\hat V_d$, we have that $\bar V$ is concave in $(0,\infty)$. Since $0 \leq \bar V \leq 1$, it follows that $\bar V(x) \nearrow b_d \in [0,1]$ as $x \nearrow \infty$. Since $\nabla \hat V_d \rightarrow 0$ as $x \rightarrow \infty$, we have that $\hat V_d \rightarrow b_d$ as $x \rightarrow \infty$. As $\hat U_d^2 + \hat V_d^2 \rightarrow 1$ and $\nabla \hat U_d \rightarrow 0$ as $x \rightarrow \infty$, it also follows that $\hat U_d \rightarrow a_d$ as $x \rightarrow \infty$ for some $a_d$ satisfying $a_d^2 + b_d^2 = 1$.

\medskip
\noindent\underline{Step 2:} We show that $b_d = 0$ cannot hold.

Suppose by contradiction that $b_d = 0$. Then $a_d \in \{\pm 1\}$. We will only consider the case $a_d = 1$, as the other case is similar.

Since $\bar V$ is nonpositive and convex in $(0,\infty)$, $\bar V(0) = 0$ and $\bar V \rightarrow b_d = 0$ as $x \rightarrow \infty$, we have that $\bar V \equiv 0$ in $(0,\infty)$. Since $\hat V_d\geq 0$, this implies that $\hat V_d \equiv 0$ in $\{x > 0\} \cap \Omega_d$ and hence on all of $\Omega_d$ by the phase imprinting condition \eqref{Eq:Imprinting}. In particular, we have
\begin{equation}
\begin{cases}
-\Delta \hat U_d = (1 - \hat U_d^2)\hat U_d \text{ in } \Omega_d,\\
-\partial_\nu \hat U_d = 0 \text{ on } \partial\Omega_d,\\
\hat U_d(x,\cdot) \rightarrow 1 \text{ as } x \rightarrow \infty.
\end{cases}
	\label{Eq:Sliding-0}
\end{equation}

We claim that $\hat U_d$ changes sign in $\Omega_d$. Arguing indirectly, assume that $\hat U_d$ does not change sign in $\Omega_d$. Then since $\hat U_d \rightarrow 1$ as $x \rightarrow \pm \infty$, we have that $\hat U_d \geq 0$ in $\Omega_d$. By \eqref{Eq:Sliding-0}, the strong maximum principle implies that $\hat U_d > 0$ in $\Omega_d$. The proof of Lemma \ref{Lem:Pos=>Min} can now be applied to show that $u_d$ minimizes $\mcF_d$ in $(1,0) + H^1(\Omega_d)$ and hence $\hat u_d \equiv (1,0)$. This contradicts the fact that $\hat U_d$ has a zero on $\{x = 0\} \cap \bar \Omega_d$. The claim is proved.

By the claim, the fact that $\hat U_d \rightarrow 1$ as $x \rightarrow \infty$, and by the phase imprinting condition \eqref{Eq:Imprinting}, there exists $x_0 \geq 0$ such that
\[
\hat U_d > 0 \text{ in } \{x > x_0\} \text{ and } \min U_d(x_0,\cdot) = 0.
\]
We proceed to use the sliding method (\cite{BN91}) to show that
\begin{equation}
\hat U_d(x,y) = \tanh \frac{x - x_0}{\sqrt{2}} \text{ in } \Omega_d,
	\label{Eq:Sliding-C}
\end{equation}
which would contradict the fact that $\hat U_d \rightarrow 1$ as $x \rightarrow -\infty$ and conclude this step. 

Fix some small $\delta > 0$ for the moment. Differentiating the first equation in \eqref{Eq:Sliding-0}, we have 
\[
\begin{cases}
-\Delta (\partial_x \hat U_d) = (1 - 3\hat U_d^2) \partial_x \hat U_d \text{ in } \Omega_d,\\
\partial_\nu (\partial_x \hat U_d) = 0 \text{ on } \partial \Omega_d.
\end{cases}
\]
Fix some $x_1 > x_0$ such that $\hat U_d^2 \geq 1 - \delta$ in $\{x \geq x_1\}$. Let
\[
B_\mu(x) = \Big(\sup_{x = x_1} |\partial_x \hat U_d|\Big) e^{-\sqrt{2 - 3\delta}(x - x_1)} + \mu e^{\sqrt{2 - 3\delta}x}, \quad \mu > 0.
\]
Note that $B_\mu \geq 0$ and satisfies
\[
\begin{cases}
-\Delta B_\mu = -(2 - 3\delta) B_\mu \geq (1 - 3U_d^2) B_\mu \text{ in } \{x > x_1\},\\
B_\mu \geq |\partial_x \hat U_d| \text{ on } \{x = x_1\},\\
\partial_\nu B_\mu = 0 \text{ on } \partial \Omega_d,\\
B_\mu \rightarrow \infty \text{ as } x \rightarrow \infty.
\end{cases}
\]
A standard argument using the strong maximum principle and the Hopf lemma then shows that $|\partial_x \hat U_d| \leq B_\mu$ in $\{x > x_1\}$ for any $\mu > 0$. Sending $\mu \rightarrow 0$, we obtain
\begin{equation}
|\partial_x \hat U_d| \leq \Big(\sup_{x = x_1} |\partial_x \hat U_d|\Big) e^{-\sqrt{2 - 3\delta}(x - x_1)} \text{ in } x > x_1.
	\label{Eq:pxUdDecay}
\end{equation}

For $x_2 \geq x_0$, consider the function
\[
T_{\delta,x_2}(x) = \tanh \frac{\sqrt{2 - 4\delta}(x - x_2)}{2}.
\]
By \eqref{Eq:pxUdDecay}, the fact that $\hat U_d \rightarrow 1$ as $x \rightarrow \infty$ and the positivity of $\hat U_d$ in $\{x > x_0\}$, we have that 
\begin{itemize} 
\item $\hat U_d \geq T_{\delta,x_2}$ in $\{x > x_0\}$ for all large $x_2$, and
\item for any $x_2 \geq x_0$, $\hat U_d > T_{\delta,x_2}$ for all large $x$.
\end{itemize}
Let
\[
x_3 = \inf\Big\{x_2 \geq x_0: \hat U_d \geq T_{\delta,x_2} \text{ in }\{x > x_0\}\Big\}.
\]
We have
\[
\begin{cases}
-\Delta T_{\delta,x_3} = (1-2\delta)(1 - T_{\delta,x_3}^2)T_{\delta,x_3} \leq (1 - T_{\delta,x_3}^2)T_{\delta,x_3} \text{ in } \{x > x_3\},\\
-\Delta \hat U_d = (1 - \hat U_d^2) \hat U_d \text{ in } \{x > x_3\},\\
\hat U_d \geq T_{\delta,x_3} \text{ in } \{x > x_3\},\\
\hat U_d > T_{\delta,x_3} \text{ for all large } x,\\
\end{cases}
\]
By the strong maximum principle, we then have that $\hat U_d > T_{\delta,x_3}$ in $\{x > x_3\}$. The minimality of $x_3$ then implies that $\hat U_d = T_{\delta,x_3} = 0$ somewhere on $\{x = x_3\}$. Since $\hat U_d > 0$ in $\{x > x_0\}$, we deduce that $x_3 = x_0$. We deduce that $\hat U_d > T_{\delta,x_0}$ in $\{x > x_0\}$. Sending $\delta \rightarrow 0$, we obtain
\begin{equation}
\hat U_d \geq T_{0,x_0} = \tanh \frac{x - x_0}{\sqrt{2}} \text{ in } \{x \geq x_0\}.
	\label{Eq:Sliding-1}
\end{equation}

Next, note that, in view of \eqref{Eq:Sliding-0} and the fact that $\hat U_d \geq -1$, we have by the strong maximum principle and the Hopf lemma that 
\[
\inf_{\Omega_d} \hat U_d > -1.
\]
This together with \eqref{Eq:Sliding-1} implies that $\hat U_d \geq T_{0,x_2}$ in $\Omega_d$ for all large $x_2$. Let
\[
x_4 = \inf\Big\{x_2 \geq x_0: \hat U_d \geq T_{0,x_2} \text{ in }\Omega_d\Big\}.
\]
The minimality of $x_4$ together with \eqref{Eq:Sliding-1} and the fact that $\hat U_d \rightarrow 1$ and $T_{0,x_4} \rightarrow - 1$ as $x \rightarrow \infty$ implies that there exists $\hat U_d = T_{0,x_4}$ somewhere in $\bar\Omega_d$. Recalling \eqref{Eq:Sliding-0} and the fact that $-\Delta T_{0,x_4} = (1 - T_{0,x_4}^2)T_{0,x_4}$, we deduce from the strong maximum principle and the Hopf lemma that $\hat U_d \equiv T_{0,x_4}$. Since $\hat U_d$ vanishes somewhere on $\{x = x_2\}$, this also implies $x_4 = x_2$ and proves \eqref{Eq:Sliding-C}. As pointed out earlier, this give a contradiction to the behavior of $\hat U_d$ as $x \rightarrow -\infty$ and so concludes Step 2.

\medskip
\noindent\underline{Step 3:} We show that $\hat U_d$ changes sign on $\{x = 0\} \cap \Omega_d$.

Suppose by contradiction that $\hat U_d$ does not change sign on $\{x = 0\} \cap \Omega_d$, e.g. $\hat U_d \geq 0$ on $\{x = 0\} \cap \Omega_d$. By the minimizing property of $\hat u_{R,d}$, we have that $(\hat U_d,\hat V_d)$ is minimizing with respect to smooth perturbations compactly supported in $\{ x > 0 \} \cap \Omega_d$. It follows that $\hat U_d \geq 0$ in $\{x \geq 0\} \cap \Omega_d$, and hence on all of $\Omega_d$ by the phase imprinting condition \eqref{Eq:Imprinting}. We may then follow the proof of Lemma \ref{Lem:Pos=>Min} to see that $(\hat U_d,\hat V_d)$ is minimizing with respect to smooth perturbations compactly supported in $\Omega_d$. It follows that
\[
\mcF_d[(\hat U_d,\hat V_d) \leq \mcF_d[u] \text{ for all } u \in (a_d ,b_d\tanh \frac{x}{\sqrt{2}})  + H^1(\Omega_d,\RR^2) =: \mcA_*.
\]
On the other hand, if we write $(a_d,b_d) = (\cos \chi_*,\sin \chi_*)$ for some $\chi_* \in [-\pi/2,0)$ and if $\chi_0 \in C^\infty(\RR)$ satisfying $\chi_0 \equiv \chi_*$ in $[1,\infty)$, then the map
\[
\zeta_n(x,x') = ( \cos \chi_0(x/n), \sin \chi_0(x/n))
\]
belongs to $\mcA_*$ and has $\mcF_d[\zeta_n] \rightarrow 0$ as $n \rightarrow \infty$. It follows that $\mcF_d[(\hat U_d,\hat V_d)] = 0$, which is impossible since $\hat V_d$ is not a constant. We conclude that $\hat U_d$ changes sign on $\{x = 0\} \cap \Omega_d$.
\end{proof}

\begin{proof}[Proof of Theorem \ref{Thm:RingCand}] This follows from Lemma \ref{Lem:MPLimit}.
\end{proof}



\section*{Rights retention statement.} 

For the purpose of Open Access, the authors have applied a CC BY public copyright licence to any Author Accepted Manuscript (AAM) version arising from this submission.

\section*{Acknowledgements.} The research of AA is supported by France 2030 PIA
funding: ANR-21-EXES-0003 and she was partially supported by the UKRC (Grant 2019-55900).

\bibliographystyle{siam}
\bibliography{LiquidCrystals,ref}

\end{document}